\documentclass[12pt]{article}

\usepackage{amscd,latexsym,amsthm,amsfonts,amssymb,amsmath,amsxtra}
\usepackage{fullpage}
\usepackage[mathscr]{eucal}
\usepackage{pictexwd,dcpic}
\usepackage[normalem]{ulem}
\usepackage{hyperref}
\usepackage{enumerate}

\newtheorem{thm}{Theorem}[section]
\newtheorem{cor}[thm]{Corollary}
\newtheorem{lem}[thm]{Lemma}

\newtheorem{prop}[thm]{Proposition}
\newtheorem {conj}[thm]{Conjecture}

\newtheorem {ques/conj}[thm]{Question/Conjecture}

\newtheorem{rmk}[thm]{Remark}

\DeclareMathOperator{\Tr}{Tr}
\DeclareMathOperator{\vol}{vol}

\DeclareMathOperator{\Irr}{Irr}
\DeclareMathOperator{\Norm}{Norm}

\newcommand{\BA}{{\mathbb {A}}}

\newcommand{\BC}{{\mathbb {C}}}

\newcommand{\CA}{{\mathcal {A}}}

\newcommand{\CC}{{\mathcal {C}}}

\newcommand{\CT}{{\mathcal {T}}}

\newcommand{\Fg}{{\mathfrak {g}}}

\newcommand{\GL}{{\mathrm{GL}}}

\newcommand{\Hom}{{\mathrm{Hom}}}

\newcommand{\Mat}{{\mathrm{Mat}}}

\newcommand{\SO}{{\mathrm{SO}}}

\newcommand{\ul}{\underline}

\begin{document}

\title{A Local Trace Formula for the Generalized Shalika Model}

\author{Rapha\"el Beuzart-Plessis,\; Chen Wan}

\date{\today}

\maketitle

\begin{abstract}
\noindent
We study local multiplicities associated to the so-called generalized Shalika models. By establishing a local trace formula for these kind of models, we are able to prove a multiplicity formula for discrete series. As a result, we can show that these multiplicities are constant over every discrete Vogan $L$-packet and that they are related to local exterior square $L$-functions.
\end{abstract}

\tableofcontents

\section{Introduction}

Let $G$ be a $p$-adic reductive group, $H$ a closed subgroup of $G$ and $\chi$ a character of $H$ (potentially the trivial one). To every smooth irreducible representation $\pi$ of $G$, we associate a multiplicity

$$\displaystyle m(\pi,\chi):=\dim \Hom_H(\pi,\chi)$$

\noindent If the subgroup $H$ is {\it spherical} (that is it admits an open orbit on the flag variety of $G$) then we expect these multiplicities to always be finite (this is already known in a certain number of cases see \cite{Del} Theorem 4.5 and \cite{SV} Theorem 5.1.5) and to roughly detect certain kind of functorial lifts. For a good references on this circle of ideas, that has come to be called the {\it relative local Langlands program}, we refer the reader to \cite{Pras} and to the monograph \cite{SV} which set forth a general formalism '\`a la Langlands' for these kind of problems.

In the foundational papers \cite{WalGGPI}, \cite{WalGGPII}, Waldspurger has discovered a new way to attack these questions by proving a certain integral formula computing the multiplicity $m(\pi,\chi)$ in the case of the so-called orthogonal Gross-Prasad models which, together with some twisted version of it related to epsilon factors of pair, has found a remarkable application to the local Gross-Prasad conjecture for orthogonal groups (see \cite{WalGGPIII}, \cite{MWGGP}). This line of attack has then been adapted by the first author \cite{BeuGGP0}, \cite{BeuGGP} to deal with the local Gross-Prasad conjecture for unitary groups and by the second author \cite{Wan15}, \cite{Wan16a} in the setting of the so-called Ginzburg-Rallis models. Subsequently, in \cite{BeuGalP} the first author has also find another application of this method to a conjecture of Prasad concerning Galois pairs. In all these cases the basic tool to prove the aforementioned multiplicity formulas has been some new kind of local (simple) trace formulas in the spirit of Arthur \cite{ArthurlocalTF}. However, the proofs of these trace formulas, and particularly of their geometric sides, has each time been done in some ad hoc way pertaining to the particular features of the case at hand. It makes now little doubt that such trace formulas should exist in some generality and we provide here another example for the 'generalized Shalika models' in the hope that it can shed some light on the general features of a potential generalization.

\subsection{Main results}
Let $F$ be a $p$-adic field and $\CA$ be a central simple algebra over $F$ of rank $n$ (i.e. $\CA=\Mat_m(\mathcal{D})$ where $\mathcal{D}/F$ is a division algebra of degree $r$ and $n=mr$). We will denote by $\Tr_{\mathcal{A}/F}:\mathcal{A}\to F$ and $N_{\mathcal{A}/F}:\mathcal{A}\to F^\times$ the reduced trace and norm respectively. Set $G:=GL_2(\mathcal{A})$ and define the following subgroups of $G$:

\begin{itemize}
\item $H_0:=\{\begin{pmatrix}
\lambda & \\ & \lambda \end{pmatrix}\mid \lambda\in \mathcal{A}^\times \}$;

\item $N:=\{ \begin{pmatrix}
 1 & X \\ & 1 \end{pmatrix}\mid X\in \mathcal{A}\}$;

\item $H:=H_0\ltimes N$.
\end{itemize}

Fix a continuous character $\omega:F^\times\to \mathbf{C}^\times$ that we identify with a character of $H_0$ through composition with $N_{\mathcal{A}/F}:H_0\to F^\times$. Let $\psi:F\to \BC^{\times}$ be a nontrivial character and define $\xi:N\to \mathbf{C}^\times$ by

$$\displaystyle \xi\begin{pmatrix} 1 & X \\ & 1 \end{pmatrix}:=\psi(\Tr_{\mathcal{A}/F} X),\;\;\; X\in \mathcal{A}.$$

\noindent Then $\xi$ is invariant under the $H_0$-conjugation and thus extends to a character, again denoted $\xi$, of $H$ trivial on $H_0$. Similarly, we consider $\omega$ as a character on $H$ by composition with the projection $H\twoheadrightarrow H_0$ and we denote by $\omega\otimes \xi$ the product of these two characters of $H$. We refer to the triple $(G,H,\omega\otimes \xi)$ as a \textit{generalized Shalika triple}. In particular, if $\CA=\Mat_n(F)$, we recover the usual Shalika model for $GL_{2n}$. For all irreducible admissible representation $\pi$ of $G$, we define the \textit{multiplicity} $m(\pi,\omega)$ to be

$$\displaystyle m(\pi, \omega):=\dim \Hom_{H}(\pi,\omega\otimes \xi).$$

\noindent By Theorem 4.5 of \cite{Del}, we know that this multiplicity is always finite. The goal of this paper is to study the behavior of the multiplicity $m(\pi,\omega)$ inside the discrete local Vogan $L$-packets i.e. under the Jacquet-Langlands correspondences.

\begin{rmk}
In fact, by \cite{JR} and \cite{CS}, we even know the multiplicity $m(\pi,\omega)$ is less or equal to 1 (i.e. the generalized Shalika models are Gelfand pairs). But we don't need this result in the proof of the main theorem.
\end{rmk}

Let $\mathcal{A}'$ be another degree $n$ central simple algebra over $F$. Set $G':=GL_2(\mathcal{A}')$ and define subgroups $H_0'$, $N'$, $H':=H_0'\ltimes N'$ analogous to the subgroups $H_0$, $N$ and $H$ of $G$. We also define similarly characters characters $\xi'$, $\omega\otimes \xi'$ of $N'$, $H'$ respectively and for all irreducible admissible representation $\pi'$ of $G'$, we set

$$\displaystyle m(\pi', \omega):=\dim\Hom_{H'}(\pi',\omega\otimes \xi').$$
The main result of this paper is the following theorem which says that these multiplicities are constant over every discrete Vogan $L$-packet.

\begin{thm}\label{main theorem 1}
Let $\pi$ (resp. $\pi'$) be a discrete series of $G$ (resp. $G'$). Assume that $\pi$ and $\pi'$ correspond to each other under the local Jacquet-Langlands correspondence (see \cite{DKV84}). Then
$$\displaystyle m(\pi,\omega)=m(\pi',\omega).$$
\end{thm}

Assume one moment that $\omega\mathbf{1}$ (the trivial character) and set for simplicity $m(\pi):=m(\pi,\mathbf{1})$. Then, by work of Kewat \cite{K11}, Kewat-Ragunathan \cite{KR}, Jiang-Nien-Qin \cite{JNY} and the multiplicity one theorem of Jacquet-Rallis \cite{JR}, in the particular case where $\mathcal{A}=M_n(F)$ we know that for all discrete series $\pi$ we have $m(\pi)=1$ if and only if $L(s,\pi,\wedge^2)$ (the Artin exterior square $L$-function) has a pole at $s=0$ (i.e. the Langlands parameter of $\pi$ is symplectic) and $m(\pi)=0$ otherwise. Actually, to our knowledge, a full proof of this result has not appeared in the literature and thus for completeness we provide the necessary complementary arguments in Section \ref{application 1}. Together with Theorem \ref{main theorem 1} this immediately implies

\begin{thm}\label{main theorem 2}
For all discrete series representation $\pi$ of $G$, we have $m(\pi)=1$ if and only if the local exterior square $L$-function $L(s,\pi,\wedge^2)$ has a pole at $s=0$ and $m(\pi)=0$ otherwise.
\end{thm}

We will prove Theorem \ref{main theorem 1} in Section \ref{application 1}. The key ingredient of our proof is a certain integral formula computing the multiplicity $m(\pi,\omega)$ that we now state. Recall that following Harish-Chandra, any irreducible representation $\pi$ has a well-defined character $\Theta_\pi$ which is a locally integrable function on $G$ locally constant on the regular semi-simple locus. Moreover, Harish-Chandra has completely described the possible singularities of $\Theta_\pi$ near singular semi-simple elements leading to certain local expansions of the character near such point. Using these, we can define a certain regularization $x\mapsto c_\pi(x)$ of $\Theta_\pi$ at all semi-simple point by taking the average of the 'leading coefficients' of these local expansions (see Section \ref{representations} for details, actually for the groups considered in this paper there is always at most one such leading coefficient). Given this, our multiplicity formula can be stated as follows (see Proposition \ref{prop multiplicities})

\begin{thm}\label{main theorem 3}
For all essentially square-integrable representation $\pi$ of $G$ with central character $\chi=\omega^n$ (seen as a character of $A_G=F^\times$), we have

$$\displaystyle m(\pi,\omega)=\sum_{T\in \mathcal{T}_{ell}(H_0)} \lvert W(H_0,T)\rvert^{-1} \int_{A_G\backslash T} D^{H}(t) c_\pi(t) \omega(t)^{-1}dt$$

\noindent where $\mathcal{T}_{ell}(H_0)$ stands for a set of representatives of elliptic maximal tori in $H_0$, $W(H_0,T)=Norm_{H_0}(T)/T$ is the corresponding Weyl group, $D^H(t)$ is the usual Weyl discriminant, the measure on the tori $A_G\backslash T$ are chosen to be of total mass one and the expression on the right hand side is absolutely convergent.
\end{thm}

Theorem \ref{main theorem 1} is then an easy consequence of Theorem \ref{main theorem 3} and the characters relations characterizing the local Jacquet-Langlands correspondences (see \S \ref{application 1} for details).

\begin{rmk}
In Appendix A, we will prove a slight generalization of a result of M\oe{}glin and Waldspurger a consequence of which is that the multiplicity formula above holds more generally for all irreducible admissible representations of $G$ when $\CA=\mathcal{D}$ is a division algebra. On the other hand, if $\CA$ is not a division algebra, the multiplicity formula will only hold for discrete series (see Remark \ref{multiplicity formula fails} for more details).
\end{rmk}

For its part Theorem \ref{main theorem 3} is a consequence of a certain local simple trace formula for the generalized Shalika models of the same kind as the local trace formulas developed in \cite{BeuGGP}, \cite{BeuGalP} and \cite{Wan15}. To be specific, let $f\in {}^{\circ}\CC(G)$ be an Harish-Chandra cusp form (see Section \ref{cusp forms and theta-strongly cuspidal functions} for the definition of these) and for all $x,y\in G$, set

$$\displaystyle K_f(x,y):=\int_H f(x^{-1}hy) (\omega\otimes \xi)(h)^{-1}dh.$$

\noindent We define a distribution $J$ on the space of cusp forms by

$$\displaystyle J(f):=\int_{H\backslash G}K_f(x,x) dx.$$
In later sections, we will show that both integrals above are absolutely convergent.

The aforementioned trace formula gives two expansions of $J(f)$: one geometric and one spectral. The geometric side is given by
$$\displaystyle J_{geom}(f)=\sum_{T\in \mathcal{T}_{ell}(H_0)} \lvert W(H_0,T)\rvert^{-1} \int_T D^{H}(t) c_f(t) \omega(t)^{-1}dt$$
where $\mathcal{T}_{ell}(H_0)$ denotes a set of representatives of conjugacy classes of maximal elliptic tori in $H_0$, $W(H_0,T)$ stands for the corresponding Weyl group, $D^H$ is the usual Weyl discriminant and $c_f(t)$ is a certain weighted orbital integral of $f$ in the sense of Arthur (see \S \ref{cusp forms and theta-strongly cuspidal functions} for a precise definition). The spectral side, on the other hand, is given by the following expression
$$\displaystyle J_{spec}(f)=\sum_{\pi\in \Pi_{2}(G,\chi)}  m(\pi,\omega) \Tr \pi^{\vee}(f)$$
where $\Pi_{2}(G,\chi)$ denotes the set of (isomorphism classes of) discrete series of $G$ with central character $\chi=\omega^n$ seen as a character of $A_G=F^\times$ and $\pi^{\vee}$ stands for the contragredient of $\pi$. Then the trace formula we proved in this paper is just (see Theorem \ref{theo trace formula})
\begin{thm}\label{trace formula intro}
For all $f\in {}^{\circ}\CC(G)$, we have
\begin{equation}
J_{spec}(f)=J(f)=J_{geom}(f).
\end{equation}
\end{thm}
More precisely, the spectral side of the trace formula will be proved in Section \ref{The spectral side} and the geometric side will be proved in Section \ref{geom side}. Moreover, Theorem \ref{main theorem 3} is, by standard means, an easy consequence of this trace formula (see \S \ref{A formula for the multiplicities}).

In Section \ref{application 2}, we will discuss another application of the multiplicity formula. By applying Theorem \ref{main theorem 3} together with another multiplicity formula for the so-called Ginzburg-Rallis model proved in the previous papers \cite{Wan15} and \cite{Wan16a} of the second author, we are able to establish some relationship between the two kind of multiplicities (cf. Theorem \ref{GR 1} and Theorem \ref{GR 2}). This will also allow us to prove the epsilon dichotomy conjecture for the Ginzburg-Rallis model in some cases. We refer the readers to Section \ref{application 2} for details.

Finally, in Section \ref{r trace formula}, guided by the idea of beyond endoscopy, together with Theorem \ref{main theorem 2} relating the multiplicities for generalized Shalika models to poles of local exterior square $L$-function, we restate our trace formula in the form of a (local) '$r$-trace formula' for $r=\wedge^2$ the exterior square representation of the $L$-group ${}^LG=\GL_{2n}(\BC)$.

\subsection{Organization of the paper and remarks on the proofs}
In Section \ref{preliminaries}, we introduce basic notations and conventions of this paper. This include some extended discussions of ($\theta$-)weighted orbital integrals, germ expansions and the Harish-Chandra-Schwartz space. In Section \ref{A multiplicity formula for Shalika models}, we state our (simple) local trace formula (Theorem \ref{trace formula intro}) and prove that the multiplicity formula (Theorem \ref{main theorem 3}) is a consequence of it.

Sections \ref{The spectral side} and \ref{geom side} are devoted to the proof of the trace formula. More precisely, in Section \ref{The spectral side} we prove the spectral side of the trace formula. It is the easy part and moreover the arguments are very similar to \cite{BeuGalP} \S 3. Section \ref{geom side} contains the proof of the geometric side which is more involved. The general idea is inspired by the work of Waldspurger (\cite{WalGGPI}, \cite{WalGGPII}) and the first author (\cite{BeuGGP}, \cite{BeuGalP}) on the Gan-Gross-Prasad and Galois models. However, due to significant differences between generalized Shalika models and the previous cases, our proof of the geometric side is quite different. Indeed, as in the Gan-Gross-Prasad cases, singular orbits are contributing to the geometric side and these contributions are reflected in singularities of the original expression. Due to the fact that the generalized Shalika models are usually not {\it strongly tempered} in the sense of \cite{SV}, we were unable to linearize the problem in order to perform a Fourier transform as in \cite{WalGGPI},\cite{BeuGGP} where it had the effect of killing the problematic singularities. As a result, we have to deal with them directly and for that we have in particular computed explicitly certain singular weighted (or rather {\it $\theta$-weighted}) orbital integrals (see \S \ref{first application}).

Sections \ref{section Applications} and \ref{r trace formula} contain applications of the trace formula and multiplicity formula. In section \ref{application 1}, we prove the two main theorems (Theorem \ref{main theorem 1} and Theorem \ref{main theorem 2}) of this paper and in section \ref{application 2}, we study the relations between the multiplicities for the generalized Shalika model and the Ginzburg-Rallis model. Using this, we will prove new cases of the epsilon dichotomy conjecture for the Ginzburg-Rallis model. Finally, in section \ref{r trace formula}, we rewrite our local trace formula as some kind of 'local $r$-trace formula'.

Finally, Appendix \ref{appendix} contains a slight generalization of a result of M\oe{}glin and Waldspurger concerning (generalized) Whittaker models.

\subsection{Acknowledgments}
R.B.P. has benefited from a grant of Agence Nationale de la Recherche with reference ANR-13-BS01-0012 FERPLAY. Collaboration on this work started at a workshop in the Mathematisches Forschungsinstitut in Oberwolfach on harmonic analysis and the trace formula. The authors thank this institution for its hospitality and the organizers for the invitations.

\section{Preliminaries}\label{preliminaries}

\subsection{Groups, measures, notations}\label{groups, measures, notations}

Throughout this paper $F$ will denote a $p$-adic field (i.e. a finite extension of $\mathbf{Q}_p$ for a certain prime number $p$) with ring of integer $\mathcal{O}_F$ and normalized absolute value $\lvert .\rvert_F$. We will denote by $v_F$ the normalized valuation on $F$, by $q$ the cardinal of the residue field of $F$ and by $\log$ the logarithm in base $q$ (so that $v_F(\lambda)=-\log \lvert \lambda\rvert_F$ for all $\lambda\in F^\times$). Moreover, for all finite extension $K$ of $F$ we will set $v_K(\lambda):=v_F(N_{K/F}(\lambda))$ for all $\lambda\in K$ where $N_{K/F}:K\to F$ stands for the norm. We fix throughout a nontrivial additive character $\psi:F\to \mathbf{C}^\times$. We will slightly abuse notations and denote algebraic groups and Lie algebras defined over $F$ and their sets of $F$-points by the same letters.

\vspace{2mm}

Let $G$ be a connected reductive group over $F$. We will denote by $A_G$ its maximal central split torus and set

$$\displaystyle \mathcal{A}_G:=X_*(A_G)\otimes \mathbf{R}$$

\noindent whose dual naturally identifies to

$$\displaystyle \mathcal{A}^*_G:=X^*(A_G)\otimes \mathbf{R}$$

\noindent where $X_*(A_G)$ and $X^*(A_G)$ stand for the groups of cocharacters and characters of $A_G$ respectively. There is a natural morphism $H_G:G\to \mathcal{A}_G$ characterized by

$$\displaystyle \langle \chi,H_G(g)\rangle=\log(\lvert \chi(g)\rvert)$$

\noindent for all $\chi\in X^*(G)$. We set $\mathcal{A}_{G,F}:=H_G(A_G)$. It is a lattice in $\mathcal{A}_G$. The same notations will be used for the Levi subgroups of $G$ (i.e. the Levi components of parabolic subgroups of $G$): if $M$ is a Levi subgroup of $G$, we define similarly $A_M$, $\mathcal{A}_M$, $H_M$ and $\mathcal{A}_{M,F}$. For such a Levi $M$, we will set

$$\displaystyle \mathcal{A}_M^G:=\mathcal{A}_M/\mathcal{A}_G,\;\; \mathcal{A}_{M,F}^G:=\mathcal{A}_{M,F}/\mathcal{A}_{G,F}.$$

\noindent We will also use Arthur's notations: $\mathcal{P}(M)$, $\mathcal{F}(M)$ and $\mathcal{L}(M)$ will stand for the sets of parabolic subgroups with Levi component $M$, parabolic subgroups containing $M$ and Levi subgroups containing $M$ respectively. Let $K$ be a special maximal compact subgroup of $G$. Then, for all parabolic subgroup $P$ with Levi decomposition $P=MU$, the Iwasawa decomposition $G=MUK$ allows us to extend $H_M$ to a map $H_P:G\to \mathcal{A}_M$ defined by $H_P(muk):=H_M(m)$ for all $m\in M$, $u\in U$ and $k\in K$. The Lie algebra of $G$ will be denoted by $\mathfrak{g}$ and more generally for any algebraic group we will denote its Lie algebra by the corresponding Gothic letter. We will write $Ad$ for the adjoint action of $G$ on $\mathfrak{g}$. We denote by $\exp$ the exponential map which is an $F$-analytic map from an open neighborhood of $0$ in $\mathfrak{g}$ to $G$. We define $G_{reg}$ as the open subset of regular semisimple elements of $G$. The notation $\mathcal{T}(G)$ (resp. $\mathcal{T}_{ell}(G)$) will be used to denote a set of representatives for the $G$-conjugacy classes of maximal tori (resp. elliptic maximal tori) in $G$.

Let $H$ be an algebraic group over $F$. For any subset $S\subset H$, we write $Cent_H(S)$ (resp. $Norm_{H}(S)$) for the centralizer of $S$ in $H$ (resp. the normalizer of $S$ in $H$). If $S=\{x\}$ we will write $H_x$ for the neutral connected component of $Cent_H(x):=Cent_H(\{x\})$. The Weyl discriminant $D^H$ is defined by

$$\displaystyle D^H(x):=\left\lvert det(1-Ad(x)_{\mid \mathfrak{h}/\mathfrak{h}_x})\right\rvert$$

\noindent for all semisimple element $x\in H$. For every subtorus $T$ of $H$, we will denote by

$$\displaystyle W(H,T):=Norm_{H}(T)/Cent_{H}(T)$$

\noindent the corresponding Weyl group. If $A\subset H$ is a split subtorus which normalizes a unipotent subgroup $U\subset H$ we will write $R(A,U)$ for the set of roots of $A$ in $\mathfrak{u}$.

\vspace{2mm}

If $T$ is a torus over $F$, we will denote by $T^c$ its maximal compact subgroup.

\vspace{2mm}

In this paper, we will assume that all the groups that we encounter have been equipped with Haar measures (left and right invariants as we will only consider measures on unimodular groups). In the particular case of tori $T$ we normalize these Haar measures as follows: we fix on $A_T$ the unique Haar measure giving $A_T^c$ volume $1$ and we choose on $T$ the unique Haar measure such that $\vol(T/A_T)=1$. For any connected reductive group $G$, we equip $\mathcal{A}_G$ with the unique Haar measure such that $\vol(\mathcal{A}_G/\mathcal{A}_{G,F})=1$. Thus this requirement also fixes Haar measures on $\mathcal{A}_M$ for all Levi subgroup $M$ of $G$. If $M\subset L$ are two Levi subgroups then we give $\mathcal{A}_M^L\simeq \mathcal{A}_M/\mathcal{A}_L$ the quotient measure.

\vspace{2mm}

We will adopt the following slightly imprecise but convenient notations. If $f$ and $g$ are positive functions on a set $X$, we will write

\begin{center}
$f(x)\ll g(x)$ for all $x\in X$
\end{center}

\noindent and we will say that $f$ is essentially bounded by $g$, if there exists a $c>0$ such that

\begin{center}
$f(x)\leqslant cg(x)$ for all $x\in X.$
\end{center}

\noindent We will also say that $f$ and $g$ are equivalent and we will write

\begin{center}
$f(x)\sim g(x)$ for all $x\in X$
\end{center}

\noindent if both $f$ is essentially bounded by $g$ and $g$ is essentially bounded by $f$.

\vspace{2mm}

In this paper we will freely use the notion of log-norms on varieties over $F$. The concept of norm on varieties over local fields has been introduced by Kottwitz in \cite{KottHA} \S 18. A log-norm is essentially just the log of a Kottwitz's norm and we refer the readers to \cite{BeuGGP} \S 1.2 for the definition and the basic properties of these log-norms. We will assume that all the algebraic varieties $X$ over $F$ that we encounter have been equipped with log norms $\sigma_X$. And for all $C>0$, we will denote by $\mathbf{1}_{X,\leqslant C}$ (resp. $\mathbf{1}_{X,>C}$) the characteristic function of $\{x\in X; \sigma_X(x)\leqslant C \}$ (resp. $\{x\in X; \sigma_X(x)> C \}$).

\vspace{2mm}

For any connected reductive group $G$ over $F$, we will denote by $\Xi^G$ the Xi function of Harish-Chandra on $G$ (see \cite{BeuGGP} \S 1.5 for the definition and basic properties of this function) and we will denote by $\mathcal{C}(G)$ the Harish-Chandra Schwartz space of $G$. This space consists of functions $f:G\to \mathbf{C}$ which are biinvariant by a certain compact-open subgroup $J\subset G$ and such that for all $d>0$, we have an inequality

$$\displaystyle \lvert f(g)\rvert\ll \Xi^G(g)\sigma_G(g)^{-d}$$

\noindent for all $g\in G$. Let $\chi$ be a unitary character of $A_G$, then we will denote by $\mathcal{C}(G,\chi)$ the space of functions $f:G\to \mathbf{C}$ which are biinvariant by a certain compact-open subgroup $J\subset G$ such that $f(ag)=\chi(a)f(g)$ for all $a\in A_G$ and $g\in G$, and such that for all $d>0$, we have an inequality

$$\displaystyle \lvert f(g)\rvert\ll \Xi^G(g)\sigma_{A_G\backslash G}(g)^{-d}$$

\noindent for all $g\in G$. There is a natural surjective map

$$\mathcal{C}(G)\to \mathcal{C}(G,\chi):\;f\mapsto f_\chi$$

\noindent given by

$$\displaystyle f_\chi(g):=\int_{A_G} f(ag)\chi(a)^{-1}da,\;\;\; f\in \mathcal{C}(G), g\in G.$$

\vspace{2mm}

For any set $S$ we will denote by $\mathbf{1}_S$ its characteristic function.

\subsection{Representations}\label{representations}

Let $G$ be a connected reductive group over $F$. We will write $\Irr(G)$ for the set of isomorphism classes of (complex-valued) irreducible smooth representations of $G$. We will identify any element of $\Irr(G)$ with one of its representative. For $\pi\in \Irr(G)$, we will also write $\pi$ for the space on which $\pi$ acts. We will denote by $\Pi_2(G)\subset \Irr(G)$ the subset of essentially square-integrable representations. And if $\chi$ is a character of $A_G$, we will denote by $\Pi_2(G,\chi)\subset \Pi_2(G)$ the subset of representations with central character $\chi$. When $\chi$ is unitary, the matrix coefficients of any representation $\pi\in \Pi_2(G,\chi)$ lie in $\mathcal{C}(G,\chi)$. For $\pi\in \Irr(G)$, we will denote by $\pi^\vee$ its smooth contragredient; and for $\pi\in \Pi_2(G)$, we will denote by $d(\pi)$ the formal degree of $\pi$. It is the unique positive real number (depending on the Haar measure on $G$) such that

$$\displaystyle \int_{A_G\backslash G} \langle \pi(g)v_1,v_1^\vee\rangle \langle v_2,\pi^\vee(g)v_2^\vee\rangle dg=\frac{1}{d(\pi)} \langle v_1,v_2^\vee\rangle \langle v_2,v_1^\vee\rangle$$

\noindent for all $v_1,v_2\in \pi$ and $v_1^\vee,v_2^\vee\in \pi^\vee$. For any $f\in C_c^\infty(G)$ and $\pi\in \Irr(G)$, we write

$$\displaystyle \pi(f):=\int_G f(g)\pi(g)dg.$$

\noindent When $\pi\in \Pi_2(G,\chi)$ where the character $\chi$ is unitary, the map $f\mapsto \pi(f)$ extends by continuity to $\mathcal{C}(G)$ and $\mathcal{C}(G,\chi^{-1})$. In all cases, the operator $\pi(f)$ has finite rank. If $f$ is a matrix coefficient of $\pi\in \Pi_2(G,\chi)$ (with $\chi$ unitary), we have

$$\displaystyle \Tr \pi^\vee(f)=d(\pi)^{-1}f(1). \leqno (1)$$

\noindent Moreover, for any $\pi\in \Irr(G)$, Harish-Chandra has shown (\cite{HCH} Theorem 16.3) the existence of a locally integrable function $\Theta_\pi$ on $G$ which is locally constant on $G_{reg}$ and such that

$$\displaystyle \Tr \pi(f)=\int_G f(g)\Theta_\pi(g)dg$$

\noindent for all $f\in C_c^\infty(G)$. We shall refer to $\Theta_\pi$ as the \textit{Harish-Chandra character} of $\pi$. Fixing a $G$-invariant symmetric bilinear pairing $\langle .,.\rangle:\mathfrak{g}\times \mathfrak{g}\to F$. Near every semi-simple element $x\in G$, there is a local expansion (see \cite{HCH} Theorem 16.2)

$$\displaystyle \Theta_\pi(x \exp(X))=\sum_{\mathcal{O}\in Nil(\mathfrak{g}_x)} c_{\pi,\mathcal{O}}(x)\widehat{j}(\mathcal{O},X)$$

\noindent for $X\in \mathfrak{g}_{x,reg}$ sufficiently close to $0$ and where

\begin{itemize}
\item $Nil(\mathfrak{g}_x)$ stands for the set of nilpotent $G_x$-orbits in $\mathfrak{g}_x$ (for the adjoint action);

\item $c_{\pi,\mathcal{O}}(x)$ are complex numbers;

\item For all $\mathcal{O}\in Nil(\mathfrak{g}_x)$, $\widehat{j}(\mathcal{O},.)$ is the unique locally integrable function on $\mathfrak{g}_x$ (whose existence is guaranteed by \cite{HCH} Theorem 7.7. and Lemma 7.9) which is locally constant on $\mathfrak{g}_{x,reg}$, and such that

$$\displaystyle \int_{\mathfrak{g}_x}\varphi(X)\widehat{j}(\mathcal{O},X)dX=\int_{\mathcal{O}}\widehat{\varphi}(Z)dZ,\;\; \mbox{ for all } \varphi\in C_c^\infty(\mathfrak{g}_x)$$

\noindent where $dX$ is any Haar measure on $\mathfrak{g}_x$, $\varphi\in C_c^\infty(\mathfrak{g}_x)\mapsto \widehat{\varphi}$ is the Fourier transform given by $\widehat{\varphi}(Z):=\int_{\mathfrak{g}_x} \varphi(X) \psi(\langle Z,X\rangle)dX$ and $dZ$ is the $G_x$-invariant measure on $\mathcal{O}$ associated to the self-dual Haar measure on $F$ corresponding to $\psi$ and the volume form on $\mathcal{O}$ derived from the symplectic form descended from $\langle .,.\rangle$ (see \cite{MW} I.8 for more details on this).
\end{itemize}

\noindent For every semisimple element $x\in G$, we set

$$\displaystyle c_\pi(x):=\left\{
    \begin{array}{ll}
        \frac{1}{\lvert Nil_{reg}(\mathfrak{g}_x)\rvert}\sum_{\mathcal{O}\in Nil_{reg}(\mathfrak{g}_x)} c_{\pi,\mathcal{O}}(x), & \mbox{if } G_x \mbox{ is quasi-split;} \\
        0, & \mbox{ otherwise}
    \end{array}
\right.$$

\noindent where $Nil_{reg}(\mathfrak{g}_x)$ denotes the subset of regular nilpotent orbits in $\mathfrak{g}_x$ (this set is empty if $G_x$ is not quasi-split). This value does not depend on the choices of $\langle .,.\rangle$ and $\psi$. If $G_x$ is quasi-split and we fix a Borel subgroup $B_x\subset G_x$ and a maximal torus $T_{x,qd}\subset B_x$, then by Proposition 4.5.1(ii) of \cite{BeuGGP}, we have

$$\displaystyle D^G(x)^{1/2}c_\pi(x)=\lvert W(G_x,T_{x,qd})\rvert^{-1}\lim\limits_{x'\in T_{x,qd}\to x} D^G(x')^{1/2}\Theta_\pi(x'). \leqno (2)$$

\subsection{$(G,M)$- and $(G,M,\theta)$-orthogonal sets}\label{(G,M) and (G,M,theta)-orthogonal sets}

Let $G$ be a connected reductive group over $F$ and $M$ be a Levi subgroup of $G$. For all $Q\in \mathcal{F}(M)$, we will denote by $U_Q$ the unipotent radical of $Q$, $L_Q$ the unique Levi component of $Q$ such that $M\subset L_Q$ and $\overline{Q}=L_QU_{\overline{Q}}$ the parabolic subgroup opposite to $Q$ (with respect to $L_Q$). Let $A_M$ be the split center of $M$. For all $P\in \mathcal{P}(M)$, denote by $\Delta_P$ (resp. $\Sigma_P^+$) the set of simple roots (resp. of all roots) of $A_M$ in $P$. For all $\alpha\in \Sigma_P^+$, we shall denote by $\alpha^\vee\in \mathcal{A}_M$ the corresponding coroot and we set $\Delta_P^\vee:=\{\alpha^\vee; \alpha\in \Delta_P \}$. We recall the notion of $(G,M)$-orthogonal set due to Arthur: a family $\mathcal{Y}=(\mathcal{Y}_P)_{P\in \mathcal{P}(M)}$ is a \textit{$(G,M)$-orthogonal set} if for all $P,P'\in \mathcal{P}(M)$, we have

$$\displaystyle \mathcal{Y}_P-\mathcal{Y}_{P'}\in \sum_{\alpha\in \Sigma_P^+\cap -\Sigma_{P'}^+}\mathbf{R}\alpha^\vee.$$

\noindent Moreover, if the stronger relation

$$\displaystyle \mathcal{Y}_P-\mathcal{Y}_{P'}\in \sum_{\alpha\in \Sigma_P^+\cap -\Sigma_{P'}^+}\mathbf{R}_+\alpha^\vee$$

\noindent is satisfied for all $P,P'\in \mathcal{P}(M)$, then we say that the $(G,M)$-orthogonal set $\mathcal{Y}$ is \textit{positive}. To a $(G,M)$-orthogonal set $\mathcal{Y}$ we can associate a smooth function $\gamma_M(.,\mathcal{Y})$ on $\mathcal{A}_M^*$ defined by (see Lemma 1.9.3 of \cite{LW})

$$\displaystyle \gamma_M(\lambda,\mathcal{Y}):=\sum_{P\in \mathcal{P}(M)} \vol(\mathcal{A}_M^G/\mathbf{Z}[\Delta_P^\vee])\prod_{\alpha\in \Delta_P} \langle \lambda, \alpha^\vee\rangle^{-1} q^{\langle \lambda, \mathcal{Y}_P\rangle},\;\;\; \lambda\in \mathcal{A}_M^*$$

\noindent where $\mathbf{Z}[\Delta_P^\vee]\subset \mathcal{A}_M^G$ denotes the lattice generated by $\Delta_P^\vee$. And we set

$$\displaystyle v_M(\mathcal{Y}):=\gamma_M(0,\mathcal{Y}).$$

\noindent More generally we can associate to a $(G,M)$-orthogonal set $\mathcal{Y}$ smooth functions $\gamma_M^Q(.,\mathcal{Y})$ on $\mathcal{A}_M^{L_Q}$ and complex numbers $v_M^Q(\mathcal{Y}):=\gamma_M^Q(0,\mathcal{Y})$ for all $Q\in \mathcal{F}(M)$ with $\gamma_M^G(.,\mathcal{Y})=\gamma_M(.,\mathcal{Y})$ (see \cite{LW} \S 1.9). If the $(G,M)$-family $\mathcal{Y}$ is positive, then $v^Q_M(\mathcal{Y})$ is just the volume of the convex hull of the projection of $(Y_P)_{P\in \mathcal{P}(M),P\subset Q}$ onto $\mathcal{A}_M^{L_Q}$. Also to a $(G,M)$-orthogonal set $\mathcal{Y}$ we can associate a certain function $\Gamma_M^G(.,\mathcal{Y})$ on $\mathcal{A}_M$ (see \cite{LW} \S 1.8) which is just the characteristic function of the sum of the convex hull of $\mathcal{Y}$ with $\mathcal{A}_G$ if $\mathcal{Y}$ is positive.

An easy way to construct $(G,M)$-orthogonal sets is as follows. Let $M_0\subset M$ be a minimal Levi subgroup with split center $A_0$ and pick $P_0\in \mathcal{P}(M_0)$. Fix $Y\in \mathcal{A}_{M_0}$. For all $P\in \mathcal{P}(M)$, define $Y_{P}$ to be the projection of $wY$ onto $\mathcal{A}_M$ where $w\in W(G,A_0)$ is any element such that $wP_0\subset P$. Then $(Y_P)_{P\in \mathcal{P}(M)}$ is a $(G,M)$-orthogonal set.

Another way to construct $(G,M)$-orthogonal sets is as follows. Choose a maximal special compact subgroup $K$ of $G$ and use it to define maps $H_P:G\to \mathcal{A}_M$ as in \S \ref{groups, measures, notations}. Then for all $g\in G$, the family $\mathcal{Y}_M(g):=(H_{\overline{P}}(g))_{P\in \mathcal{P}(M)}$ is a positive $(G,M)$-orthogonal set. In this situation, we define

$$\displaystyle v_M(g):=v_M(\mathcal{Y}_M(g)),\;\;\; g\in G$$

\noindent and more generally

$$\displaystyle v^Q_M(g):=v^Q_M(\mathcal{Y}_M(g)),\;\;\; g\in G$$

\noindent for all $Q\in \mathcal{F}(M)$.

Assume now given an algebraic involution $\theta$ of $G$. Then we recall that a parabolic subgroup $P$ of $G$ is said to be $\theta$-split if $\theta(P)$ is opposite to $P$; and a Levi subgroup $M$ of $G$ is said to be $\theta$-split if there exists a $\theta$-split parabolic subgroup $P$ such that $M=P\cap \theta(P)$. Also, a torus $T\subset G$ is said to be $\theta$-split if $\theta(t)=t^{-1}$ for all $t\in T$. We refer the reader to \cite{BeuGalP} \S 1.7.1 for a recapitulation of the basic structure of these $\theta$-split subgroups. For $M$ a $\theta$-split Levi subgroup, we shall denote by $\mathcal{P}^\theta(M)$, resp. $\mathcal{F}^\theta(M)$, resp. $\mathcal{L}^\theta(M)$, the sets of $\theta$-split parabolic subgroups with Levi component $M$, resp. $\theta$-split parabolic subgroups containing $M$, resp. $\theta$-split Levi subgroups containing $M$. For all $Q\in \mathcal{F}^\theta(M)$, we define

$$\displaystyle \mathcal{P}^{Q,\theta}(M):=\{P\in \mathcal{P}^\theta(M);\; P\subset Q \}.$$

\noindent This set is in bijection with the set of $\theta$-split parabolic subgroups of $L_Q$ with Levi component $M$ by the map $P\mapsto P\cap L_Q$. We will denote by $A_{M,\theta}$ the maximal split and $\theta$-split central subtorus of $M$ and set

$$\displaystyle \mathcal{A}_{M,\theta}:=X_*(A_{M,\theta})\otimes \mathbf{R},$$

$$\displaystyle \mathcal{A}^L_{M,\theta}:=\mathcal{A}_{M,\theta}/\mathcal{A}_{L,\theta}, \;\;\; L\in \mathcal{L}^\theta(M).$$

\noindent Then for all $\theta$-split Levi subgroup $M$, there is a natural decomposition

$$\displaystyle \mathcal{A}_M=\mathcal{A}_{M,\theta}\oplus \mathcal{A}_M^\theta$$

\noindent where $\mathcal{A}_M^\theta$ denotes the subspace of $\theta$-invariant vectors and we define an homomorphism $H_{M,\theta}:M\to \mathcal{A}_{M,\theta}$ as the composition of $H_M$ with the projection $\mathcal{A}_M\twoheadrightarrow \mathcal{A}_{M,\theta}$. There is also a natural decomposition

$$\displaystyle \mathcal{A}_{M,\theta}=\mathcal{A}_{M,\theta}^G\oplus \mathcal{A}_{G,\theta}$$

\noindent and more generally a natural decomposition

$$\displaystyle \mathcal{A}_{M,\theta}=\mathcal{A}_{M,\theta}^L\oplus \mathcal{A}_{L,\theta}$$

\noindent for all $L\in \mathcal{L}^\theta(M)$. Set $\mathcal{A}_{M,\theta,F}:=H_{M,\theta}(A_M)$. We equip $\mathcal{A}_{M,\theta}$ with the unique Haar measure such that $\vol(\mathcal{A}_{M,\theta}/\mathcal{A}_{M,\theta,F})=1$ and $\mathcal{A}_{M,\theta}^L$ for $L\in \mathcal{L}^\theta(M)$ with the quotient Haar measure (where the Haar measure on $\mathcal{A}_{L,\theta}$ is defined similarly).

\vspace{2mm}

To every $P\in \mathcal{P}^\theta(M)$ is associated a cone $\mathcal{A}_{P,\theta}^+\subset \mathcal{A}_{M,\theta}$ defined by

$$\displaystyle \mathcal{A}_{P,\theta}^+:=\{\Lambda\in \mathcal{A}_{M,\theta};\;\; \langle \alpha, \Lambda\rangle>0 \; \forall \alpha\in R(A_{M,\theta},U_P) \}.$$

\noindent We shall denote by $\overline{\mathcal{A}_{P,\theta}^+}$ the closure of $\mathcal{A}_{P,\theta}^+$ and by $\tau_{P,\theta}^G$ the characteristic function of $\mathcal{A}_{P,\theta}^+$. For all $Q\in \mathcal{F}^\theta(M)$, we will also consider the function $\tau_{Q,\theta}^G$ as a function on $\mathcal{A}_{M,\theta}$ via the projection $\mathcal{A}_{M,\theta}\to \mathcal{A}_{L_Q,\theta}$.

\vspace{2mm}

In \cite{BeuGalP} \S 1.7.2 was defined a notion of $(G,M,\theta)$-orthogonal set which generalizes Arthur's classical notion of $(G,M)$-orthogonal set and we refer the readers to \textit{loc. cit.} for basic definitions and properties of these. A $(G,M,\theta)$-orthogonal set is a family $\mathcal{Y}=(\mathcal{Y}_P)_{P\in \mathcal{P}^\theta(M)}$ of points of $\mathcal{A}_{M,\theta}$ satisfying certain compatibility conditions. There is also a notion of positive $(G,M,\theta)$-orthogonal set. If $\mathcal{Y}_P\in \mathcal{A}_{P,\theta}^+$ for all $P\in \mathcal{P}^\theta(M)$, then the $(G,M,\theta)$-orthogonal set $\mathcal{Y}$ is positive. To any $(G,M,\theta)$-orthogonal set $\mathcal{Y}$ is associated functions

$$\displaystyle \Gamma^Q_{L,\theta}(.,\mathcal{Y}),\;\;\; L\in \mathcal{L}^\theta(M), Q\in \mathcal{F}^\theta(L)$$

\noindent on $\mathcal{A}_{L,\theta}^{L_Q}$ and complex numbers

$$\displaystyle v^Q_{L,\theta}(\mathcal{Y}),\;\;\; L\in \mathcal{L}^\theta(M), Q\in \mathcal{F}^\theta(L)$$

\noindent which are related by

$$\displaystyle v^Q_{L,\theta}(\mathcal{Y})=\int_{\mathcal{A}_{L,\theta}^{L_Q}} \Gamma^Q_{L,\theta}(\Lambda,\mathcal{Y}) d\Lambda.$$

\noindent For simplicity, when $Q=G$, we will write $v_{L,\theta}(\mathcal{Y}):=v_{L,\theta}^G(\mathcal{Y})$. When $\mathcal{Y}$ is positive, $\Gamma_{M,\theta}^G(.,\mathcal{Y})$ is the characteristic function of the sum of $\mathcal{A}_G$ with the convex hull of $\mathcal{Y}$; and more generally, for $Q\in \mathcal{F}^\theta(M)$, $\Gamma_{M,\theta}^Q(.,\mathcal{Y})$ is the characteristic function of the sum of $\mathcal{A}_{L_Q,\theta}$ with the convex hull of $(\mathcal{Y}_P)_{P\in \mathcal{P}^{Q,\theta}(M)}$. We have the basic relation (see \cite{LW} Lemme 1.8.4 (3) for the case of $(G,M)$-orthogonal sets, the proof being completely similar for $(G,M,\theta)$-orthogonal sets)

$$\displaystyle \sum_{Q\in \mathcal{F}^\theta(M)} \Gamma_{M,\theta}^Q(\Lambda,\mathcal{Y}) \tau_{Q,\theta}^G(\Lambda-\mathcal{Y}_Q)=1,\;\;\; \Lambda\in \mathcal{A}_{M,\theta} \leqno (1)$$

\noindent where for all $Q\in \mathcal{F}^\theta(M)$, we have denoted by $\mathcal{Y}_Q$ the projection of $\mathcal{Y}_P$ onto $\mathcal{A}_{L_Q,\theta}$ for any $P\in \mathcal{P}^{Q,\theta}(M)$ (the result does not depend on the choice of $P$). One basic property of the function $\Gamma_{M,\theta}^G(.,\mathcal{Y})$ that we shall use repeatedly is the following (see \cite{LW} Corollaire 1.8.5, again for the case of $(G,M)$-orthogonal sets):

\vspace{2mm}

\hspace{5mm} (2) Let $\lvert .\rvert$ be a norm on $\mathcal{A}_{M,\theta}$. Then, there exists a constant $c>0$ independent

\vspace{0.4pt}

\hspace{12mm} of the $(G,M,\theta)$-family $\mathcal{Y}$ such that for all $\Lambda\in \mathcal{A}_{M,\theta}$ in the support of $\Gamma_{M,\theta}^G(.,\mathcal{Y})$,

\vspace{0.4pt}

\hspace{12mm} we have $\lvert \Lambda^G\rvert\leqslant c\sup_{P\in \mathcal{P}^\theta(M)} \lvert \mathcal{Y}_P\rvert$ where  $\Lambda^G$ is the projection of $\Lambda$ onto $\mathcal{A}_{M,\theta}^{G}$.

\vspace{2mm}

\noindent In particular, this implies

\vspace{2mm}

\hspace{5mm} (3) There exists $k>0$ (for example $k=\dim(\mathcal{A}_{M,\theta}^G)$ would work) such that

$$\displaystyle \lvert v_{M,\theta}^G(\mathcal{Y})\rvert\ll \left( \sup_{P\in \mathcal{P}^\theta(M)} \lvert \mathcal{Y}_P\rvert\right)^k$$

\hspace{12mm} for all $(G,M,\theta)$-orthogonal set $\mathcal{Y}$.

\vspace{2mm}

\noindent We will also need the following property:

\vspace{2mm}

\hspace{5mm} (4) Let $Q\in \mathcal{F}^\theta(M)$ and $P\in \mathcal{P}^\theta(M)$. Then, for all $(G,M,\theta)$-orthogonal set $\mathcal{Y}$ such

\vspace{0.4pt}

\hspace{12mm} that $\mathcal{Y}_{P'}\in \mathcal{A}_{P',\theta}^+$ for all $P'\in \mathcal{P}^\theta(M)$, the restriction of the function

$$\Lambda\mapsto \Gamma_{M,\theta}^Q(\Lambda, \mathcal{Y})\tau_{Q,\theta}^G(\Lambda-\mathcal{Y}_Q)$$

\hspace{12mm} to $\overline{\mathcal{A}_{P,\theta}^+}$ only depends on $\mathcal{Y}_P$.

\vspace{2mm}

\noindent\ul{Proof}: Since $\mathcal{Y}$ is positive, the function $\Gamma_{M,\theta}^Q(., \mathcal{Y})\tau_{Q,\theta}^G(.-\mathcal{Y}_Q)$ is the characteristic function of the sum of $\mathcal{A}_{Q,\theta}^+$ with the convex hull of $(\mathcal{Y}_{P'})_{P'\in \mathcal{P}^{Q,\theta}(M)}$. In particular, for all $\Lambda$ in the support of this function, we have $\langle \alpha, \Lambda\rangle\geqslant \inf_{P'\in \mathcal{P}^{\theta,Q}(M)} \langle \alpha,\mathcal{Y}_{P'}\rangle>0$ for all $\alpha\in R(A_{M,\theta},U_Q)$. This implies that if $P$ is not included in $Q$, the restriction of $\Gamma_{M,\theta}^Q(., \mathcal{Y})\tau_{Q,\theta}^G(.-\mathcal{Y}_Q)$ to $\overline{\mathcal{A}_{P,\theta}^+}$ is just identically zero. Assume now that $P\subset Q$. Then, by adapting Lemma 3.1 of \cite{ArthurlocalTF} to the case of $(G,M,\theta)$-orthogonal sets, we see that the restriction of $\Gamma_{M,\theta}^Q(.,\mathcal{Y})$ to $\overline{\mathcal{A}_{P,\theta}^+}$ only depends on $\mathcal{Y}_P$. On the other hand, $\tau_{Q,\theta}^G(.-\mathcal{Y}_Q)$ only depends on $\mathcal{Y}_Q$ which is the projection of $\mathcal{Y}_P$ onto $\mathcal{A}_{L_Q,\theta}$. The claim follows.

\vspace{2mm}

Let $\mathcal{Y}_1$ and $\mathcal{Y}_2$ be two $(G,M,\theta)$-orthogonal sets. Then, we have the following splitting formula (see \cite{ArtInvTF} Corollary 7.4 for the case of $(G,M)$-orthogonal sets, the proof being again similar for $(G,M,\theta)$-orthogonal sets)

$$v_{M,\theta}(\mathcal{Y}_1+\mathcal{Y}_2)=\sum_{L_1,L_2\in \mathcal{L}^\theta(M)} d^G_{M,\theta}(L_1,L_2)v^{Q_1}_{M,\theta}(\mathcal{Y}_1)v_{M,\theta}^{Q_2}(\mathcal{Y}_2)\leqno (5)$$

\noindent where for all $L_1,L_2\in \mathcal{L}^\theta(M)$, $Q_1$ and $Q_2$ are elements of $\mathcal{P}^\theta(L_1)$ and $\mathcal{P}^\theta(L_2)$ respectively, which depend on the auxiliary choice of a generic point $\xi\in \mathcal{A}_{M,\theta}$, and $d^G_{M,\theta}(L_1,L_2)$ is a nonnegative real number which is nonzero if and only if $\mathcal{A}_{M,\theta}^G=\mathcal{A}_{M,\theta}^{L_1}\oplus \mathcal{A}_{M,\theta}^{L_2}$. Moreover, we have $d^G_{M,\theta}(G,M)=1$.

\vspace{2mm}

As for $(G,M)$-orthogonal sets, there is the following easy way to produce $(G,M,\theta)$-orthogonal sets. Let $M_0\subset M$ be a minimal $\theta$-split Levi subgroup and pick $P_0\in \mathcal{P}^\theta(M_0)$. Let $A_0$ be the maximal split and $\theta$-split central subtorus of $M_0$ and set

$$\displaystyle W_0:=\Norm_G(A_0)/M_0$$

\noindent for the \textit{little Weyl group} of $M_0$. Then, the natural action of $W_0$ on $\mathcal{P}^\theta(M_0)$ is simply transitive (see \cite{HW} Proposition 5.9). To every point $Y\in \mathcal{A}_{M_0,\theta}$, we can now associate a $(G,M,\theta)$-orthogonal set $(Y_P)_{P\in \mathcal{P}^\theta(M)}$ as follows: for each $P\in \mathcal{P}^\theta(M)$, set $Y_P$ to be the projection of $wY$ to $\mathcal{A}_{M,\theta}$ where $w\in W_0$ is any element such that $wP_0\subset P$.

\vspace{2mm}

Let $K$ be a maximal special compact subgroup of $G$. Then by using the Iwasawa decomposition $G=PK$, we can define maps $H_{P,\theta}:G\to \mathcal{A}_{M,\theta}$ for all $P\in \mathcal{P}^\theta(M)$ by setting $H_{P,\theta}(muk):=H_{M,\theta}(m)$ for all $m\in M$, $u\in U_P$ and $k\in K$. Then for all $g\in G$, the family $\mathcal{Y}_{M,\theta}(g):=(H_{\overline{P},\theta}(g))_{P\in \mathcal{P}^\theta(M)}$ is a positive $(G,M,\theta)$-orthogonal set and we will set

$$\displaystyle v_{M,\theta}(g):=v_{M,\theta}(\mathcal{Y}_{M,\theta}(g)),\;\;\; g\in G$$

\noindent and more generally

$$\displaystyle v^Q_{M,\theta}(g):=v^Q_{M,\theta}(\mathcal{Y}_{M,\theta}(g)),\;\;\; g\in G$$

\noindent for all $Q\in \mathcal{F}^\theta(M)$. We have the following descent formula which is a special case of a general result of Arthur (see \cite{ArtInvTF} Proposition 7.1 and \cite{BeuGalP} 1.7.2 (4))

$$\displaystyle v_{M,\theta}(g)=\sum_{L\in \mathcal{L}(L)} d^G_{M,\theta}(L) v_M^G(g),\;\;\; g\in G \leqno (6)$$

\noindent where for all $L\in \mathcal{L}(M)$, $Q$ is a certain parabolic subgroup with Levi component $L$ which depends on the choice of a generic point $\xi\in \mathcal{A}_M$ and $d^G_{M,\theta}(L)$ is a coefficient which is nonzero only if $\mathcal{A}_M^G=\mathcal{A}_M^{G,\theta}\oplus \mathcal{A}_M^L$. Moreover, if $\mathcal{A}_M^{G,\theta}=0$, then $d^G_{M,\theta}(G)=1$.

\subsection{Weighted and $\theta$-weighted orbital integrals}\label{weighted and theta weighted orbital integrals}

Let $G$ be a connected reductive group over $F$, $M$ be a Levi subgroup of $G$ and $f\in \mathcal{C}(G)$. Fix a special maximal compact subgroup $K$ of $G$ that we use to define weights $g\mapsto v_M^Q(g)$, for $Q\in \mathcal{F}(M)$, as in the previous section. Then, for all $x\in M\cap G_{reg}$ and $Q\in \mathcal{F}(M)$ we define, following Arthur, a \textit{weighted orbital integral} by

$$\displaystyle \Phi^Q_M(x,f):=\int_{G_x\backslash G} f(g^{-1}xg) v_M^Q(g)dg.$$

\noindent In the particular case where $Q=G$, we simply set $\Phi_M(x,f):=\Phi^G_M(x,f)$.

Assume now given an algebraic involution $\theta$ of $G$ and that $M$ is $\theta$-split. Using the same special maximal compact subgroup $K$ we associate, as in the previous paragraph, to any $Q\in \mathcal{F}^\theta(M)$, a weight $g\in G\mapsto v_{M,\theta}^Q(g)$. Then for all $x\in M\cap G_{reg}$ and $Q\in \mathcal{F}^\theta(M)$, we define a \textit{$\theta$-weighted orbital integral} by

$$\displaystyle \Phi^Q_{M,\theta}(x,f):=\int_{G_x\backslash G} f(g^{-1}xg) v_{M,\theta}^Q(g)dg.$$

\noindent In the particular case where $Q=G$, we simply set $\Phi_{M,\theta}(x,f):=\Phi^G_{M,\theta}(x,f)$.

\subsection{Cusp forms and $\theta$-strongly cuspidal functions}\label{cusp forms and theta-strongly cuspidal functions}

Let $G$ be a connected reductive group over $F$. Following \cite{WalGGPI}, we say that a function $f\in \mathcal{C}(G)$ is \textit{strongly cuspidal} if for all proper parabolic subgroup $P=MU$ of $G$, we have

$$\displaystyle \int_U f(mu)du=0$$

\noindent for all $m\in M$. By \cite{BeuGGP} Lemma 5.2.1 (i), if $f\in \mathcal{C}(G)$ is strongly cuspidal, $M$ is a Levi subgroup of $G$ and $Q\in \mathcal{F}(M)$ is different from $G$, then we have

$$\displaystyle \Phi_M^Q(x,f)=0$$

\noindent for all $x\in M\cap G_{reg}$ where the weighted orbital integral $\Phi_M^Q(x,f)$ is defined by using any special maximal compact subgroup $K$ of $G$.

\vspace{2mm}

Let $\theta$ an algebraic involution of $G$. We say that a function $f\in \mathcal{C}(G)$ is \textit{$\theta$-strongly cuspidal} if for all proper $\theta$-split parabolic subgroup $P=MU\subset G$, we have

$$\displaystyle \int_U f(g^{-1}mug)du=0$$

\noindent for all $m\in M$ and $g\in G$. By a standard change of variable, $f\in \mathcal{C}(G)$ is $\theta$-strongly cuspidal if and only if for all proper $\theta$-split parabolic subgroup $P=MU$, all $m\in M\cap G_{reg}$ and all $g\in G$, we have

$$\displaystyle \int_{U} f(g^{-1}u^{-1}mug)du=0.$$

\noindent By a proof similar to \cite{BeuGGP} Lemma 5.2.1 (i), if $f\in \mathcal{C}(G)$ is $\theta$-strongly cuspidal, $M$ is a $\theta$-split Levi subgroup of $G$ and $Q\in \mathcal{F}^{\theta}(M)$ is different from $G$, then we have

$$\displaystyle \Phi_{M,\theta}^Q(x,f)=0$$

\noindent for all $x\in M\cap G_{reg}$ where the $\theta$-weighted orbital integral $\Phi_{M,\theta}^Q(x,f)$ is defined by using any special maximal compact subgroup $K$ of $G$.

\vspace{2mm}

To a strongly cuspidal function $f\in \mathcal{C}(G)$ we associate a function $\Theta_f$ on $G_{reg}$ defined by

$$\displaystyle \Theta_f(x):=(-1)^{a_{G_x}-a_G}\Phi_{M(x)}^G(x,f)$$

\noindent where $M(x):=Cent_G(A_{G_x})$ (i.e. the minimal Levi subgroup containing $x$), $a_{G_x}:=\dim(A_{G_x})$, $a_G:=\dim(A_G)$ and the weighted orbital integral $\Phi_{M(x)}^G(x,f)$ is defined by using any special maximal compact subgroup $K$ of $G$ (the result is independent of this choice, see \cite{WalGGPI} Lemme 5.2). Then, by \cite{WalGGPI} Corollaire 5.9, $\Theta_f$ is a quasi-character in the sense of {\it loc. cit.}. This means that for all semi-simple element $x\in G$, we have a local expansion

$$\displaystyle \Theta_f(x\exp(X))=\sum_{\mathcal{O}\in Nil(\mathfrak{g}_x)}c_{f,\mathcal{O}}(x) \widehat{j}(\mathcal{O},X)$$

\noindent for all $X\in \mathfrak{g}_{x,reg}$ sufficiently near $0$, where $c_{f,\mathcal{O}}(x)$, $\mathcal{O}\in Nil(\mathfrak{g}_x)$, are complex numbers and the other notations have been defined in \S \ref{representations}. For all semi-simple element $x\in G$, we set

$$\displaystyle c_f(x):=\left\{
    \begin{array}{ll}
        \frac{1}{\lvert Nil_{reg}(\mathfrak{g}_x)\rvert}\sum_{\mathcal{O}\in Nil_{reg}(\mathfrak{g}_x)} c_{f,\mathcal{O}}(x), & \mbox{if } G_x \mbox{ is quasi-split;} \\
        0, & \mbox{ otherwise}
    \end{array}
\right.$$

\noindent where we recall that $Nil_{reg}(\mathfrak{g}_x)$ denotes the subset of regular nilpotent orbits in $\mathfrak{g}_x$. This value does not depend on the choices of $\langle .,.\rangle$ and $\psi$. If $G_x$ is quasi-split and we fix a Borel subgroup $B_x\subset G_x$ and a maximal torus $T_{x,qd}\subset B_x$, then by Proposition 4.5.1(ii) of \cite{BeuGGP}, we have

$$\displaystyle D^G(x)^{1/2}c_f(x)=\lvert W(G_x,T_{x,qd})\rvert^{-1}\lim\limits_{x'\in T_{x,qd}\to x} D^G(x')^{1/2}\Theta_f(x'). \leqno (1)$$

\noindent Moreover, by \cite{BeuGGP} Proposition 4.5.1 (iii), the function $(D^G)^{1/2}c_f$ is locally bounded on $G$.

Let $\chi$ be a unitary character of $A_G$. We say, following Harish-Chandra, that a function $f\in \mathcal{C}(G)$ or $f\in \mathcal{C}(G,\chi)$ is a {\it cusp form} if for all proper parabolic subgroup $P=MU$ of $G$, we have

$$\displaystyle \int_U f(xu)du=0$$

\noindent for all $x\in G$. Of course, for functions in $\mathcal{C}(G)$ being a cusp form implies being strongly cuspidal. We shall denote by ${}^0\mathcal{C}(G)$ and ${}^0\mathcal{C}(G,\chi)$ the spaces of cusp forms in $\mathcal{C}(G)$ and $\mathcal{C}(G,\chi)$ respectively. For each $\pi\in \Pi_2(G,\chi)$, the matrix coefficients of $\pi$ belong to ${}^0\mathcal{C}(G,\chi)$ (\cite{HCH} Theorem 29). And if $f$ is such a matrix coefficient, we have (see \cite{BeuGalP} 1.6(3))

$$\displaystyle \Theta_f=d(\pi)^{-1}f(1)\Theta_\pi. \leqno (2)$$

\noindent Moreover, any element in the space ${}^0\mathcal{C}(G,\chi)$ can be written as a finite linear combination of matrix coefficients of representations inside $\Pi_2(G,\chi)$. As a special case of Harish-Chandra-Plancherel formula (\cite{WalPlanch} Theorem VIII.4.2), for all $f\in {}^0\mathcal{C}(G,\chi)$, we have an equality

$$\displaystyle f=\sum_{\pi\in \Pi_2(G,\chi)}d(\pi)f_\pi \leqno (3)$$

\noindent where we have set $f_\pi(g):=\Tr(\pi^\vee(g^{-1})\pi^\vee(f))$ for all $\pi\in \Pi_2(G,\chi)$.

\section{A multiplicity formula for generalized Shalika models}\label{A multiplicity formula for Shalika models}

\subsection{Generalized Shalika triples}\label{Shalika triples}

From now on and until the end of the paper we fix a central simple algebra $\mathcal{A}$ over $F$ of rank $n$ (i.e. $\CA=\Mat_{m\times m}(\mathcal{D})$ where $\mathcal{D}/F$ is a division algebra of degree $r$ and $n=mr$). $\Tr_{\mathcal{A}/F}:\mathcal{A}\to F$ and $N_{\mathcal{A}/F}:\mathcal{A}\to F$ will stand for the reduced trace and norm respectively, and we will set $\nu(.):=\lvert N_{\mathcal{A}/F}(.)\rvert_F$. We also fix a maximal order $\mathcal{O}_{\mathcal{A}}$ of $\mathcal{A}$. Set $G:=GL_2(\mathcal{A})$ and define the following subgroups of $G$:

\begin{itemize}
\item $K:=GL_2(\mathcal{O}_{\mathcal{A}})$ (a maximal compact subgroup of $G$);

\item $H_0:=\{\begin{pmatrix}
\lambda & \\ & \lambda \end{pmatrix}\mid \lambda\in \mathcal{A}^\times \}$;

\item $N:=\{ \begin{pmatrix}
 1 & X \\ & 1 \end{pmatrix}\mid X\in \mathcal{A}\}$;

\item $H:=H_0\ltimes N$;

\item $L:=\{\begin{pmatrix} \lambda & \\ & \mu\end{pmatrix}\mid \lambda,\mu \in \mathcal{A}^\times \}$ (a Levi subgroup of $G$);

\item $Q:=LN=\{\begin{pmatrix} \lambda & X\\ & \mu\end{pmatrix}\mid \lambda,\mu \in \mathcal{A}^\times, X\in \mathcal{A} \}$ (a parabolic subgroup of $G$).
\end{itemize}

\noindent Note that by combining Lemme II.1.5 and Proposition II.4.5 of \cite{WalPlanch}, the subgroup $H$ has the following property (we shall say that $H$ is \textit{strongly discrete} following \cite{GO}):

$$\displaystyle \mbox{There exists } d>0 \mbox{ such that the integral } \int_{H} \Xi^G(h)\sigma_{G}(h)^{-d}dh \mbox{ converges.} \leqno (1)$$

\noindent We fix a continuous character $\omega:F^\times\to \mathbf{C}^\times$ that we identify to a character of $H_0$ through composition with $N_{\mathcal{A}/F}:H_0\to F^\times$. We then define a character $\xi:N\to \mathbf{C}^\times$ by

$$\displaystyle \xi\begin{pmatrix} 1 & X \\ & 1 \end{pmatrix}:=\psi(\Tr_{\mathcal{A}/F} X),\;\;\; X\in \mathcal{A}.$$

\noindent Then $\xi$ is invariant under the $H_0$-conjugation and thus extends to a character $\xi$ of $H$ trivial on $H_0$. We can also consider $\omega$ as a character on $H$ by composition with the projection $H\twoheadrightarrow H_0$ and we will denote by $\omega\otimes \xi$ the product of these two characters of $H$. We refer to the triple $(G,H,\omega\otimes \xi)$ as a \textit{generalized Shalika triple}. In particular, if $\CA=\Mat_n(F)$, this is the usual Shalika model. For all $\pi\in \Irr(G)$ we define the \textit{multiplicity} $m(\pi,\omega)$ to be

$$\displaystyle m(\pi, \omega):=\dim \Hom_{H}(\pi,\omega\otimes \xi).$$

\noindent By \cite{Del} Theorem 4.5 we know that this multiplicity is always finite.

\subsection{A simple local trace formula for the generalized Shalika models}\label{A simple local trace formula for Shalika models}

Here we assume that the character $\omega$ is unitary. Let $f\in \mathcal{C}(G)$. For all $x,y\in G$ we set

$$\displaystyle K_f(x,y):=\int_H f(x^{-1}hy) (\omega\otimes \xi)(h)^{-1}dh.$$

\noindent This integral is absolutely convergent by \ref{Shalika triples}(1). Moreover, whenever convergent, we define the following expression

$$\displaystyle J(f):=\int_{H\backslash G}K_f(x,x) dx.$$

\noindent One of the main results of this paper is the following theorem which might be seen as some sort of simple local trace formula in the setting of the generalized Shalika models.

\begin{thm}\label{theo trace formula}
Assume that $f\in {}^0\mathcal{C}(G)$ and $\omega$ is unitary. Then, the expression defining $J(f)$ is absolutely convergent and we have the following two expansions of it:

$$\displaystyle \sum_{T\in \mathcal{T}_{ell}(H_0)} \lvert W(H_0,T)\rvert^{-1} \int_T D^{H}(t) c_f(t) \omega(t)^{-1}dt=J(f)=\sum_{\pi\in \Pi_2(G,\chi)} m(\pi,\omega) \Tr \pi^\vee(f)$$

\noindent where $\mathcal{T}_{ell}(H_0)$ is a set of representatives of conjugacy classes of maximal elliptic tori in $H_0$, $c_f(t)$ is defined in Section \ref{cusp forms and theta-strongly cuspidal functions}, and $\chi=\omega^n$ seen as a character of $A_G=F^\times$.
\end{thm}

Note that the summation on the right hand side of the equality above is a finite sum (by \cite{WalPlanch} Th\'eor\`eme VIII.1.2), hence it is convergent. The integrals on the left hand side are absolutely convergent by the following lemma.

\begin{lem}\label{convergence of geometric side}
With the same assumptions as in Theorem \ref{theo trace formula}, the integral
$$\int_T D^{H}(t) c_f(t) \omega(t)^{-1}dt$$
is absolutely convergent for all $T\in \CT_{ell}(H_0)$.
\end{lem}

\begin{proof}
We can rewrite the integral as
$$\int_{A_G\backslash T} D^{H}(t) c_{f_\chi}(t) \omega(t)^{-1}dt$$
where we recall that $f_\chi(g):=\int_{A_G} f(ag)\chi(a)^{-1}da$. Since $T$ is elliptic, $A_G\backslash T$ is compact. Together with the assumption that $\omega$ is unitary, it is enough to show that the function
$$t\in T_{reg}\to D^{H}(t) c_{f_\chi}(t)$$
is locally bounded on $T$. This just follows from the fact that the function $(D^G)^{1/2}c_f$ is locally bounded on $G$ (Proposition 4.5.1 (iii) of \cite{BeuGGP}), and $D^{H}(t)=D^{G}(t)^{1/2}$ for all $t\in T_{reg}$.
\end{proof}

\noindent The proof of this Theorem will occupy Sections \ref{The spectral side} and \ref{geom side} entirely. In section \ref{The spectral side}, we will prove the absolute convergence of $J(f)$ when $f\in {}^0\mathcal{C}(G)$ together with the spectral expansion (that is the second equality of the Theorem). It is the easy part and moreover the arguments are very similar to \cite{BeuGalP} \S 3. Section \ref{geom side} on the other hand contains the proof of the geometric side (i.e. the first equality of the Theorem) which is more involved than that of the spectral side.

\subsection{The multiplicity formula}\label{A formula for the multiplicities}

The main interest of Theorem \ref{theo trace formula} is the following consequence of it.

\begin{prop}\label{prop multiplicities}
Let $\chi=\omega^n$ seen as a character of $A_G=F^\times$. Then, for all $\pi\in \Pi_2(G,\chi)$, we have

$$\displaystyle m(\pi,\omega)=\sum_{T\in \mathcal{T}_{ell}(H_0)} \lvert W(H_0,T)\rvert^{-1} \int_{A_G\backslash T} D^{H}(t) c_\pi(t) \omega(t)^{-1}dt$$

\noindent where we recall that the Haar measures on the tori $A_G\backslash T$, $T\in \mathcal{T}_{ell}(H_0)$, are chosen so that $\vol(A_G\backslash T)=1$ (see \S \ref{groups, measures, notations}). By a similar argument as in Lemma \ref{convergence of geometric side}, we know that the integrals on the right hand side are absolutely convergent. Moreover, if $\mathcal{A}=\mathcal{D}$ is a division algebra, then the same formula holds for all $\pi\in \Irr(G,\chi)$.
\end{prop}

\vspace{2mm}

\noindent\ul{Proof}: The case when $\mathcal{A}$ is a division algebra directly follows from Corollary \ref{cor appendix}. Assume now that $\pi\in \Pi_2(G,\chi)$. The absolute value $\lvert \chi \rvert$ of $\chi$ extends uniquely to a positive valued character on $G$ that we shall denote the same way. Then, up to multiplying $\omega$ by $\lvert \chi\rvert^{-1}$ and replacing $\pi$ by $\pi\otimes \lvert \chi\rvert^{-\frac{1}{2n}}$ we may assume that $\chi$ is unitary and then so is $\omega$. We can then use the equality of Theorem \ref{theo trace formula} which may be rewritten as

$$\displaystyle \sum_{T\in \mathcal{T}_{ell}(H_0)} \lvert W(H_0,T)\rvert^{-1} \int_{A_G\backslash T} D^{H}(t) c_{f_\chi}(t) \omega(t)^{-1}dt=\sum_{\pi'\in \Pi_2(G,\chi)} m(\pi',\omega) \Tr (\pi')^\vee(f_\chi)$$

\noindent for all $f\in {}^0\mathcal{C}(G)$. Let $\pi\in \Pi_2(G,\chi)$. We choose $f\in {}^0\mathcal{C}(G)$ so that $f_\chi$ is a matrix coefficient of $\pi$ with $f_\chi(1)\neq 0$. Then by Schur's orthogonality relations, the spectral side reduces to

$$\displaystyle m(\pi,\omega) \Tr \pi^\vee(f_\chi)=d(\pi)^{-1}m(\pi,\omega) f_\chi(1).$$

\noindent On the other hand, by \ref{cusp forms and theta-strongly cuspidal functions}(2), the geometric side equals

$$\displaystyle d(\pi)^{-1}f_\chi(1)\sum_{T\in \mathcal{T}_{ell}(H_0)} \lvert W(H_0,T)\rvert^{-1} \int_{A_G\backslash T} D^{H}(t) c_{\pi}(t) \omega(t)^{-1}dt.$$

\noindent This proves the proposition. $\blacksquare$

\begin{rmk}\label{multiplicity formula fails}
In general, if $\CA$ is not a division algebra, then the multiplicity formula will not holds for all tempered (or generic) representations. For example, let $\CA=\Mat_{m\times m}(D)$ with $m>1$, and let $\pi$ be a tempered representation of $G=\GL_{2m}(D)$ with central character $\chi$. Assume that $\pi$ is the parabolic induction of some discrete series $\tau=\tau_1\otimes \cdots\otimes \tau_{2m}$ of the minimal parabolic subgroup $P_0=M_0N_0$ of $G$ (here $M_0=(\GL_1(D))^{2m}$). By Lemma 2.3 of \cite{WalGGPII}, the right hand side of the multiplicity formula is always equal to zero. On the mean time, we can choose some nice $\tau$ such that $m(\pi)\neq 0$ (e.g. when the character $\omega$ is trivial, we just need to let $\tau=\tau_1\otimes \cdots\otimes \tau_{2m}$ with $\tau_{2i-1}\simeq \tau_{2i}^{\vee}$ for $1\leq i\leq m$).
\end{rmk}

\section{The spectral side}\label{The spectral side}

In this section we prove, following \cite{BeuGalP} \S 3, the absolute convergence as well as the spectral side of Theorem \ref{theo trace formula}.

For $\pi\in \Pi_2(G,\chi)$, let

$$\displaystyle \mathcal{B}_\pi:\pi\times \pi^\vee\to \mathbf{C}$$

\noindent be the bilinear form defined by

$$\displaystyle \mathcal{B}_\pi(v,v^\vee):=\int_{A_G\backslash H} \langle \pi(h)v,v^\vee\rangle (\omega\otimes \xi)(h)^{-1}dh$$

\noindent for all $(v,v^\vee)\in \pi\times \pi^\vee$. Note that the integral above is always absolutely convergent by \ref{Shalika triples}(1). Obviously $\mathcal{B}_\pi$ descents to a bilinear pairing

$$\displaystyle \mathcal{B}_\pi: \pi_{\omega\otimes\xi} \times \pi^\vee_{(\omega\otimes \xi)^{-1}}\to \mathbf{C}$$

\noindent where $\pi_{\omega\otimes\xi}$ and $\pi^\vee_{(\omega\otimes \xi)^{-1}}$ denote the $(H,\omega\otimes \xi)$- and $(H,(\omega\otimes \xi)^{-1})$-coinvariant spaces of $\pi$ and $\pi^\vee$ respectively. As in \cite{BeuGalP} \S 4 the main ingredient of the proof is the following proposition which is a variation of \cite{SV} Theorem 6.4.1 (a similar idea also appears in \cite{WalGGPII} Proposition 5.6):

\begin{prop}\label{prop intertwinings}
$\mathcal{B}_\pi$ induces a perfect pairing between $\pi_{\omega\otimes\xi}$ and $\pi^\vee_{(\omega\otimes \xi)^{-1}}$.
\end{prop}

\noindent This proposition can be proved by the exactly the same way as \cite{BeuGalP} Proposition 3.21 once we establish the next lemma.

\begin{lem}\label{lemma intertwinings}
For all $\ell\in Hom_H(\pi,\omega\otimes \xi)$ and all $v\in \pi$, we have

$$\displaystyle \int_{H\backslash G} \lvert \ell(\pi(x)v)\rvert^2 dx<\infty.$$

\noindent Moreover, for all $f\in \mathcal{C}(G)$, the integral

$$\displaystyle \int_{G} f(g) \ell(\pi(g)v) dg$$

\noindent is absolutely convergent and equals $\ell(\pi(f)v)$.
\end{lem}

\vspace{2mm}

\noindent Proof of Lemma \ref{lemma intertwinings}: Let $\ell\in Hom_H(\pi,\omega\otimes \xi)$ and $v\in \pi$. Set $G_1:=\mathcal{A}^\times$ and embed $G_1$ in $G$ via $g\mapsto \begin{pmatrix} g & \\ & 1\end{pmatrix}$. Let $P_1=L_1U_1\subset G_1$ be a minimal parabolic subgroup and $A_1\subset L_1$ the maximal split torus. Let $\Delta_1$ be the set of simple roots of $A_1$ in $P_1$ and set

$$\displaystyle A_1^+:=\left\{a\in A_1; \lvert \alpha(a)\rvert\leqslant 1 \;\forall \alpha\in \Delta_1 \right\},$$

$$\displaystyle A_1^{++}:=\left\{a\in A_1^+;\;  \lvert \alpha(a)\rvert\leqslant 1 \;\forall \alpha\in R(A_1,N)\right\}.$$

\noindent Then, by the Iwasawa decomposition $G=QK$ and Cartan decomposition for $G_1$, there exists a compact subset $C_0\subset G$ such that

$$\displaystyle G=HA_1^+C_0. \leqno (1)$$

\noindent Moreover, there exists a compact subset $C_A \subset A_1$ such that for all $a\in A_1^+$ with $\ell(\pi(a)v)\neq 0$, we have

$$\displaystyle a\in A_1^{++} C_A.$$

\noindent Indeed, it suffices to show that for all $\alpha\in R(A_1,N)$, there exists $c_\alpha>0$ such that for all $a\in A_1$ with $\ell(\pi(a)v)\neq 0$, we have $\lvert \alpha(a)\rvert\leqslant c_\alpha$. Let $\alpha$ be such a root and note that the character $\xi$ has a nontrivial restriction to the corresponding root subspace $N_\alpha$. Let $K_\alpha\subset N_\alpha$ be a compact-open subgroup which leaves $v$ invariant. Then, there exists $c_\alpha>0$ such that for all $a\in A_1$ with $\lvert \alpha(a)\rvert>c_\alpha$, the restriction of $\xi$ to $aK_{\alpha}a^{-1}$ is nontrivial which easily implies that $\ell(\pi(a)v)=0$. This proves the claim.

\vspace{2mm}

\noindent Thus there exists a compact-open subset $C\subset G$ such that

$$\displaystyle \mbox{The function } g\in G\mapsto \ell(\pi(g)v) \mbox{ has support in } HA_1^{++}C. \leqno (2)$$

\noindent Let $\overline{P}_1=L_1\overline{U}_1$ be the parabolic subgroup opposite to $P_1$ and introduce the following parabolic subgroup of $G$:

$$\displaystyle P:=\left\{\begin{pmatrix} p_1 & 0 \\ X & g_1 \end{pmatrix};\; p_1\in \overline{P}_1, g_1\in G_1, X\in \mathcal{A} \right\}.$$

\noindent It has a Levi decomposition $P=L_PU_P$ where

$$\displaystyle L_P:=\left\{\begin{pmatrix} l_1 & 0 \\ 0 & g_1 \end{pmatrix};\; l_1\in L_1, g_1\in G_1 \right\}$$

\noindent and

$$\displaystyle U_P:=\left\{\begin{pmatrix} u_1 & 0 \\ X & 1 \end{pmatrix};\; u_1\in \overline{U}_1, X\in \mathcal{A} \right\}.$$

\noindent Note that $A_1$ is contained in the center of $L_P$ and

$$\displaystyle A_1^{++}=\left\{a\in A_1;\; \lvert \alpha(a)\rvert \geqslant 1\; \forall \alpha\in R(A_1,U_P) \right\}. \leqno (3)$$

\noindent Moreover, we have

$$\displaystyle HP \mbox{ is open in } G. \leqno (4)$$

\noindent Define a function $\Xi^{H\backslash G}$ on $H\backslash G$ by

$$\displaystyle \Xi^{H\backslash G}(x):=\vol_{H\backslash G}(xC)^{-1/2},\;\;\; x\in H\backslash G.$$

\noindent Then as in Proposition 6.7.1 of \cite{BeuGGP}, we can show in turn that

\begin{enumerate}[(1)]
\setcounter{enumi}{4}
\item $\Xi^{H\backslash G}(xk)\sim \Xi^{H\backslash G}(x)$ and $\sigma_{H\backslash G}(xk)\sim \sigma_{H\backslash G}(x)$ for all $x\in H\backslash G$ and all $k\in C$;

\item There exists $d>0$ such that $\Xi^G(a)\ll \Xi^{H\backslash G}(a)\sigma_G(a)^d$ for all $a\in A_1^{++}$ (this uses (3) and (4));

\item $\sigma_{H\backslash G}(a)\sim \sigma_G(a)$ for all $a\in A_1$ (this is because the regular map $G_1\to H\backslash G$ is a closed embedding);

\item There exists $d>0$ such that the integral

$$\displaystyle \int_{H\backslash G} \Xi^{H\backslash G}(x)^2 \sigma_{H\backslash G}(x)^{-d}dx$$

\noindent converges (this uses decomposition (1));

\item For all $d>0$, there exists $d'>0$ such that

$$\displaystyle \int_H \Xi^G(hx)\sigma_G(hx)^{-d'}dx\ll \Xi^{H\backslash G}(x) \sigma_{H\backslash G}(x)^{-d}$$

\noindent for all $x\in H\backslash G$ (this uses (4) together with decomposition (1)).
\end{enumerate}

\noindent Finally, by the above points and (2), in order to prove the lemma, it remains to show that

$$\displaystyle \mbox{For all } d>0, \mbox{ we have } \lvert \ell(\pi(a)v)\rvert \ll \Xi^G(a)\sigma_G(a)^{-d} \mbox{ for all } a\in A_1^{++}. \leqno (10)$$

\noindent Let $K_v\subset G$ be an open subgroup which stabilizes $v$. Then, from (3) and (4) we deduce the existence of a compact-open subgroup $J$ of $G$ such that $J\subset HaK_va^{-1}$ for all $a\in A_1^{++}$. It follows that

$$\displaystyle \ell(\pi(a)v)=\langle \pi(a)v,e_J\ast \ell\rangle$$

\noindent for all $a\in A_1^{++}$ where $e_J\ast \ell$ denotes the element of $\pi^\vee$ defined by

$$\displaystyle \langle w,e_J\ast \ell\rangle:=\vol(J)^{-1}\int_J \ell(\pi(k)w)dk,\;\;\; w\in \pi.$$

\noindent (10) now follows from standard estimates for coefficients of square-integrable representations. $\blacksquare$

\vspace{2mm}

We now prove the absolute convergence and the spectral side of Theorem \ref{theo trace formula}. Let $f\in {}^0\mathcal{C}(G)$. Then, we have

$$\displaystyle K_f(x,y)=\int_{A_G\backslash H} f_\chi(x^{-1}hy) (\omega\otimes \xi)(h)^{-1}dh$$

\noindent for all $x,y\in G$ where we recall that $\chi$ denotes the restriction of $\omega$ to $A_G$ and $f_\chi(g):=\int_{A_G} f(ag) \chi(a)^{-1} da$ for all $g\in G$. By \ref{cusp forms and theta-strongly cuspidal functions}(3), we may assume that there exists $\pi\in \Pi_2(G,\chi)$ such that $f_\chi$ is a matrix coefficient of $\pi$, i.e. there exist $(v,v^\vee)\in \pi\times \pi^\vee$ such that $f_\chi(g)=\langle \pi(g)v,v^\vee\rangle$ for all $g\in G$. In this case, we have

$$\displaystyle K_f(x,y)=\mathcal{B}_\pi(\pi(y)v,\pi(x)v^\vee)$$

\noindent for all $x,y\in G$. Let $N:=m(\pi,\omega)$ and let $v_1,\ldots, v_N$ be vectors in $\pi$ whose images in $\pi_{\omega\otimes \xi}$ form a basis. Let $v_1^\vee,\ldots,v_N^\vee$ be vectors in $\pi^\vee$ whose images in $\pi^\vee_{(\omega\otimes \xi)^{-1}}$ form the dual basis with respect to $\mathcal{B}_\pi$. Then, we have

$$\displaystyle K_f(x,y)=\sum_{i=1}^N \mathcal{B}_\pi(\pi(y)v,v_i^\vee)\mathcal{B}_\pi(v_i,\pi(x)v^\vee)$$

\noindent for all $x,y\in G$. From there and Lemma \ref{lemma intertwinings} we deduce the absolute convergence of $J(f)$. Now the rest part of the proof is the same as that of \cite{BeuGalP} Theorem 3.11: using Schur's orthogonality relations we have

\[\begin{aligned}
\displaystyle J(f) & =\sum_{i=1}^N \int_{H\backslash G} \mathcal{B}_\pi(\pi(x)v,v_i^\vee)\mathcal{B}_\pi(v_i,\pi^\vee(x)v^\vee) dx=\sum_{i=1}^N \int_{A_G\backslash G}\langle \pi(g)v,v_i^\vee\rangle \mathcal{B}_\pi(v_i,\pi^\vee(g)v^\vee) dg \\
 & =\sum_{i=1}^N \frac{\langle v,v^\vee\rangle}{d(\pi)} \mathcal{B}_\pi(v_i,v_i^\vee) =m(\pi,\omega) \frac{\langle v,v^\vee\rangle}{d(\pi)} =m(\pi,\omega) \Tr \pi(f).
\end{aligned}\]

$\blacksquare$

\section{The geometric side}\label{geom side}

The goal of this chapter is to prove the geometric side of Theorem \ref{theo trace formula}. This proof will be given in Section \ref{proof of the geometric side}. We will continue to use the notations introduced in Chapter \ref{A multiplicity formula for Shalika models} and we will assume as in Theorem \ref{theo trace formula} that $\omega$ is unitary. We will also need the following extra notations:

\begin{itemize}
\item $\theta:=Ad\begin{pmatrix} & 1 \\1 & \end{pmatrix}$ (an involution of $G$);

\item $G_1:=\mathcal{A}^\times$, $K_1:=\mathcal{O}_{\mathcal{A}}^\times$ and $\mathfrak{g}_1:=\mathcal{A}$ (the Lie algebra of $G_1$), note that we have a natural identification $L\simeq G_1\times G_1$;

\item There is a natural open embedding $G_1\hookrightarrow \mathfrak{g}_1$ and we will denote by $\mathfrak{g}_1^*$ its image. Similarly, for any maximal torus $T$ of $G_1$ the previous embedding restricts to an open embedding $T\hookrightarrow \mathfrak{t}$, where $\mathfrak{t}$ denotes the Lie algebra of $T$, and we will denote by $\mathfrak{t}^*$ its image (i.e. the open subset of all $X\in \mathfrak{t}$ such that $\nu(X)\neq 0$);

\item $\langle .,.\rangle:\mathfrak{g}_1\times \mathfrak{g}_1\to F$ the bilinear pairing given by $\langle X,Y\rangle:=\Tr_{\mathcal{A}/F}(XY)$;
\end{itemize}

Recall that in this paper we are assuming that every algebraic variety $X$ over $F$ that we encounter has been equipped with a log-norm $\sigma_X$. For simplicity we will assume, as we may, that $\sigma_{G_1}$ and $\sigma_{\mathfrak{g}_1^*}$ are both left and right invariant by $K_1$. Also, by Proposition 18.3 of \cite{KottHA}, we may, and will, assume that for all maximal torus $T$ of $G_1$ and all $g\in G_1$, we have

$$\displaystyle \sigma_{T\backslash G_1}(g)=\inf_{t\in T}\sigma_{G_1}(tg). \leqno (1)$$

\noindent Let $T\subset G$ be a maximal torus, then we recall the following inequality from \cite{BeuGalP}1.2(2):

$$\displaystyle \sigma_{T\backslash G}(g)\ll \sigma_G(g^{-1}tg)\log\left(2+D^G(t)^{-1}\right) \leqno (2)$$

\noindent for all $g\in G$ and all $t\in T\cap G_{reg}$. We will also need the following easy-to-check lemma:

\begin{lem}\label{lem basic estimate}
Let $K$ be a finite extension of $F$ and set $v_K:=v_F\circ N_{K/F}$. Then for all $k>0$, the inequality

$$\displaystyle \int_{y\in K;\; v_K(y)\geqslant C} \max(1,v_K(y))^k dy\ll C^k\times vol\{ y\in K;\; v_K(y)\geqslant C\}$$

\noindent holds for all $C>0$.
\end{lem}

We now describe roughly how we will prove the geometric side of Theorem \ref{theo trace formula}. In Section \ref{def of an expression} we will introduce a sequence of truncations $(J_N(f))_{N\geqslant 1}$ of $J(f)$ such that $\lim\limits_{N\to \infty}J_N(f)=J(f)$ whenever $J(f)$ is absolutely convergent. Then, in Section \ref{computation of the limit} we will show that $J_N(f)$ admits a limit whenever $f$ is {\it $\theta$-strongly cuspidal} (see \S \ref{cusp forms and theta-strongly cuspidal functions}) and we compute this limit. In the particular case where $f$ is strongly cuspidal (in particular if $f$ is a cusp form), we prove in Section \ref{proof of the geometric side} that this limit is equal to the geometric side of Theorem \ref{theo trace formula}. The bulk of the proof is contained in Section \ref{change of weight}, where we show that we can replace certain weights appearing naturally from our truncations by other weights that are related to certain (singular) $\theta$-weighted orbital integrals.

\subsection{Definition of a truncation}\label{def of an expression}

Fix a maximal split torus $A_1$ of $G_1$ such that $K_1$ is the fixator of a special point in the apartment associated to $A_1$. Let $M_1:=Cent_{G_1}(A_1)$ and $P_1$ be a minimal parabolic subgroup with Levi component $M_1$. Let $\Delta_1$ denotes the set of simple roots of $A_1$ in $P_1$. We have a Cartan decomposition

$$\displaystyle G_1=K_1 M_1^+K_1$$

\noindent with $M_1^+:=\{m\in M_1; \langle \alpha, H_{M_1}(m)\rangle\geqslant 0 \; \forall \alpha\in \Delta_1 \}$. Let $N\geqslant 1$ be an integer and let $T_N\in \mathcal{A}_{M_1,F}^{G_1}$ be the point characterized by $\langle \alpha, T_N\rangle=N$ for all $\alpha\in \Delta_1$. In \cite{ArthurlocalTF}\S 3, Arthur has defined a certain characteristic function $u(.,T_N)$ associated to $T_N$. More precisely, if we denote for all $\alpha\in \Delta_1$ by $\varpi_\alpha\in \mathcal{A}_{M_1}^*$ the corresponding simple weight and set

$$\displaystyle M_1^+(N):=\{m\in M_1^+;\; \langle H_{M_1}(m)-T_N,\varpi_\alpha\rangle\leqslant 0 \; \forall \alpha\in \Delta_1  \}.$$

\noindent Then $u(.,T_N)$ is the characteristic function of $K_1M_1^+(N)K_1$ (see \cite{ArthurlocalTF} Lemma 3.1). Because of the center, the set $K_1M_1^+(N)K_1$ is not compact and we define another truncation function $\kappa_N:G_1\to \{0,1 \}$, this time of compact support, by setting

$$\displaystyle \kappa_N(g)=\mathbf{1}_{[q^{-N},q^N]}(\nu(g))u(g,T_N),\;\;\; g\in G_1$$

\noindent where $\mathbf{1}_{[q^{-N},q^N]}$ stands for the characteristic function of the segment $[q^{-N},q^N]$. The next lemma summarizes some basic properties of the sequence $(\kappa_N)_{N\geqslant 1}$ that we will need.

\begin{lem}\label{lemma Arthur}
\begin{enumerate}[(i)]
\item There exist $c_1,c_2>0$ such that for all $g\in G_1$ and all $N\geqslant 1$, if $\sigma_G(g)\leqslant c_1 N$, then $\kappa_N(g)=1$; and if $\kappa_N(g)=1$, then $\sigma_G(g)\leqslant c_2 N$.

\item Let $T\subset G_1$ be a maximal torus. Then, there exist $c>0$ and $N_0\geqslant 1$ such that for all $N\geqslant N_0$ and all $g,h\in G_1$ with $\max\left(\sigma_{T\backslash G_1}(g),\sigma_{T\backslash G_1}(h)\right)\leqslant cN$, the function

$$\displaystyle a\in T\mapsto \kappa_N(h^{-1}ag)$$

\noindent is invariant by the maximal compact subgroup $T^c$ of $T$.
\end{enumerate}
\end{lem}

\noindent\ul{Proof}:
\begin{enumerate}[(i)]
\item is obvious;
\item Up to conjugating $T$ we may assume that $T^c\subset K_1$ and $A_T\subset A_1$. Since the property we want to prove is invariant by left translation of both $g$ and $h$ by $T$, by \ref{geom side}(1) we may replace the condition $\max\left(\sigma_{T\backslash G_1}(g),\sigma_{T\backslash G_1}(h)\right)\leqslant cN$ by the seemingly stronger one $\max\left(\sigma_{G_1}(g),\sigma_{G_1}(h)\right)\leqslant cN$. Also, there exists a finite family $(t_i)_{i=1,\ldots,k}$ of elements of $T$ such that $\displaystyle T=\bigsqcup_{i=1}^k A_TT^ct_i$. Hence, it suffices to show the existence of $c>0$ and $N_0\geqslant 1$ such that for all $N\geqslant N_0$ and all $g,h\in G_1$ with $\max\left(\sigma_{G_1}(g),\sigma_{G_1}(h)\right)\leqslant cN$, we have

$$\displaystyle \kappa_N(h^{-1}tag)=\kappa_N(h^{-1}ag)$$

\noindent for all $t\in T^c$ and all $a\in A_T$. This property is obviously satisfied by the function $g\mapsto \mathbf{1}_{[q^{-N},q^N]}(\nu(g))$ and thus we may replace $\kappa_N$ by $u(.,T_N)$. Set $M:=Cent_{G_1}(A_T)$ and define the maps $H_P:G_1\to \mathcal{A}_M$, $P\in \mathcal{P}(M)$, by using the maximal compact subgroup $K_1$. These maps are $T^c$-invariant on the left. As $\sigma_{G_1}$ is also $T^c$-invariant on the left (since we are assuming $T^c\subset K_1$), it suffices to show that for some $c>0$ and $N_0\geqslant 1$, the following holds: for all $N\geqslant N_0$ and $g,h\in G_1$ with $\max\left(\sigma_{G_1}(g),\sigma_{G_1}(h)\right)\leqslant cN$, the function

$$\displaystyle a\in A_T\mapsto u(h^{-1}ag,T_N)$$

\noindent only depends on the families $(H_P(h))_{P\in \mathcal{P}(M)}$ and $(H_P(g))_{P\in \mathcal{P}(M)}$. Such a property is provided by the proof of Lemma 4.4 of \cite{ArthurlocalTF}. Indeed, by the equation on the last line of p.38 of {\it loc. cit.}, there exists $c>0$ and $N_0\geqslant 1$ such that for all $N$, $g$ and $h$ as before, we have

$$\displaystyle u(h^{-1}ag,T_N)=\Gamma_M^{G_1}(H_M(a),\mathcal{Y}_M(h,g,T_N))$$

\noindent for all $a\in A_T(F)$ and where the $(G_1,M)$-orthogonal set $\mathcal{Y}_M(h,g,T_N)$ is defined by $\mathcal{Y}_M(h,g,T_N)_P:=T_{N,P}+H_P(h)-H_{\overline{P}}(g)$, $P\in \mathcal{P}(M)$. Obviously this $(G_1,M)$-orthogonal set only depends on the sets $(H_P(h))_{P\in \mathcal{P}(M)}$ and $(H_P(g))_{P\in \mathcal{P}(M)}$ and this proves the claim. $\blacksquare$
\end{enumerate}

\vspace{2mm}

\noindent Choose Haar measures on $G_1$, $\mathfrak{g}_1$ and $K$ such that

$$\displaystyle \int_G \varphi(g)dg=\int_{\mathfrak{g}_1}\int_{G_1\times G_1}\int_{K_1} \varphi\left(\begin{pmatrix} 1 & X \\ & 1 \end{pmatrix}\begin{pmatrix} g & \\ & h \end{pmatrix} k  \right)dk \frac{dg}{\nu(g)^n} \nu(h)^n dh dX$$

\noindent and

$$\displaystyle \int_H \varphi(h)dh=\int_{G_1}\int_{\mathfrak{g}_1}\varphi\left(\begin{pmatrix} h_0 & \\ & h_0 \end{pmatrix} \begin{pmatrix} 1 & X \\ & 1\end{pmatrix} \right) dX dh_0$$

\noindent for all $\varphi\in C_c(G)$ (resp. for all $\varphi\in C_c(H)$).

\vspace{2mm}

\noindent For every integer $N\geqslant 1$ and every function $f\in \mathcal{C}(G)$, we define the following expression:

$$\displaystyle J_N(f):=\int_{G_1}\int_{G_1} \int_{\mathfrak{g}_1} f^K\left( \begin{pmatrix} g^{-1} & \\ & 1 \end{pmatrix} \begin{pmatrix} h & \\ & h \end{pmatrix} \begin{pmatrix} 1 & X \\ & 1\end{pmatrix}\begin{pmatrix} g & \\ & 1 \end{pmatrix}\right) \xi(X)^{-1} dX \omega(h)^{-1} dh \kappa_N(g)\frac{dg}{\nu(g)^n}$$

\noindent where we have set

$$\displaystyle f^K(g):=\int_K f(k^{-1}gk)dk$$

\noindent for all $g\in G$. Note that if the integral defining $J(f)$ (see \S \ref{A simple local trace formula for Shalika models}) is absolutely convergent, then we have

$$\displaystyle J(f)=\lim\limits_{N\to \infty} J_N(f).$$

\noindent By Weyl's integration formula, we have

$$\displaystyle J_N(f)=\sum_{T\in \mathcal{T}(G_1)} \lvert W(G_1,T)\rvert^{-1} J_{N,T}(f) \leqno (1)$$

\noindent for all $N\geqslant 1$ and all $f\in \mathcal{C}(G)$ where $\mathcal{T}(G_1)$ is a set of representatives of the $G_1$-conjugacy classes of maximal tori in $G_1$, and for all $T\in \mathcal{T}(G_1)$ we have set

\[\begin{aligned}
\displaystyle J_{N,T}(f):=\int_{T} D^{G_1}(t)\int_{T\times T\backslash G_1\times G_1} \int_{\mathfrak{g}_1} & f^K\left(\begin{pmatrix} g^{-1} & \\ & h^{-1} \end{pmatrix}\begin{pmatrix} t & \\ & t \end{pmatrix} \begin{pmatrix} 1 & X \\ & 1 \end{pmatrix}\begin{pmatrix} g & \\ & h \end{pmatrix}\right) \\
 & \kappa_{N,T,\xi}(g,h,X) dX \frac{dg}{\nu(g)^n}\nu(h)^n dh \omega(t)^{-1} dt
\end{aligned}\]

\noindent with

$$\displaystyle \kappa_{N,T,\xi}(g,h,X):=\int_{T} \xi(aX)^{-1} \kappa_N(h^{-1}ag) da.$$

\noindent From now on and until the end of section \ref{computation of the limit} we fix a torus $T\in \mathcal{T}(G_1)$. Without loss of generality (i.e. up to conjugating $T$ by an element of $G_1$), we may assume that $A_T\subset A_1$ and $T\cap K_1=T^c$. From Lemma \ref{lemma Arthur}(i) and \ref{geom side}(1) we easily infer that there exists $k>0$ such that the following estimate

$$\displaystyle \lvert \kappa_{N,T,\xi}(g,h,X)\rvert\ll \sigma_{T\backslash G_1}(g)^k\sigma_{T\backslash G_1}(h)^kN^k \leqno (2)$$

\noindent holds for all $N\geqslant 1$, all $g,h\in G_1$ and all $X\in \mathfrak{g}_1$. More precisely, there exists $k>0$ such that

$$\displaystyle \int_T \kappa_N(h^{-1}ag)da\ll \sigma_{T\backslash G_1}(g)^k\sigma_{T\backslash G_1}(h)^kN^k \leqno (3)$$

\noindent for all $N\geqslant 1$ and all $g,h\in G_1$.

\subsection{Concrete description of $T$ and of related objects}\label{concrete T}

Set $M_\natural:=Cent_{G_1}(A_T)$ (a Levi subgroup of $G_1$), $M:=M_\natural \times M_\natural$ (a Levi subgroup of $G$). Then, the Levi $M$ is $\theta$-split. Moreover, there exist field extensions $K_1,\ldots, K_d$ of $F$ such that

$$\displaystyle T\simeq K_1^\times\times\ldots\times K_d^\times . \leqno (1)$$

\noindent More precisely, there exist:

\begin{itemize}
\item A division algebra $\mathcal{D}$, central over $F$ and of degree $r$ dividing $n$;
\item A right $\mathcal{D}$-module $V$ free of rank $n/r$;
\item An isomorphism of $F$-algebras

$$\displaystyle \mathcal{A}\simeq End_{\mathcal{D}}(V)$$

\noindent inducing an isomorphism

$$\displaystyle G_1\simeq GL_{\mathcal{D}}(V).$$

\item For all $1\leqslant i\leqslant d$, a degree $r$ subextension $F_i$ of $K_i$ together with an embedding $F_i\hookrightarrow \mathcal{D}$ and an isomorphism of (right) $\mathcal{D}$-modules

$$\displaystyle V\simeq K_1\otimes_{F_1}\mathcal{D}\oplus\ldots\oplus K_d\otimes_{F_d}\mathcal{D} \leqno (2)$$

\noindent through which the action of $T$ is given by multiplication by $K_i^\times$ on the $i$-th factor.
\end{itemize}

\noindent We fix such data (and isomorphisms) once and for all and we fix a basis of $V$ compatible with the decomposition (2). Doing so we will identify $G_1$ (resp. $G$) with $GL_{n/r}(\mathcal{D})$ (resp. $GL_{2n/r}(\mathcal{D})$). Besides (1), we also get identifications

$$\displaystyle \mathfrak{t}\simeq K_1\oplus\ldots \oplus K_d \leqno (3)$$

\noindent and

$$\displaystyle M\simeq GL_{k_1}(\mathcal{D})\times\ldots\times GL_{k_d}(\mathcal{D})\times GL_{k_1}(\mathcal{D})\times\ldots\times GL_{k_d}(\mathcal{D}) \leqno (4)$$

\noindent where we have set $k_i:=[K_i:F_i]$ for all $1\leqslant i\leqslant d$. We will denote by $\mathfrak{t}^\perp$ the orthogonal of $\mathfrak{t}$ in $\Fg_1$ with respect to the symmetric bilinear form $\langle .,.\rangle$ and by $X\mapsto X_{\mathfrak{t}}$ (resp. $X\mapsto X_{\mathfrak{t}^\perp}$) the projection $\mathfrak{g}_1\to \mathfrak{t}$ (resp. $\mathfrak{g}_1\to \mathfrak{t}^\perp$) relative to the decomposition $\mathfrak{g}_1=\mathfrak{t}\oplus \mathfrak{t}^\perp$.

\subsection{An estimate}\label{an estimate}

For all $f\in \mathcal{C}(G)$, all $t_1,t_2\in T\cap G_{1,reg}$ and all $k>0$, we define the following expressions:

\[\begin{aligned}
\displaystyle I_k(f,t_1,t_2):= & \int_{T\times T\backslash G_1\times G_1} \int_{\mathfrak{g}_1} f\left(\begin{pmatrix} g^{-1} & \\ & h^{-1} \end{pmatrix} \begin{pmatrix}t_1 & \\ & t_2 \end{pmatrix} \begin{pmatrix} 1 & X \\ & 1\end{pmatrix} \begin{pmatrix} g & \\ & h \end{pmatrix} \right) \\
 & \sigma_{T\backslash G_1}(g)^k \sigma_{T\backslash G_1}(h)^k \sigma_{\mathfrak{g}_1^*}(g^{-1}X_{\mathfrak{t}}h)^k \sigma_{\mathfrak{g}_1}(g^{-1}X_{\mathfrak{t}^\perp}h)^k dX \frac{dg}{\nu(g)^n} \nu(h)^n dh
\end{aligned}\]

\noindent and

$$\displaystyle I_{T,k}(f):=\int_{T} D^{G_1}(t) I_k(f,t,t)dt$$

\noindent where we recall that for all $X\in \mathfrak{g}_1$, $X_{\mathfrak{t}}$ and $X_{\mathfrak{t}^\perp}$ stand for the projections of $X$ onto $\mathfrak{t}$ and $\mathfrak{t}^\perp$ respectively with respect to the decomposition $\mathfrak{g}_1=\mathfrak{t}\oplus \mathfrak{t}^\perp$. Similarly, for all $f\in \mathcal{C}(G)$, all $t_1,t_2\in T\cap G_{1,reg}$, all $k>0$ and all $C>0$, we define the following expressions:

\[\begin{aligned}
\displaystyle I_{k,\leqslant C}(f,t_1,t_2):= & \int_{T\times T\backslash G_1\times G_1} \int_{\mathfrak{g}_1} f\left(\begin{pmatrix} g^{-1} & \\ & h^{-1} \end{pmatrix} \begin{pmatrix}t_1 & \\ & t_2 \end{pmatrix} \begin{pmatrix} 1 & X \\ & 1\end{pmatrix} \begin{pmatrix} g & \\ & h \end{pmatrix} \right) \\
 & \sigma_{T\backslash G_1}(g)^k \sigma_{T\backslash G_1}(h)^k \sigma_{\mathfrak{g}_1^*}(g^{-1}X_{\mathfrak{t}}h)^k \sigma_{\mathfrak{g}_1}(g^{-1}X_{\mathfrak{t}^\perp}h)^k \\
 & \mathbf{1}_{T\backslash G_1,\leqslant C}(g)\mathbf{1}_{T\backslash G_1,\leqslant C}(h)\mathbf{1}_{\mathfrak{g}_1^*,\leqslant C}(g^{-1}X_{\mathfrak{t}}h)\mathbf{1}_{\mathfrak{g}_1,\leqslant C}(g^{-1}X_{\mathfrak{t}^\perp}h)dX \frac{dg}{\nu(g)^n} \nu(h)^n dh
\end{aligned}\]

\noindent and

$$\displaystyle I_{T,k,\leqslant C}(f):=\int_{T} D^{G_1}(t) I_{k,\leqslant C}(f,t,t)dt.$$

\begin{prop}\label{prop estimate}
\begin{enumerate}[(i)]
\item For all $f\in \mathcal{C}(G)$, all $t_1,t_2\in T\cap G_{1,reg}$ and all $k>0$, the expressions defining $I_k(f,t_1,t_2)$ and $I_{T,k}(f)$ are absolutely convergent.

\item For all $f\in \mathcal{C}(G)$, all $t_1,t_2\in T\cap G_{1,reg}$, all $k>0$ and all $r>0$, we have inequalities

$$\displaystyle \left\lvert I_k(f,t_1,t_2)-I_{k,\leqslant C}(f,t_1,t_2)\right\rvert\ll C^{-r}$$

\noindent and

$$\displaystyle \left\lvert I_{T,k}(f)-I_{T,k,\leqslant C}(f)\right\rvert\ll C^{-r}$$

\noindent for all $C>0$.
\end{enumerate}
\end{prop}

\vspace{2mm}

\noindent\ul{Proof}:
\begin{enumerate}[(i)]
\item Fix $f\in \mathcal{C}(G)$, $t_1,t_2\in T\cap G_{1,reg}$ and $k>0$. Up to replacing $f$ by its absolute value, we will assume that $f$ only takes nonnegative values. By Corollary 2 of \cite{Clo}, to prove the absolute convergence of $I_k(f,t_1,t_2)$, it suffices to prove that for all $r>0$, we have

\[\begin{aligned}
\displaystyle \left(\frac{\nu(h)}{\nu(g)}\right)^n & \int_{\mathfrak{g}_1} f\left(\begin{pmatrix} g^{-1} & \\ & h^{-1} \end{pmatrix} \begin{pmatrix}t_1 & \\ & t_2 \end{pmatrix} \begin{pmatrix} 1 & X \\ & 1\end{pmatrix} \begin{pmatrix} g & \\ & h \end{pmatrix} \right) \sigma_{T\backslash G_1}(g)^k \sigma_{T\backslash G_1}(h)^k \\
 & \sigma_{\mathfrak{g}_1^*}(g^{-1}X_{\mathfrak{t}}h)^k \sigma_{\mathfrak{g}_1}(g^{-1}X_{\mathfrak{t}^\perp}h)^k dX \ll \Xi^{G_1}(g^{-1}t_1g)\Xi^{G_1}(h^{-1}t_2h)\sigma_{G_1}(g^{-1}t_1g)^{-r}\sigma_{G_1}(h^{-1}t_2h)^{-r}
\end{aligned}\]

\noindent for all $g,h\in G_1$. Since $\sigma_{T\backslash G_1}(g)\ll \sigma_{G_1}(g^{-1}t_1g)$ and $\sigma_{T\backslash G_1}(h)\ll \sigma_{G_1}(h^{-1}t_2h)$ for all $g,h\in G_1$ (this is because $g\in T\backslash G_1\mapsto g^{-1}t_1g\in G_1$ is a finite morphism and similarly for $t_2$), we are reduced to prove that for all $r>0$, we have

\[\begin{aligned}
\displaystyle \left(\frac{\nu(h)}{\nu(g)}\right)^n & \int_{\mathfrak{g}_1} f\left(\begin{pmatrix} g^{-1} & \\ & h^{-1} \end{pmatrix} \begin{pmatrix}t_1 & \\ & t_2 \end{pmatrix} \begin{pmatrix} 1 & X \\ & 1\end{pmatrix} \begin{pmatrix} g & \\ & h \end{pmatrix} \right) \sigma_{\mathfrak{g}_1^*}(g^{-1}X_{\mathfrak{t}}h)^k \sigma_{\mathfrak{g}_1}(g^{-1}X_{\mathfrak{t}^\perp}h)^k dX \\
 &  \ll \Xi^{G_1}(g^{-1}t_1g)\Xi^{G_1}(h^{-1}t_2h)\sigma_{G_1}(g^{-1}t_1g)^{-r}\sigma_{G_1}(h^{-1}t_2h)^{-r}\sigma_{T\backslash G_1}(g)^{(d+4)k} \sigma_{T\backslash G_1}(h)^{(d+4)k}
\end{aligned}\]

\noindent for all $g,h\in G_1$, where we recall that $d$ denotes the rank of $A_T$ (see \S \ref{concrete T}). As the left hand side of the above inequality is invariant by left translations of both $g$ and $h$ by $T$, by \ref{geom side}(1) we see that we may replace $\sigma_{T\backslash G_1}(g)$ and $\sigma_{T\backslash G_1}(h)$ by $\sigma_{G_1}(g)$ and $\sigma_{G_1}(h)$ respectively in the right hand side of the inequality. Moreover, we have

\[\begin{aligned}
\displaystyle \sigma_{\mathfrak{g}_1^*}(g^{-1}X_{\mathfrak{t}}h) \sigma_{\mathfrak{g}_1}(g^{-1}X_{\mathfrak{t}^\perp}h) & \ll \log\left(2+\nu(X_{\mathfrak{t}}^{-1})\right)\sigma_{\mathfrak{g}_1}(X)^2 \sigma_{G_1}(g)^2\sigma_{G_1}(h)^2 \\
 & \ll \log\left(2+\nu(X_{\mathfrak{t}}^{-1})\right)\sigma_{\mathfrak{g}_1}(g^{-1}Xh)^2 \sigma_{G_1}(g)^4\sigma_{G_1}(h)^4 \\
 & \ll \log\left(2+\nu(X_{\mathfrak{t}}^{-1})\right) \sigma_{G_1}(g)^4\sigma_{G_1}(h)^4 \\
 & \times \sigma_G\left(\begin{pmatrix} g^{-1} & \\ & h^{-1} \end{pmatrix} \begin{pmatrix}t_1 & \\ & t_2 \end{pmatrix} \begin{pmatrix} 1 & X \\ & 1\end{pmatrix} \begin{pmatrix} g & \\ & h \end{pmatrix} \right)^2
\end{aligned}\]

\noindent for all $g,h\in G_1$ and all $X\in \mathfrak{g}_1$. Thus, as the function $\gamma\in G\mapsto f(\gamma)\sigma_G(\gamma)^{2k}$ is again Harish-Chandra Schwartz, up to replacing $f$ by this function, we just need to show that for all $r>0$, we have

\[\begin{aligned}
\displaystyle \left(\frac{\nu(h)}{\nu(g)}\right)^n & \int_{\mathfrak{g}_1} f\left(\begin{pmatrix} g^{-1} & \\ & h^{-1} \end{pmatrix} \begin{pmatrix}t_1 & \\ & t_2 \end{pmatrix} \begin{pmatrix} 1 & X \\ & 1\end{pmatrix} \begin{pmatrix} g & \\ & h \end{pmatrix} \right) \log\left(2+\nu(X_{\mathfrak{t}}^{-1})\right)^k dX \\
 &  \ll \Xi^{G_1}(g^{-1}t_1g)\Xi^{G_1}(h^{-1}t_2h)\sigma_{G_1}(g^{-1}t_1g)^{-r}\sigma_{G_1}(h^{-1}t_2h)^{-r}\sigma_{G_1}(g)^{dk} \sigma_{G_1}(h)^{dk}
\end{aligned}\]

\noindent for all $g,h\in G_1$. Similarly, by using Lemma 1.8.4 of \cite{BeuGalP} (together with the same reductions as before), we see that the absolute convergence of $I_{T,k}(f)$ is implied by the following stronger inequality for all $r>0$:

\[\begin{aligned}
\displaystyle & \left(\frac{\nu(h)}{\nu(g)}\right)^n \int_{\mathfrak{g}_1} f\left(\begin{pmatrix} g^{-1} & \\ & h^{-1} \end{pmatrix} \begin{pmatrix}a_1 & \\ & a_2 \end{pmatrix} \begin{pmatrix} 1 & X \\ & 1\end{pmatrix} \begin{pmatrix} g & \\ & h \end{pmatrix} \right) \log\left(2+\nu(X_{\mathfrak{t}}^{-1})\right)^k dX \\
 &  \ll \left(\frac{\nu(a_2)}{\nu(a_1)}\right)^{n/2} \Xi^{G_1}(g^{-1}a_1g)\Xi^{G_1}(h^{-1}a_2h)\sigma_{G_1}(g^{-1}a_1g)^{-r}\sigma_{G_1}(h^{-1}a_2h)^{-r}\sigma_{G_1}(g)^{dk} \sigma_{G_1}(h)^{dk}
\end{aligned}\]

\noindent for all $g,h\in G_1$ and all $a_1,a_2\in T$. Fix $r>0$ and for all $Y\in \mathfrak{t}$, let $Y=Y_1+\ldots+Y_d$ be the decomposition of $Y$ according to the identification \ref{concrete T}.(3) (so that $Y_i\in K_i$ for all $1\leqslant i\leqslant d$). Then, we have

$$\displaystyle \log\left(2+\nu(Y^{-1})\right)\ll \prod_{i=1}^d \max(1,v_{K_i}(Y_i))$$

\noindent for all $Y\in \mathfrak{t}^*$. By H\"older inequality, we have

\[\begin{aligned}
\displaystyle & \int_{\mathfrak{g}_1} f\left(\begin{pmatrix} g^{-1} & \\ & h^{-1} \end{pmatrix} \begin{pmatrix}a_1 & \\ & a_2 \end{pmatrix} \begin{pmatrix} 1 & X \\ & 1\end{pmatrix} \begin{pmatrix} g & \\ & h \end{pmatrix} \right) \log\left(2+\nu(X_{\mathfrak{t}}^{-1})\right)^k dX \\
 & \ll \left( \prod_{i=1}^d \int_{\mathfrak{g}_1} f\left(\begin{pmatrix} g^{-1} & \\ & h^{-1} \end{pmatrix} \begin{pmatrix}a_1 & \\ & a_2 \end{pmatrix} \begin{pmatrix} 1 & X \\ & 1\end{pmatrix} \begin{pmatrix} g & \\ & h \end{pmatrix} \right) \max(1,v_{K_i}(X_{\mathfrak{t},i}))^{dk} dX\right)^{1/d}.
\end{aligned}\]

\noindent Hence, we just need to establish that for all $1\leqslant i\leqslant d$, we have

\[\begin{aligned}
\displaystyle (1)\;\;\; & \left(\frac{\nu(h)}{\nu(g)}\right)^n \int_{\mathfrak{g}_1} f\left(\begin{pmatrix} g^{-1} & \\ & h^{-1} \end{pmatrix} \begin{pmatrix}a_1 & \\ & a_2 \end{pmatrix} \begin{pmatrix} 1 & X \\ & 1\end{pmatrix} \begin{pmatrix} g & \\ & h \end{pmatrix} \right) \max(1,v_{K_i}(X_{\mathfrak{t},i}))^{dk} dX \\
 &  \ll \left(\frac{\nu(a_2)}{\nu(a_1)}\right)^{n/2} \Xi^{G_1}(g^{-1}a_1g)\Xi^{G_1}(h^{-1}a_2h)\sigma_{G_1}(g^{-1}a_1g)^{-r}\sigma_{G_1}(h^{-1}a_2h)^{-r}\sigma_{G_1}(g)^{dk} \sigma_{G_1}(h)^{dk}
\end{aligned}\]

\noindent for all $g,h\in G_1$ and all $a_1,a_2\in T$. Fix $1\leqslant i\leqslant d$ and let $\mathfrak{t}_i$ be the subspace of $\mathfrak{t}$ corresponding to $K_i$ via the identification \ref{concrete T}.(3). Denote by $\mathfrak{t}_i^\perp$ the orthogonal complement of $\mathfrak{t}_i$ in $\Fg_1$ with respect to the bilinear pairing $\langle .,.\rangle$. We can write

\[\begin{aligned}
\displaystyle & \int_{\mathfrak{g}_1} f\left(\begin{pmatrix} g^{-1} & \\ & h^{-1} \end{pmatrix} \begin{pmatrix}a_1 & \\ & a_2 \end{pmatrix} \begin{pmatrix} 1 & X \\ & 1\end{pmatrix} \begin{pmatrix} g & \\ & h \end{pmatrix} \right) \max(1,v_{K_i}(X_{\mathfrak{t},i}))^{dk} dX \\
 & =\int_{\mathfrak{t}_i^\perp}\int_{\mathfrak{t}_i} f\left(\begin{pmatrix} g^{-1} & \\ & h^{-1} \end{pmatrix} \begin{pmatrix}a_1 & \\ & a_2 \end{pmatrix} \begin{pmatrix} 1 & X+Y \\ & 1\end{pmatrix} \begin{pmatrix} g & \\ & h \end{pmatrix} \right) \max(1,v_{K_i}(Y))^{dk} dYdX
\end{aligned}\]

\noindent for all $g,h\in G_1$ and all $a_1,a_2\in T$. Let $J\subset G$ be an open subgroup by which $f$ is biinvariant. Then, there exists $A>0$ such that $\displaystyle \begin{pmatrix} 1 & g^{-1}Yh \\ & 1\end{pmatrix}\in J$ for all $g,h\in G_1$ and all $Y\in \mathfrak{t}_i$ with $v_{K_i}(Y)\geqslant A\sigma_{G_1}(g)\sigma_{G_1}(h)$. For all $g,h\in G_1$, set

$$\displaystyle \mathfrak{t}_i[<,g,h]:=\{Y\in \mathfrak{t}_i\mid v_{K_i}(Y)<A\sigma_{G_1}(g)\sigma_{G_1}(h) \},$$

$$\displaystyle \mathfrak{t}_i[\geqslant,g,h]:=\{Y\in \mathfrak{t}_i\mid v_{K_i}(Y)\geqslant A\sigma_{G_1}(g)\sigma_{G_1}(h) \}.$$

\noindent Then by further decomposing the above integral, we have

\[\begin{aligned}
\displaystyle & \int_{\mathfrak{g}_1} f\left(\begin{pmatrix} g^{-1} & \\ & h^{-1} \end{pmatrix} \begin{pmatrix}a_1 & \\ & a_2 \end{pmatrix} \begin{pmatrix} 1 & X \\ & 1\end{pmatrix} \begin{pmatrix} g & \\ & h \end{pmatrix} \right) \max(1,v_{K_i}(X_{\mathfrak{t},i}))^{dk} dX \\
 & =\int_{\mathfrak{t}_i^\perp}\int_{\mathfrak{t}_i[<,g,h]} f\left(\begin{pmatrix} g^{-1} & \\ & h^{-1} \end{pmatrix} \begin{pmatrix}a_1 & \\ & a_2 \end{pmatrix} \begin{pmatrix} 1 & X+Y \\ & 1\end{pmatrix} \begin{pmatrix} g & \\ & h \end{pmatrix} \right) \max(1,v_{K_i}(Y))^{dk} dYdX \\
 & + \int_{\mathfrak{t}_i^\perp}\int_{\mathfrak{t}_i[\geqslant,g,h]} f\left(\begin{pmatrix} g^{-1} & \\ & h^{-1} \end{pmatrix} \begin{pmatrix}a_1 & \\ & a_2 \end{pmatrix} \begin{pmatrix} 1 & X \\ & 1\end{pmatrix} \begin{pmatrix} g & \\ & h \end{pmatrix} \right) \max(1,v_{K_i}(Y))^{dk} dYdX
\end{aligned}\]

\noindent for all $g,h\in G_1$ and all $a_1,a_2\in T$. By applying Lemma \ref{lem basic estimate}, we see that this last expression is essentially bounded by the product of $\sigma_{G_1}(g)^{dk}\sigma_{G_1}(h)^{dk}$ times

\[\begin{aligned}
\displaystyle & \int_{\mathfrak{t}_i^\perp}\int_{\mathfrak{t}_i[<,g,h]} f\left(\begin{pmatrix} g^{-1} & \\ & h^{-1} \end{pmatrix} \begin{pmatrix}a_1 & \\ & a_2 \end{pmatrix} \begin{pmatrix} 1 & X+Y \\ & 1\end{pmatrix} \begin{pmatrix} g & \\ & h \end{pmatrix} \right) dYdX \\
 & + \int_{\mathfrak{t}_i^\perp}\int_{\mathfrak{t}_i[\geqslant,g,h]} f\left(\begin{pmatrix} g^{-1} & \\ & h^{-1} \end{pmatrix} \begin{pmatrix}a_1 & \\ & a_2 \end{pmatrix} \begin{pmatrix} 1 & X \\ & 1\end{pmatrix} \begin{pmatrix} g & \\ & h \end{pmatrix} \right) dYdX \\
 & =\int_{\mathfrak{t}} f\left(\begin{pmatrix} g^{-1} & \\ & h^{-1} \end{pmatrix} \begin{pmatrix}a_1 & \\ & a_2 \end{pmatrix} \begin{pmatrix} 1 & X \\ & 1\end{pmatrix} \begin{pmatrix} g & \\ & h \end{pmatrix} \right) dX
\end{aligned}\]

\noindent for all $g,h\in G_1$ and all $a_1,a_2\in T$. Finally, it follows from Proposition II.4.5 of \cite{WalPlanch} that

\[\begin{aligned}
\displaystyle \left(\frac{\nu(h)}{\nu(g)}\right)^n  & \int_{\mathfrak{t}} f\left(\begin{pmatrix} g^{-1} & \\ & h^{-1} \end{pmatrix} \begin{pmatrix}a_1 & \\ & a_2 \end{pmatrix} \begin{pmatrix} 1 & X \\ & 1\end{pmatrix} \begin{pmatrix} g & \\ & h \end{pmatrix} \right) dX \\
 & =\int_{\mathfrak{t}} f\left(\begin{pmatrix}g^{-1}a_1g & \\ & h^{-1}a_2h \end{pmatrix} \begin{pmatrix} 1 & X \\ & 1\end{pmatrix} \right) dX \\
& \ll \left(\frac{\nu(a_2)}{\nu(a_1)}\right)^{n/2} \Xi^{G_1}(g^{-1}a_1g)\Xi^{G_1}(h^{-1}a_2h)\sigma_{G_1}(g^{-1}a_1g)^{-r}\sigma_{G_1}(h^{-1}a_2h)^{-r}
\end{aligned}\]

\noindent for all $g,h\in G_1$ and all $a_1,a_2\in T$. This proves inequality (1) and ends the proof of (i).

\item We prove the first inequality, the proof of the second one being similar. Fix $f\in \mathcal{C}(G)$, $t_1,t_2\in T\cap G_{1,reg}$, $k>0$ and $r>0$. For all $C>0$ we have

\[\begin{aligned}
\displaystyle \displaystyle \left\lvert I_k(f,t_1,t_2)-I_{k,\leqslant C}(f,t_1,t_2)\right\rvert\leqslant & I^1_{k,>C}(\lvert f\rvert,t_1,t_2)+I^2_{k,>C}(\lvert f\rvert,t_1,t_2)+I^3_{k,>C}(\lvert f\rvert,t_1,t_2) \\
 & +I^4_{k,>C}(\lvert f\rvert,t_1,t_2)
\end{aligned}\]

\noindent where

\[\begin{aligned}
\displaystyle I^1_{k,>C}(\lvert f\rvert,t_1,t_2):= & \int_{T\times T\backslash G_1\times G_1} \int_{\mathfrak{g}_1} \left\lvert f\left(\begin{pmatrix} g^{-1} & \\ & h^{-1} \end{pmatrix} \begin{pmatrix}t_1 & \\ & t_2 \end{pmatrix} \begin{pmatrix} 1 & X \\ & 1\end{pmatrix} \begin{pmatrix} g & \\ & h \end{pmatrix} \right)\right\rvert \\
 & \sigma_{T\backslash G_1}(g)^k \sigma_{T\backslash G_1}(h)^k \sigma_{\mathfrak{g}_1^*}(g^{-1}X_{\mathfrak{t}}h)^k \sigma_{\mathfrak{g}_1}(g^{-1}X_{\mathfrak{t}^\perp}h)^k \\
 & \mathbf{1}_{T\backslash G_1,> C}(g)dX \frac{dg}{\nu(g)^n} \nu(h)^n dh,
\end{aligned}\]

\noindent and $I^2_{k,>C}(\lvert f\rvert,t_1,t_2)$, $I^3_{k,>C}(\lvert f\rvert,t_1,t_2)$ and $I^4_{k,>C}(\lvert f\rvert,t_1,t_2)$ are defined similarly by replacing $\mathbf{1}_{T\backslash G_1,> C}(g)$ by $\mathbf{1}_{T\backslash G_1,> C}(h)$, $\mathbf{1}_{\mathfrak{g}_1^*,> C}(g^{-1}X_{\mathfrak{t}}h)$ and $\mathbf{1}_{\mathfrak{g}_1,> C}(g^{-1}X_{\mathfrak{t}^\perp}h)$ respectively. Then for all $1\leqslant j\leqslant 4$ and all $C>0$, we have

$$\displaystyle I^j_{k,>C}(\lvert f\rvert,t_1,t_2)\leqslant C^{-r} I_{k+r}(\lvert f\rvert,t_1,t_2).$$

\noindent By (i) $I_{k+r}(\lvert f\rvert,t_1,t_2)$ is absolutely convergent. This proves the claimed inequality and finishes the proof of the Proposition. $\blacksquare$
\end{enumerate}

\subsection{Computation of certain $(G,M,\theta)$-orthogonal sets}\label{Computation}

\noindent Recall that we have fixed a basis of $V$ compatible with the decomposition \ref{concrete T} (2). In this basis, the maximal $(\theta,F)$-split central torus $A_{M,\theta}$ of $M$ is the subgroup of matrices of the form

$$\displaystyle \begin{pmatrix} \lambda_1I_{k_1} & & & & & \\ & \ddots & & & & \\ & & \lambda_d I_{k_d} & & & \\ & & & \lambda_1^{-1} I_{k_1}& & \\ & & & & \ddots & \\ & & & & & \lambda_d^{-1} I_{k_d}\end{pmatrix}$$

\noindent where $\lambda_i\in F^\times$ for all $1\leqslant i\leqslant d$ and we recall that $k_i:=[K_i:F_i]$. Thus we have an identification of $\mathcal{A}_{M,\theta}$ with $\mathbf{R}^d$ such that for all $m=(m_1,\ldots,m_{2d})\in M$ (decomposition according to the identification \ref{concrete T}.(4)), we have

$$\displaystyle H_{M,\theta}(m)=\left(\frac{v_F(N_{M_{k_1}(\mathcal{D})/F} m_1)-v_{F}(N_{M_{k_1}(\mathcal{D})/F} m_{d+1})}{2},\ldots,  \frac{v_{F}(N_{M_{k_d}(\mathcal{D})/F} m_d)-v_{F}(N_{M_{k_d}(\mathcal{D})/F} m_{2d})}{2}\right)$$

\noindent where for all $1\leqslant i\leqslant d$, $N_{M_{k_i}(\mathcal{D})/F}:M_i(\mathcal{D})\to F$ denotes the reduced norm. For all $X=(X_1,\ldots, X_d)\in \mathfrak{t}^*$, we will also set

$$\displaystyle V(X):=\left(v_{K_1}(X_1),\ldots, v_{K_d}(X_d) \right)$$

\noindent where $v_{K_i}:=v_F\circ N_{K_i/F}$ for all $1\leqslant i\leqslant d$. Note that we have

$$\displaystyle H_{M,\theta}\begin{pmatrix} X & 0 \\ 0 & X^{-1} \end{pmatrix}=V(X) \leqno (1)$$

\noindent for all $X\in \mathfrak{t}^*$. We extend the map $H_{M,\theta}:M\to \mathcal{A}_{M,\theta}$ to the maps $H_{P,\theta}:G\to\mathcal{A}_{M,\theta}$ for all $P\in \mathcal{P}^\theta(M)$ by using the maximal compact subgroup $K$.

\vspace{2mm}

Set $W:=\left(\mathbf{Z}/2\mathbf{Z}\right)^d$. We identify $W$ with a subgroup of $G$ by sending the $i$-th element $e_i$ of the canonical basis of $W$ to the element of $G$ which switch the two copies of $K_i\otimes_{F_i}\mathcal{D}$ in $V\oplus V$. By the assumption that $T^c\subset K_1$, we have $W\subset K$. Conjugation by $W$ preserves $M$ and commutes with $\theta$. Hence it induces an action on the set $\mathcal{F}^\theta(M)$ of $\theta$-split parabolic subgroup containing $M$ and also on $\mathcal{A}_{M,\theta}$ that we shall simply denote by $(w,P)\mapsto wP$ and $(w,X)\mapsto wX$ respectively. We have

$$\displaystyle \mathcal{P}^\theta(M)=\bigsqcup_{w\in W} w\mathcal{P}^{\theta,Q}(M) \leqno (2)$$

\noindent where we recall that $\mathcal{P}^\theta(M)$ is the set of $\theta$-split parabolic subgroups with Levi component $M$ and $\mathcal{P}^{\theta,Q}(M)$ is the subset of $P\in \mathcal{P}^\theta(M)$ such that $P\subset Q$.

\vspace{2mm}

\noindent\ul{Proof of (2)}: It suffices to show that

$$\displaystyle \bigcup_{P\in \mathcal{P}^{\theta,Q}(M)} \overline{\mathcal{A}_{P,\theta}^+}$$

\noindent is a fundamental domain for the action of $W$ on $\mathcal{A}_{M,\theta}$. With the identification above, we have

$$\displaystyle \bigcup_{P\in \mathcal{P}^{\theta,Q}(M)} \overline{\mathcal{A}_{P,\theta}^+}=\left\{X\in \mathcal{A}_{M,\theta};\; \langle \alpha,X\rangle\geqslant 0\; \forall \alpha\in R(A_{M,\theta},U_Q) \right\}=\mathbf{R}_-^d.$$

\noindent Now the claim follows easily from the fact that the action of $W$ on $\mathcal{A}_{M,\theta}$ is given by sign changes of the coordinates.

\begin{lem}\label{lem limit of GMT family}
There exist $c>0$ and $N_0\geqslant 1$ such that for all $X\in \mathfrak {t}^*$, all $Y\in \mathfrak{t}^\perp$, all $g,h\in G_1$ and all $\lambda\in F^\times$ satisfying

\begin{itemize}
\item $v_F(\lambda)\geqslant N_0$;

\item $\sigma_{G_1}(g)\leqslant c v_F(\lambda)$, $\sigma_{G_1}(h)\leqslant c v_F(\lambda)$, $\sigma_{\mathfrak{t}^*}(X)\leqslant c v_F(\lambda)$ and $\sigma_{\mathfrak{g}_1}(Y)\leqslant c v_F(\lambda)$,
\end{itemize}

\noindent we have

$$\displaystyle H_{wP,\theta}\left(\begin{pmatrix} 1 & \lambda^{-1}X+Y \\ & 1 \end{pmatrix} \begin{pmatrix} g & \\ & h \end{pmatrix}\right)-\frac{1}{2}V(\lambda^{-1}X)=w(H_{P,\theta}\begin{pmatrix} g & \\ & h \end{pmatrix}-\frac{1}{2}V(\lambda^{-1}X))$$

\noindent for all $P\in \mathcal{P}^{\theta,Q}(M)$ and all $w\in W$.
\end{lem}

\vspace{2mm}

\noindent\ul{Proof}: Let $P\in \mathcal{P}^{\theta,Q}(M)$, $w\in W$ and fix $N_0\geqslant 1$ and $c>0$. Let $X,Y,g,h,\lambda$ satisfying the conditions of the lemma. We will show that the claimed equality is true provided $N_0$ is sufficiently large and $c$ sufficiently small. For simplicity, we will set $z=\lambda^{-1}$.

\vspace{2mm}

\noindent Since $W\subset K$, we have

$$\displaystyle H_{wP,\theta}\left(\begin{pmatrix} 1 & zX+Y \\ & 1 \end{pmatrix} \begin{pmatrix} g & \\ & h \end{pmatrix}\right)=w.H_{P,\theta}\left(w\begin{pmatrix} 1 & zX+Y \\ & 1 \end{pmatrix} \begin{pmatrix} g & \\ & h \end{pmatrix}\right).$$

\noindent Up to reordering the $K_i$'s, we may assume that $w=e_1+\ldots+e_t$ for some $1\leqslant t\leqslant d$. In this case we have

$$\displaystyle w=\begin{pmatrix} 0_k & 0 & I_k  & 0 \\ 0 & I_{m-k} & 0 & 0_{m-k} \\ I_k & 0 & 0_k & 0 \\ 0 & 0_{m-k} & 0 & I_{m-k}\end{pmatrix}$$

\noindent where $k:=k_1+\ldots+k_t$ and $m=n/r$. Write $X=\begin{pmatrix} X_k & \\ & X_{m-k}\end{pmatrix}$ and $Y=\begin{pmatrix} Y_k & Z \\ W & Y_{m-k}\end{pmatrix}$ where $(X_k,Y_k)\in M_k(\mathcal{D})^2$, $(X_{m-k},Y_{m-k})\in M_{m-k}(\mathcal{D})^2$, $W\in M_{m-k,k}(\mathcal{D})$ and $Z\in M_{k,m-k}(\mathcal{D})$. Note that if $N_0$ is sufficiently large and $c$ sufficiently small, the matrix $zX_k+Y_k=zX_k(I_k+\lambda X_k^{-1}Y_k)$ is invertible. Direct computation then gives

\[\begin{aligned}
\displaystyle & w\begin{pmatrix} 1 & zX+Y \\ & 1 \end{pmatrix}= \\
 & \begin{pmatrix}-(zX_k+Y_k)^{-1} & & & \\ & I_{m-k} & & \\ & & zX_k+Y_k & \\ & & & I_{m-k}\end{pmatrix} \begin{pmatrix} I_k & & -(zX_k+Y_k) & Z \\ & I_{m-k} & W & zX_{m-k}+Y_{m-k} \\ & & I_k & \\ & & & I_{m-k} \end{pmatrix} \\
 & \times \begin{pmatrix} I_k & 0 & 0_k & 0 \\ -W(zX_k+Y_k)^{-1} & I_{m-k} & 0 & -W(zX_k+Y_k)^{-1}Z \\ (zX_k+Y_k)^{-1} & 0 & I_k & (zX_k+Y_k)^{-1}Z \\ 0 & 0_{m-k} & 0 & I_{m-k}\end{pmatrix}.
\end{aligned}\]

\noindent For $N_0$ sufficiently large and $c$ sufficiently small, the matrix

$$\displaystyle \begin{pmatrix} 0_k & 0 & 0_k & 0 \\ -W(zX_k+Y_k)^{-1} & 0_{m-k} & 0 & -W(zX_k+Y_k)^{-1}Z \\ (zX_k+Y_k)^{-1} & 0 & 0_k & (zX_k+Y_k)^{-1}Z \\ 0 & 0_{m-k} & 0 & 0_{m-k}\end{pmatrix}$$

\noindent is so small compared to $g$ and $h$ that

$$\displaystyle \begin{pmatrix} g^{-1} & \\ & h^{-1} \end{pmatrix}\begin{pmatrix} I_k & 0 & 0_k & 0 \\ -W(zX_k+Y_k)^{-1} & I_{m-k} & 0 & -W(zX_k+Y_k)^{-1}Z \\ (zX_k+Y_k)^{-1} & 0 & I_k & (zX_k+Y_k)^{-1}Z \\ 0 & 0_{m-k} & 0 & I_{m-k}\end{pmatrix}\begin{pmatrix} g & \\ & h \end{pmatrix}$$

\noindent belongs to $K$. Hence, for $N_0$ sufficiently large and $c$ sufficiently small, we have

$$\displaystyle H_{P,\theta}\left(w\begin{pmatrix} 1 & zX+Y \\ & 1 \end{pmatrix} \begin{pmatrix} g & \\ & h \end{pmatrix}\right)=H_{P,\theta}\begin{pmatrix}\begin{pmatrix}-(zX_k+Y_k)^{-1} & \\ & I_{m-k}\end{pmatrix} g & \\ & \begin{pmatrix} zX_k+Y_k & \\ & I_{m-k}\end{pmatrix} h\end{pmatrix}.$$

\noindent Once again, for $N_0$ sufficiently large and $c$ sufficiently small, the matrix $z^{-1}X_k^{-1}Y_k$ is so small compared to $g$ and $h$ that we have

$$\displaystyle g^{-1}\begin{pmatrix}(I_k+z^{-1}X_k^{-1}Y_k)^{-1} & \\ & I_{m-k}\end{pmatrix}g\in K_1, \;\;\; h^{-1}\begin{pmatrix} I_k+z^{-1}X_k^{-1}Y_k & \\ & I_{m-k}\end{pmatrix}h\in K_1.$$

\noindent So finally we get that

\[\begin{aligned}
\displaystyle H_{wP,\theta}\left(\begin{pmatrix} 1 & zX+Y \\ & 1 \end{pmatrix} \begin{pmatrix} g & \\ & h \end{pmatrix}\right)=w.H_{M,\theta}\begin{pmatrix} -z^{-1}X_k^{-1} & & & \\ & I_{m-k} & & \\ & & zX_k & \\ & & & I_{m-k}\end{pmatrix}+w.H_{P,\theta}\begin{pmatrix} g & \\ & h \end{pmatrix}
\end{aligned}\]

\noindent for $N_0$ sufficiently large and $c$ sufficiently small. Now, from (1) we easily check that

$$\displaystyle w.H_{M,\theta}\begin{pmatrix} -z^{-1}X_k^{-1} & & & \\ & I_{m-k} & & \\ & & zX_k & \\ & & & I_{m-k}\end{pmatrix}=\frac{V(zX)-wV(zX)}{2}$$

\noindent and this ends the proof of the lemma. $\blacksquare$

\subsection{Computation of certain singular $\theta$-weighted orbital integrals}\label{first application}

To all $g,h\in G_1$ and $X\in \mathfrak{t}^*\oplus \mathfrak{t}^\perp$ we associate a $(G,M,\theta)$-orthogonal set $\mathcal{Z}(g,h,X)=(\mathcal{Z}(g,h,X)_P)_{P\in \mathcal{P}^\theta(M)}$ defined by

$$\displaystyle \mathcal{Z}(g,h,X)_{w\overline{P}}:=wH_{P,\theta}\begin{pmatrix} g & \\ & h\end{pmatrix}-\frac{1}{2}(V(X_{\mathfrak{t}})-wV(X_{\mathfrak{t}}))$$

\noindent for all $P\in \mathcal{P}^{\theta,Q}(M)$ and $w\in W$ where we recall that $X_{\mathfrak{t}}$ denotes the projection of $X$ onto $\mathfrak{t}$. The fact that $\mathcal{Z}(g,h,X)$ is indeed a $(G,M,\theta)$-orthogonal set easily follows from Lemma \ref{lem limit of GMT family}. We define

$$\displaystyle w_{M,\theta}(g,h,X):=v_{M,\theta}(\mathcal{Z}(g,h,X))$$

\noindent for all $g,h\in G_1$ and $X\in \mathfrak{t}^*\oplus \mathfrak{t}^\perp$. Note that $w_{M,\theta}(g,h,X)$ is NOT always the volume of the convex hull of the set $\mathcal{Z}(g,h,X)$ (this is because this $(G,M,\theta)$-orthogonal set is not always positive). However, there exists $k>0$ such that

$$\displaystyle \lvert w_{M,\theta}(g,h,X)\rvert\ll \sigma_{T\backslash G_1}(g)^k\sigma_{T\backslash G_1}(h)^k\sigma_{\mathfrak{g}_1^*}(g^{-1}X_{\mathfrak{t}}h)^k \leqno (1)$$

\noindent for all $g,h\in G_1$ and $X\in \mathfrak{t}^*\oplus \mathfrak{t}^\perp$. Indeed, by \ref{(G,M) and (G,M,theta)-orthogonal sets}(3), there exists $k>0$ such that

$$\displaystyle \lvert w_{M,\theta}(g,h,X)\rvert\ll \sigma_{G_1}(g)^k\sigma_{G_1}(h)^k\sigma_{\mathfrak{t}^*}(X_{\mathfrak{t}})^k\ll \sigma_{G_1}(g)^k\sigma_{G_1}(h)^k\sigma_{\mathfrak{g}_1^*}(g^{-1}X_{\mathfrak{t}}h)^k$$

\noindent for all $(g,h,X)\in G_1^2\times (\mathfrak{t}^*\oplus \mathfrak{t}^\perp)$. We easily check that $w_{M,\theta}(t_1g,t_2h,t_1Xt_2^{-1})=w_{M,\theta}(g,h,X)$ for all $t_1,t_2\in T$ and thus we get

$$\displaystyle \lvert w_{M,\theta}(g,h,X)\rvert\ll \sigma_{G_1}(t_1g)^k\sigma_{G_1}(t_2h)^k\sigma_{\mathfrak{g}_1^*}(g^{-1}X_{\mathfrak{t}}h)^k$$

\noindent for all $g,h\in G_1$, $X\in \mathfrak{t}^*\oplus \mathfrak{t}^\perp$ and $t_1,t_2\in T$. Taking the infimum over $t_1$ and $t_2$ of the right hand side yields the desired inequality by \ref{geom side}(1).

\vspace{2mm}

\noindent Recall that in \S \ref{weighted and theta weighted orbital integrals}, we have defined a $\theta$-weighted orbital integral $f\in \mathcal{C}(G)\mapsto \Phi_{M,\theta}(x,f)$ for all $x\in M\cap G_{reg}$ (where we again use the maximal compact subgroup $K$ to define the maps $H_{P,\theta}:G\to \mathcal{A}_{M,\theta}$ for $P\in \mathcal{P}^\theta(M)$). For all $t\in T\cap G_{1,reg}$, we define a $\theta$-weighted orbital integral at the non-regular point $\begin{pmatrix}t & \\ & t\end{pmatrix}$ by

\[\begin{aligned}
\displaystyle \Phi_{M,\theta}\left(\begin{pmatrix} t & \\ & t\end{pmatrix},f\right):=D^{G_1}(t)^{-1}\int_{T\times T\backslash G_1\times G_1} \int_{\mathfrak{g}_1}  & f^K\left(\begin{pmatrix} g^{-1} & \\ & h^{-1} \end{pmatrix}\begin{pmatrix} t & \\ & t \end{pmatrix} \begin{pmatrix} 1 & X \\ & 1 \end{pmatrix}\begin{pmatrix} g & \\ & h \end{pmatrix}\right) \\
 & w_{M,\theta}(g,h,X)dX \frac{dg}{\nu(g)^n}\nu(h)^n dh
\end{aligned}\]

\noindent for all $f\in \mathcal{C}(G)$. Note that this integral is absolutely convergent by (1) and Lemma \ref{prop estimate}(i).

\begin{prop}\label{prop first application}
For all $\theta$-strongly cuspidal function $f\in \mathcal{C}(G)$ and all $t\in T\cap G_{1,reg}$, we have

\[\begin{aligned}
\displaystyle \lim\limits_{\lambda\to 1} D^G\begin{pmatrix} t & \\ & \lambda t\end{pmatrix}^{1/2}\Phi_{M,\theta}\left(\begin{pmatrix} t & \\ & \lambda t\end{pmatrix}, f\right)= D^{G_1}(t)^2 \Phi_{M,\theta}\left(\begin{pmatrix} t & \\ & t\end{pmatrix},f\right).
\end{aligned}\]

\noindent (In particular the limit exists).
\end{prop}

\noindent\ul{Proof}: Recall that by definition

$$\displaystyle \Phi_{M,\theta}\left(\begin{pmatrix} t & \\ & \lambda t\end{pmatrix}, f\right)=\int_{T\times T\backslash G} f\left(g^{-1}\begin{pmatrix} t & \\ & \lambda t\end{pmatrix} g\right) v_{M,\theta}(\mathcal{Y}(g)) dg$$

\noindent where for $g\in G$, $\mathcal{Y}(g)$ is the $(G,M,\theta)$-orthogonal set defined by

$$\displaystyle \mathcal{Y}(g)_P=H_{\overline{P},\theta}(g),\;\;\; P\in \mathcal{P}^\theta(M).$$

\noindent For all $z\in F^\times$, we define a $(G,M,\theta)$-orthogonal set $\mathcal{X}(z)=(\mathcal{X}(z)_P)_{P\in \mathcal{P}^\theta(M)}$ by

$$\displaystyle \mathcal{X}(z)_{w\overline{P}}:=\frac{1}{2}(V(z)-wV(z)), \mbox{ for all } P\in \mathcal{P}^{Q,\theta}(M) \mbox{ and all } w\in W.$$

\noindent Then for all $\lambda\in F^\times\backslash \{ 1 \}$, we have

$$\displaystyle \Phi_{M,\theta}\left(\begin{pmatrix} t & \\ & \lambda t\end{pmatrix}, f\right)=\int_{T\times T\backslash G} f\left(g^{-1}\begin{pmatrix} t & \\ & \lambda t\end{pmatrix} g\right) v_{M,\theta}(\mathcal{Y}(g)+\mathcal{X}(1-\lambda)) dg.$$

\noindent Indeed, this follows from \ref{(G,M) and (G,M,theta)-orthogonal sets}(5) and the equality

$$\displaystyle \Phi_{M,\theta}^Q\left(\begin{pmatrix} t & \\ & \lambda t\end{pmatrix},f \right)=0$$

\noindent for all $Q\in \mathcal{F}^\theta(M)$ with $Q\neq G$, which is a consequence of the fact that $f$ is $\theta$-strongly cuspidal.

\vspace{2mm}

\noindent Now by the Iwasawa decomposition $G=(G_1\times G_1)NK$, we can write

\[\begin{aligned}
\displaystyle \Phi_{M,\theta}\left(\begin{pmatrix} t & \\ & \lambda t\end{pmatrix}, f\right)= & \int_{T\times T\backslash G_1\times G_1} \int_{\mathfrak{g}_1} f^K\left(\begin{pmatrix}g^{-1} & \\ & h^{-1}\end{pmatrix}\begin{pmatrix} 1 & -X \\ & 1 \end{pmatrix}\begin{pmatrix} t & \\ & \lambda t\end{pmatrix} \begin{pmatrix} 1 & X \\ & 1 \end{pmatrix}\begin{pmatrix}g & \\ & h \end{pmatrix}\right)  \\
 & v_{M,\theta}\left(\mathcal{Y}\left(\begin{pmatrix} 1 & X \\ & 1 \end{pmatrix}\begin{pmatrix}g & \\ & h \end{pmatrix}\right)+\mathcal{X}(1-\lambda)\right) dX \frac{dg}{\nu(g)^n} \nu(h)^n dh.
\end{aligned}\]

\noindent After the variable change $X\mapsto (1-\lambda Ad(t^{-1}))^{-1}X$, for $\lambda$ sufficiently close to $1$, we get

\[\begin{aligned}
\displaystyle & D^G\begin{pmatrix} t & \\ & \lambda t\end{pmatrix}^{1/2}\Phi_{M,\theta}\left(\begin{pmatrix} t & \\ & \lambda t\end{pmatrix}, f\right)= \\
 & D^{G_1}(t) \int_{T\times T\backslash G_1\times G_1} \int_{\mathfrak{g}_1} f^K\left(\begin{pmatrix}g^{-1} & \\ & h^{-1}\end{pmatrix}\begin{pmatrix} t & \\ & \lambda t\end{pmatrix} \begin{pmatrix} 1 & X \\ & 1 \end{pmatrix}\begin{pmatrix}g & \\ & h \end{pmatrix}\right) \widetilde{w}(g,h,X,\lambda) dX \frac{dg}{\nu(g)^n} \nu(h)^n dh
\end{aligned}\]

\noindent where we have set

$$\displaystyle \widetilde{w}(g,h,X,\lambda):=v_{M,\theta}\left(\mathcal{Y}\left(\begin{pmatrix} 1 & (1-\lambda Ad(t^{-1}))^{-1}X \\ & 1 \end{pmatrix}\begin{pmatrix}g & \\ & h \end{pmatrix}\right)+\mathcal{X}(1-\lambda)\right).$$

\noindent By \ref{(G,M) and (G,M,theta)-orthogonal sets}(3), there exists $k>0$ such that

$$\displaystyle \lvert\widetilde{w}(g,h,X,\lambda)\rvert\ll \sigma_{T\backslash G_1}(g)^k\sigma_{T\backslash G_1}(h)^k\sigma_{\mathfrak{g}_1}(g^{-1}Xh)^k v_F(1-\lambda)^k \leqno (2)$$

\noindent for all $(g,h,X,\lambda)\in G^2_1\times (\mathfrak{t}^*\oplus \mathfrak{t}^\perp)\times F^\times$ with $\lambda$ sufficiently close to $1$.

\vspace{2mm}

\noindent For all $C>0$, set

\[\begin{aligned}
\displaystyle & \Phi_{M,\theta,\leqslant C}\left(\begin{pmatrix} t & \\ & \lambda t\end{pmatrix}, f\right):= D^G\begin{pmatrix} t & \\ & \lambda t\end{pmatrix}^{-1/2}D^{G_1}(t) \times \\
 & \int_{T\times T\backslash G_1\times G_1} \int_{\mathfrak{g}_1} f^K\left(\begin{pmatrix}g^{-1} & \\ & h^{-1}\end{pmatrix}\begin{pmatrix} t & \\ & \lambda t\end{pmatrix} \begin{pmatrix} 1 & X \\ & 1 \end{pmatrix}\begin{pmatrix} g & \\ & h \end{pmatrix}\right) \widetilde{w}(g,h,X,\lambda) \\
 & \mathbf{1}_{T\backslash G_1,\leqslant C}(g) \mathbf{1}_{T\backslash G_1,\leqslant C}(h) \mathbf{1}_{\mathfrak{g}_1^*,\leqslant C}(g^{-1}X_{\mathfrak{t}}h) \mathbf{1}_{\mathfrak{g}_1,\leqslant C}(g^{-1}X_{\mathfrak{t}^\perp}h)dX \frac{dg}{\nu(g)^n} \nu(h)^n dh
\end{aligned}\]

\noindent and

\[\begin{aligned}
\displaystyle \Phi_{M,\theta,\leqslant C} & \left(\begin{pmatrix} t & \\ & t\end{pmatrix}, f\right):=  D^{G_1}(t)^{-1}\int_{T\times T\backslash G_1\times G_1} \int_{\mathfrak{g}_1} f^K\left(\begin{pmatrix}g^{-1} & \\ & h^{-1}\end{pmatrix}\begin{pmatrix} t & \\ & t\end{pmatrix} \begin{pmatrix} 1 & X \\ & 1 \end{pmatrix}\begin{pmatrix}g & \\ & h \end{pmatrix}\right)  \\
 & w_{M,\theta}\left(g,h,X\right) \mathbf{1}_{T\backslash G_1,\leqslant C}(g) \mathbf{1}_{T\backslash G_1,\leqslant C}(h) \mathbf{1}_{\mathfrak{g}_1^*,\leqslant C}(g^{-1}X_{\mathfrak{t}}h)\mathbf{1}_{\mathfrak{g}_1,\leqslant C}(g^{-1}X_{\mathfrak{t}^\perp}h)dX \frac{dg}{\nu(g)^n} \nu(h)^n dh.
\end{aligned}\]

\noindent Fix $0<\epsilon<1$. By (1), (2) and Proposition \ref{prop estimate}(ii), we have (Note that for $\lambda$ sufficiently close to $1$, $f^K$ is invariant by $\begin{pmatrix} 1 & \\ & \lambda\end{pmatrix}$ on the left):

$$\displaystyle \lim\limits_{\lambda\to 1}D^G\begin{pmatrix} t & \\ & \lambda t\end{pmatrix}^{1/2}\left(\Phi_{M,\theta}\left(\begin{pmatrix}t & \\ & \lambda t\end{pmatrix}, f\right)-\Phi_{M,\theta,\leqslant v_F(\lambda-1)^\epsilon}\left(\begin{pmatrix}t & \\ & \lambda t\end{pmatrix}, f\right) \right)=0, \leqno (3)$$

$$\displaystyle \lim\limits_{\lambda\to 1}\Phi_{M,\theta}\left(\begin{pmatrix}t & \\ &  t\end{pmatrix}, f\right)-\Phi_{M,\theta,\leqslant v_F(\lambda-1)^\epsilon}\left(\begin{pmatrix}t & \\ &  t\end{pmatrix}, f\right)=0. \leqno (4)$$

\noindent Moreover, it follows from Lemma \ref{lem limit of GMT family} that for all $\lambda$ sufficiently close to $1$ and all $g,h\in G_1$, $X\in \mathfrak{g}_1$ with $\sigma_{T\backslash G_1}(g)\leqslant v_F(\lambda-1)^\epsilon$, $\sigma_{T\backslash G_1}(h)\leqslant v(\lambda-1)^\epsilon$, $\sigma_{\mathfrak{g}_1^*}(g^{-1}X_{\mathfrak{t}}h)\leqslant v_F(\lambda-1)^\epsilon$ and $\sigma_{\mathfrak{g}_1}(g^{-1}X_{\mathfrak{t}^\perp}h)\leqslant v_F(\lambda-1)^\epsilon$, we have

$$\displaystyle \mathcal{Y}\left(\begin{pmatrix} 1 & (1-\lambda Ad(t^{-1}))^{-1}X \\ & 1 \end{pmatrix}\begin{pmatrix} g & \\ & h \end{pmatrix}\right)+\mathcal{X}(1-\lambda)=\mathcal{Z}(g,h,X),$$

\noindent and so

$$\displaystyle \widetilde{w}(g,h,X,\lambda)=w_{M,\theta}(g,h,X).$$

\noindent Since $f^K$ is left invariant by $\begin{pmatrix} \lambda & \\ & 1 \end{pmatrix}$ for $\lambda$ sufficiently close to $1$, it follows that

$$\displaystyle D^G\begin{pmatrix} t & \\ & \lambda t\end{pmatrix}^{1/2}\Phi_{M,\theta,\leqslant v(\lambda-1)^\epsilon}\left(\begin{pmatrix}t & \\ & \lambda t\end{pmatrix}, f\right)=D^{G_1}(t)^2\Phi_{M,\theta,\leqslant v(\lambda-1)^\epsilon}\left(\begin{pmatrix}t & \\ &  t\end{pmatrix}, f\right)$$

\noindent for $\lambda$ sufficiently close to $1$ and hence, by (3) and (4),

$$\displaystyle \lim\limits_{\lambda\to 1} D^G\begin{pmatrix} t & \\ & \lambda t\end{pmatrix}^{1/2}\Phi_{M,\theta}\left(\begin{pmatrix} t & \\ & \lambda t\end{pmatrix}, f\right)= D^{G_1}(t)^2 \Phi_{M,\theta}\left(\begin{pmatrix} t & \\ & t\end{pmatrix},f\right).$$

$\blacksquare$

\subsection{Change of weight}\label{change of weight}

Choose a minimal $\theta$-split Levi subgroup $M_0$ contained in $M$ and fix $P_0\in \mathcal{P}^\theta(M_0)$. Using these data we can associate to any point $Y\in \mathcal{A}_{M_0,\theta}$ a $(G,M,\theta)$-orthogonal set $(Y_P)_{P\in \mathcal{P}^\theta(M)}$ as in \S \ref{(G,M) and (G,M,theta)-orthogonal sets}. For every $Y\in \mathcal{A}_{M_0,\theta}$, every $g,h\in G_1$ and every $X\in \mathfrak{t}^*\oplus \mathfrak{t}^\perp$, we define a $(G,M,\theta)$-orthogonal set $\mathcal{Z}(g,h,X,Y):=(\mathcal{Z}(g,h,X,Y)_P)_{P\in \mathcal{P}^\theta(M)}$ by setting

$$\displaystyle \mathcal{Z}(g,h,X,Y)_P:=Y_P-\mathcal{Z}(g,h,X)_P$$

\noindent for all $P\in \mathcal{P}^\theta(M)$ where $\mathcal{Z}(g,h,X)_P$ was defined in section \ref{first application}. To such a $(G,M,\theta)$-orthogonal set is associated a function $\Gamma_{M,\theta}^G(.,\mathcal{Z}(g,h,X,Y))$ on $\mathcal{A}_{M,\theta}$ (see \S \ref{(G,M) and (G,M,theta)-orthogonal sets}). If the $(G,M,\theta)$-orthogonal set $\mathcal{Z}(g,h,X,Y)$ is positive, this is just the characteristic function of its convex hull.

\vspace{2mm}

\noindent For all $Y\in \mathcal{A}_{M_0,\theta}$, we define a new weight $\widetilde{w}_{M,\theta}(.,.,.,Y)$ on $G_1\times G_1\times \mathfrak{g}_1$ by setting

$$\displaystyle \widetilde{w}_{M,\theta}(g,h,X,Y):=\int_{T} \Gamma_{M,\theta}^G\left(H_{M,\theta}(a), \mathcal{Z}(g,h,X,Y) \right)da$$

\noindent where for simplicity we have written $H_{M,\theta}(a)$ for $H_{M,\theta}\begin{pmatrix} a & \\ & 1\end{pmatrix}$. A proof similar to that of \ref{first application}.(1) shows that for some $k>0$, we have

$$\displaystyle \lvert \widetilde{w}_{M,\theta}(g,h,X,Y)\rvert \ll \sigma_{T\backslash G_1}(g)^k\sigma_{T\backslash G_1}(h)^k\sigma_{\mathfrak{g}_1^*}(g^{-1}X_{\mathfrak{t}}h)^k \left(1+\lvert Y\rvert\right)^k \leqno (1)$$

\noindent and even

$$\displaystyle \int_{T} \lvert\Gamma_{M,\theta}^G\left(H_{M,\theta}(a), \mathcal{Z}(g,h,X,Y) \right)\rvert da \ll \sigma_{T\backslash G_1}(g)^k\sigma_{T\backslash G_1}(h)^k\sigma_{\mathfrak{g}_1^*}(g^{-1}X_{\mathfrak{t}}h)^k \left(1+\lvert Y\rvert\right)^k \leqno (2)$$

\noindent for all $g,h\in G_1$, $X\in \mathfrak{t}^*\oplus \mathfrak{t}^\perp$ and all $Y\in \mathcal{A}_{M_0,\theta}$, where $\lvert .\rvert$ denotes a norm on $\mathcal{A}_{M_0,\theta}$. Using this weight we define a new expression

\[\begin{aligned}
\displaystyle J_{Y,T}(f):=\int_{T} D^{G_1}(t)\int_{T\times T\backslash G_1\times G_1} \int_{\mathfrak{g}_1} & f^K\left(\begin{pmatrix} g^{-1} & \\ & h^{-1} \end{pmatrix}\begin{pmatrix} t & \\ & t \end{pmatrix} \begin{pmatrix} 1 & X \\ & 1 \end{pmatrix}\begin{pmatrix} g & \\ & h \end{pmatrix}\right) \\
 & \widetilde{w}_{M,\theta}(g,h,X,Y) dX \frac{dg}{\nu(g)^n}\nu(h)^n dh \omega(t)^{-1} dt
\end{aligned}\]

\noindent which is absolutely convergent by (1) and Proposition \ref{prop estimate}(ii).

\vspace{2mm}

\noindent Let $A_0$ be the maximal central $(\theta,F)$-split subtorus of $M_0$ and let $\Delta_0$ be the set of simple roots of $A_0$ in $P_0$. The goal of this section is to show the following:

\begin{prop}\label{prop change of weight}
Let $0<\epsilon_1<\epsilon_2<1$ and assume that $f\in \mathcal{C}(G)$ is $\theta$-strongly cuspidal. Then for all $r>0$, we have

$$\displaystyle \left\lvert J_{N,T}(f)-\frac{1}{2^d}J_{Y,T}(f)\right\rvert\ll N^{-r} $$

\noindent for all $N\geqslant 1$ and all $Y\in \mathcal{A}_{P_0,\theta}^+$ satisfying the following two inequalities

$$\displaystyle N^{\epsilon_1}\leqslant \inf_{\alpha\in \Delta_0} \alpha(Y), \leqno (A)$$

$$\displaystyle \sup_{\alpha\in \Delta_0} \alpha(Y)\leqslant N^{\epsilon_2}. \leqno (B)$$
\end{prop}

\noindent\ul{Proof}: For all $C>0$, all $N\geqslant 1$ and all $Y\in \mathcal{A}_{P_0,\theta}^+$, we set

\[\begin{aligned}
\displaystyle J_{N,T,\leqslant C}(f):=\int_{T} & D^{G_1}(t) \int_{T\times T\backslash G_1\times G_1} \int_{\mathfrak{g}_1} f^K\left(\begin{pmatrix} g^{-1} & \\ & h^{-1} \end{pmatrix} \begin{pmatrix} t & \\ & t\end{pmatrix} \begin{pmatrix} 1 & X \\ & 1\end{pmatrix} \begin{pmatrix} g & \\ & h\end{pmatrix}\right) \\
 & \kappa_{N,T,\xi}(g,h,X)\mathbf{1}_{T\backslash G_1,\leqslant C}(g) \mathbf{1}_{T\backslash G_1,\leqslant C}(h) \mathbf{1}_{\mathfrak{g}_1^*,\leqslant C}(g^{-1}X_{\mathfrak{t}}h) dX \frac{dg}{\nu(g)^n} \nu(h)^n dh\omega(t)^{-1} dt
\end{aligned}\]

\noindent and

\[\begin{aligned}
\displaystyle J_{Y,T,\leqslant C}(f):=\int_{T} & D^{G_1}(t) \int_{T\times T\backslash G_1\times G_1} \int_{\mathfrak{g}_1} f^K\left(\begin{pmatrix} g^{-1} & \\ & h^{-1} \end{pmatrix} \begin{pmatrix} t & \\ & t\end{pmatrix} \begin{pmatrix} 1 & X \\ & 1\end{pmatrix} \begin{pmatrix} g & \\ & h\end{pmatrix}\right) \\
 & \widetilde{w}_{M,\theta}(g,h,X,Y)\mathbf{1}_{T\backslash G_1,\leqslant C}(g) \mathbf{1}_{T\backslash G_1,\leqslant C}(h) \mathbf{1}_{\mathfrak{g}_1^*,\leqslant C}(g^{-1}X_{\mathfrak{t}}h) dX \frac{dg}{\nu(g)^n} \nu(h)^n dh\omega(t)^{-1} dt.
\end{aligned}\]

\noindent We fix henceforth an $\epsilon>0$ satisfying $\epsilon<\epsilon_1$. It follows from (1), \ref{def of an expression}.(2) and Proposition \ref{prop estimate}(ii) that for all $r>0$, we have

$$\displaystyle \left\lvert J_{N,T}(f)-J_{N,T,\leqslant N^\epsilon}(f)\right\rvert\ll N^{-r}$$

\noindent and

$$\displaystyle \left\lvert J_{Y,T}(f)-J_{Y,T,\leqslant N^\epsilon}(f)\right\rvert\ll N^{-r}$$

\noindent for all $N\geqslant 1$ and all $Y\in \mathcal{A}_{P_0,\theta}^+$ satisfying inequality (B). Thus it suffices to establish that for all $r>0$, we have

$$\displaystyle \left\lvert J_{N,T,\leqslant N^\epsilon}(f)-\frac{1}{2^d}J_{Y,T,\leqslant N^\epsilon}(f)\right\rvert\ll N^{-r} \leqno (3)$$

\noindent for all $N\geqslant 1$ and all $Y\in \mathcal{A}_{P_0,\theta}^+$ satisfying inequalities (A) and (B).

\vspace{2mm}

\noindent For all $g,h\in G_1$, $X\in \mathfrak{t}^*\oplus \mathfrak{t}^\perp$ and $Y\in \mathcal{A}_{M_0,\theta}$, we have the identity (see \ref{(G,M) and (G,M,theta)-orthogonal sets}(1))

$$\displaystyle \sum_{R\in \mathcal{F}^\theta(M)} \Gamma_{M,\theta}^R(\Lambda, \mathcal{Z}(g,h,X,Y))\tau_{R,\theta}^G(\Lambda-\mathcal{Z}(g,h,X,Y)_R)=1 \leqno (4)$$

\noindent for all $\Lambda\in \mathcal{A}_{M,\theta}$. We are going to modify this decomposition slightly by taking its convolution with a function with integral $1$ and of small support (relative to $N^{\epsilon_1}$) in $C_c^\infty(T)$. The function we need is provided by the next lemma. Before we state the lemma, we need to introduce some more notations. For $\varphi\in C_c^\infty(T)$, we define its ``Fourier transform" (a function on $\mathfrak{t}$) by

$$\displaystyle \widehat{\varphi}(X):=\int_{T} \varphi(a) \xi(aX)^{-1} da,\;\;\;X\in \mathfrak{t}.$$

\noindent We also define a function $L\in C_c^\infty(\mathfrak{t})$ by

$$\displaystyle L(X):=L_{K_1}(X_1)\ldots L_{K_d}(X_d),\;\;\; X\in \mathfrak{t}$$

\noindent where $X=X_1+\ldots+X_d$ is the decomposition of $X$ according to the identification \ref{concrete T}.(3), and for all $1\leqslant i\leqslant d$, the function $L_{K_i}(\cdot)$ is defined to be

$$\displaystyle L_{K_i}(X_i):=\left\{
    \begin{array}{ll}
        1, & \mbox{if } v_{K_i}(X_i)>0; \\
        \frac{1}{2}, & \mbox{if } v_{K_i}(X_i)=0; \\
        0, & \mbox{otherwise.}
    \end{array}
\right.
$$

\noindent Note this function factorizes through the map $X\in \mathfrak{t}\mapsto V(X)\in \mathcal{A}_{M,\theta}$ introduced in \S \ref{Computation}. Thus we can define $L(V(X)):=L(X)$ for all $X\in \mathfrak{t}$. A crucial property of the function $L$ is that it satisfies the identity

$$\displaystyle \sum_{w\in W}L(wX)=1 \leqno (5)$$

\noindent for all $X\in \mathfrak{t}$, where $W$ is the subgroup of the normalizer of $M$ that was introduced in \S \ref{Computation}.

\begin{lem}\label{lem varphi}
There exists a sequence of nonnegative functions $(\varphi_N)_{N\geqslant N_0}$ in $C_c^\infty(T/T^c)$ with the following properties

\begin{itemize}
\item $Supp(\varphi_N)\subseteq \left\{a\in T;\; \sigma_T(a)\leqslant N^\epsilon \right\}$ for all $N\geqslant N_0$;

\item $\displaystyle \int_{T} \varphi_N=1$ for all $N\geqslant N_0$;

\item There exists $c>0$ such that $\displaystyle \left\lvert \widehat{\varphi}_N(X)-L(X)\right\rvert\leqslant e^{-cN^\epsilon}$ for all $X\in \mathfrak{t}$ and all $N\geqslant N_0$.
\end{itemize}
\end{lem}

\noindent\ul{Proof}: The Fourier transform $\varphi\mapsto \widehat{\varphi}$ extends to an isomorphism betweem the space of functions $f:T\to \mathbf{C}$ such that $\nu^{-1}f$ extends to a smooth function with compact support on $\mathfrak{t}$ and $C_c^\infty(\mathfrak{t})$. Let $\check{L}$ be the inverse of $L$ for this isomorphism. Then $\check{L}$ is $T^c$-invariant and we can write

$$\displaystyle \check{L}=\sum_{i\in I} \lambda_i \mathbf{1}_{a_iT^c}$$

\noindent where $(a_i)_{i\in I}$ is a family of distinct elements of $T/T^c$ and $(\lambda_i)_{i\in I}$ is a family of positive real numbers (this is because $L$ is a positive linear combination of characteristic functions of lattices in $\mathfrak{t}$) summing to $\vol(T^c)^{-1}$ (since $L(0)=1$). Moreover, there exists $c_1>0$ such that $\lambda_i\leqslant e^{-c_1\sigma_T(a_i)}$ for all $i\in I$ (this follows from the fact that $\nu^{-1}\check{L}$ extends to a smooth function with compact support on $\mathfrak{t}$). Set

$$\displaystyle \varphi_N=\left(\vol(T^c)\sum_{i;\sigma_T(a_i)\leqslant N^\epsilon} \lambda_i \right)^{-1} \sum_{i;\sigma_T(a_i)\leqslant N^\epsilon} \lambda_i \mathbf{1}_{a_iT^c}$$

\noindent for all positive integer $N$ big enough so that $\sum_{i;\sigma_T(a_i)\leqslant N^\epsilon} \lambda_i\neq 0$. Then this sequence of functions obviously satisfies the first and second point of the lemma. Moreover, denoting by $\lVert .\rVert_{\infty}$ the sup-norm on $\mathfrak{t}$, we have

\[\begin{aligned}
\displaystyle \lVert \widehat{\varphi}_N-L\rVert_{\infty}  & \leqslant \left( \left(\vol(T^c)\sum_{i;\sigma_T(a_i)\leqslant N^\epsilon} \lambda_i\right)^{-1}-1\right)\sum_{i\in I;\sigma_T(a_i)\leqslant N^\epsilon} \lambda_i \lVert \widehat{\mathbf{1}_{a_iT^c}}\rVert_{\infty}+ \sum_{i\in I;\sigma_T(a_i)> N^\epsilon} \lambda_i \lVert \widehat{\mathbf{1}_{a_iT^c}}\rVert_{\infty} \\
 & \leqslant 2\vol(T^c) \sum_{i\in I;\sigma_T(a_i)> N^\epsilon} e^{-c_1\sigma_T(a_i)}\leqslant 2 e^{-c_1N^\epsilon/2} \int_T e^{-c_1\sigma_T(a)/2}da
\end{aligned}\]

\noindent for all $N\gg 1$. Since the last integral above is convergent, this shows the last point of the lemma for $c<c_1/2$ and $N_0$ large enough. $\blacksquare$

\vspace{2mm}

\noindent We choose a sequence of functions $(\varphi_N)_{N\geqslant N_0}$ as in the lemma. For all $g,h\in G_1$, all $X\in \mathfrak{t}^*\oplus \mathfrak{t}^\perp$, all $Y\in \mathcal{A}_{M_0,\theta}$ and all $R\in\mathcal{F}^\theta(M)$, we set

$$\displaystyle \Delta_R(\Lambda,g,h,X,Y):=\Gamma^R_{M,\theta}(\Lambda,\mathcal{Z}(g,h,X,Y))\tau_{R,\theta}^G(\Lambda-\mathcal{Z}(g,h,X,Y)_R),\; \Lambda\in \mathcal{A}_{M,\theta}.$$

\noindent And for all $\varphi\in C_c^\infty(T)$, we define the convolution $\varphi\ast \Delta_R(.,g,h,X,Y)$ as usual by

$$\displaystyle \varphi\ast \Delta_R(\Lambda,g,h,X,Y):=\int_{T} \varphi(a)\Delta_R(\Lambda-H_{M,\theta}(a),g,h,X,Y)da,\; \Lambda\in \mathcal{A}_{M,\theta}.$$

\noindent By (4) and the second property satisfied by the sequence $(\varphi_N)_{N\geqslant N_0}$, we have

$$\displaystyle \sum_{R\in \mathcal{F}^\theta(M)} \varphi_N\ast \Delta_R(\Lambda,g,h,X,Y)=1 \leqno (6)$$

\noindent for all $N\geqslant N_0$, all $\Lambda\in \mathcal{A}_{M,\theta}$ and all $(g,h,X,Y)\in G_1^2\times (\mathfrak{t}^*\oplus \mathfrak{t}^\perp)\times \mathcal{A}_{P_0,\theta}^+$.

\vspace{2mm}

\noindent Set

\[\begin{aligned}
\displaystyle J^{R,Y}_{N,T,\leqslant C}(f):= & \int_{T} D^{G_1}(t) \int_{T\times T\backslash G_1\times G_1} \int_{\mathfrak{g}_1} f^K\left(\begin{pmatrix} g^{-1} & \\ & h^{-1} \end{pmatrix} \begin{pmatrix} t & \\ & t\end{pmatrix} \begin{pmatrix} 1 & X \\ & 1\end{pmatrix} \begin{pmatrix} g & \\ & h\end{pmatrix}\right) \\
 & \kappa^{R,Y}_{N,T,\xi}(g,h,X)\mathbf{1}_{T\backslash G_1,\leqslant C}(g) \mathbf{1}_{T\backslash G_1,\leqslant C}(h) \mathbf{1}_{\mathfrak{g}_1^*,\leqslant C}(g^{-1}X_{\mathfrak{t}}h) dX \frac{dg}{\nu(g)^n} \nu(h)^n dh\omega(t)^{-1}dt
\end{aligned}\]

\noindent for all $N\geqslant N_0$, all $Y\in \mathcal{A}_{M_0,\theta}$, all $R\in \mathcal{F}^\theta(M)$ and all $C>0$, where

$$\displaystyle \kappa_{N,T,\xi}^{R,Y}(g,h,X):=\int_{T} \xi(aX)^{-1}\kappa_N(h^{-1}ag)\varphi_N\ast \Delta_R(H_{M,\theta}(a),g,h,X,Y)da$$

\noindent for all $g,h\in G_1$ and $X\in \mathfrak{t}^*\oplus \mathfrak{t}^\perp$. Then, by (6) we have

$$\displaystyle J_{N,T,\leqslant C}(f)=\sum_{R\in \mathcal{F}^\theta(M)} J^{R,Y}_{N,T,\leqslant C}(f)$$

\noindent for all $C>0$, all $N\geqslant N_0$ and all $Y\in \mathcal{A}_{M_0,\theta}$. Thus in order to prove (3), it suffices to establish the following two facts

\vspace{2mm}

\noindent (7) \hspace{5mm} There exists $c>0$ such that

$$\displaystyle \left\lvert J^{G,Y}_{N,T,\leqslant N^\epsilon}(f)-\frac{1}{2^d}J_{Y,T,\leqslant N^\epsilon}(f)\right\rvert\ll e^{-cN^\epsilon}$$

\hspace{10mm} for all $N\geqslant N_0$ and all $Y\in \mathcal{A}_{P_0,\theta}^+$ satisfying inequality (B).

\vspace{2mm}

\noindent (8) \hspace{5mm} For all $R\in \mathcal{F}^\theta(M)$ with $R\neq G$ and all $r>0$, we have

$$\displaystyle \left\lvert J^{R,Y}_{N,T,\leqslant N^\epsilon}(f)\right\rvert\ll N^{-r}$$

\hspace{10mm} for all $N\geqslant N_0$ and all $Y\in \mathcal{A}_{P_0,\theta}^+$ satisfying inequality (A).

\vspace{2mm}

\noindent From now on and until the end of the proof, when $N\geqslant 1$ is fixed, we will say that $(g,h,X)\in G_1^2\times \mathfrak{g}_1$ is {\it in the good range} if $X\in \mathfrak{t}^*\oplus \mathfrak{t}^\perp$ and we have the inequalities $\sigma_{T\backslash G_1}(g)\leqslant N^\epsilon$, $\sigma_{T\backslash G_1}(h)\leqslant N^\epsilon$ and $\sigma_{\mathfrak{g}_1^*}(g^{-1}X_{\mathfrak{t}}h)\leqslant N^\epsilon$.

\vspace{2mm}

\noindent\ul{Proof of (7)}: By Proposition \ref{prop estimate}(i), it suffices to show the existence of $c>0$ and $k>0$ such that

$$\displaystyle \left\lvert \kappa_{N,T,\xi}^{G,Y}(g,h,X)-\frac{1}{2^d}\widetilde{w}_{M,\theta}(g,h,X,Y)\right\rvert\ll e^{-cN^\epsilon}\sigma_{T\backslash G_1}(g)^k\sigma_{T\backslash G_1}(h)^k\sigma_{\mathfrak{g}_1^*}(g^{-1}X_{\mathfrak{t}}h)^k (1+\lvert Y\rvert)^k\leqno (9)$$

\noindent for all $N\geqslant N_0$, all $Y\in \mathcal{A}_{P_0,\theta}^+$ satisfying inequality (B) and all $(g,h,X)\in G_1\times G_1\times \mathfrak{g}_1$ in the good range.

\vspace{2mm}

\noindent By \ref{(G,M) and (G,M,theta)-orthogonal sets}(2) and the first property satisfied by the sequence $(\varphi_N)_{N\geqslant N_0}$, there exists $C>0$ such that for all $(g,h,X,Y)\in G_1\times G_1\times (\mathfrak{t}^*\oplus \mathfrak{t}^\perp)\times \mathcal{A}^+_{P_0,\theta}$ and all $N\geqslant 1$, we have

$$\displaystyle \varphi_N\ast \Gamma_{M,\theta}^G(H_{M,\theta}(a),\mathcal{Z}(g,h,X,Y))\neq 0 \Rightarrow \sigma_{G_1}(a)\leqslant C\left( \sigma_{G_1}(g)+\sigma_{G_1}(h)+\sigma_{\mathfrak{g}_1^*}(g^{-1}X_{\mathfrak{t}}h)+\lvert Y\rvert+N^\epsilon\right)$$

\vspace{1mm}

\noindent for all $a\in T$. Hence by Lemma \ref{lemma Arthur}(i), together with the inequality $\epsilon<\epsilon_2<1$, for $N$ sufficiently large, we have

$$\displaystyle \varphi_N\ast \Gamma_{M,\theta}^G(H_{M,\theta}(a),\mathcal{Z}(g,h,X,Y))\neq 0 \Rightarrow \kappa_N(h^{-1}ag)=1$$

\noindent for all $Y\in \mathcal{A}_{P_0,\theta}^+$ satisfying inequality (B) and all $(g,h,X)\in G_1\times G_1\times \mathfrak{g}_1$ in the good range. This implies, again for $N$ sufficiently large, that

\[\begin{aligned}\displaystyle \kappa_{N,T,\xi}^{G,Y}(g,h,X) & =\int_T \xi(aX)^{-1}\varphi_N\ast \Gamma_{M,\theta}^G(H_{M,\theta}(a),\mathcal{Z}(g,h,X,Y))da \\
 & =\int_T \widehat{\varphi}_N(aX_{\mathfrak{t}}) \Gamma_{M,\theta}^G(H_{M,\theta}(a),\mathcal{Z}(g,h,X,Y))da
\end{aligned}\]

\noindent for all $Y\in \mathcal{A}_{P_0,\theta}^+$ satisfying inequality (B) and all $(g,h,X)\in G_1\times G_1\times \mathfrak{g}_1$ in the good range. By (2) and the third property satisfied by the sequence $(\varphi_N)_{N\geqslant N_0}$, there exist $c>0$ and $k>0$ such that

\[\begin{aligned}
\displaystyle & \left\lvert\int_T \widehat{\varphi}_N(aX_{\mathfrak{t}}) \Gamma_{M,\theta}^G(H_{M,\theta}(a),\mathcal{Z}(g,h,X,Y))da-\int_T L(aX_{\mathfrak{t}}) \Gamma_{M,\theta}^G(H_{M,\theta}(a),\mathcal{Z}(g,h,X,Y))da\right\rvert \\
 & \leqslant e^{-cN^\epsilon}\int_T \lvert\Gamma_{M,\theta}^G(H_{M,\theta}(a),\mathcal{Z}(g,h,X,Y))\rvert da \ll e^{-cN^\epsilon} \sigma_{T\backslash G_1}(g)^k \sigma_{T\backslash G_1}(h)^k \sigma_{\mathfrak{g}_1^*}(g^{-1}X_{\mathfrak{t}}h)^k (1+\lvert Y\rvert)^k
\end{aligned}\]

\noindent for all $N\geqslant N_0$ and all $(g,h,X,Y)\in G_1\times G_1\times (\mathfrak{t}^*\oplus \mathfrak{t}^\perp)\times \mathcal{A}_{M_0,\theta}$. Thus in order to show (9), it only remains to establish the following identity

$$\displaystyle \int_T L(aX_{\mathfrak{t}}) \Gamma_{M,\theta}^G(H_{M,\theta}(a),\mathcal{Z}(g,h,X,Y))da=\frac{1}{2^d}\widetilde{w}_{M,\theta}(g,h,X,Y) \leqno (10)$$

\noindent for all $(g,h,X,Y)\in G_1\times G_1\times (\mathfrak{t}^*\oplus \mathfrak{t}^\perp)\times \mathcal{A}_{M_0,\theta}$. Fix such $(g,h,X,Y)$. After the variable change $a\mapsto aX_{\mathfrak{t}}^{-1}$, we get

$$\displaystyle \int_{T} L(aX_{\mathfrak{t}})\Gamma^G_{M,\theta}(H_{M,\theta}(a),\mathcal{Z}(g,h,X,Y)) da=\int_{T} L(a)\Gamma^G_{M,\theta}(H_{M,\theta}(a),\mathcal{Z}(g,h,X,Y)+\frac{1}{2}V(X_{\mathfrak{t}}))da$$

\noindent (Note that $H_{M,\theta}(X_{\mathfrak{t}})=\frac{1}{2}V(X_{\mathfrak{t}})$). Moreover, we easily check from the definitions that the $(G,M,\theta)$-orthogonal set $\mathcal{X}(g,h,X,Y):=\mathcal{Z}(g,h,X,Y)+\frac{1}{2}V(X_{\mathfrak{t}})$ has the property that $\mathcal{X}(g,h,X,Y)_{wP}=w\mathcal{X}(g,h,X,Y)_P$ for all $w\in W$ and all $P\in \mathcal{P}^\theta(M)$. This implies that the function $\Gamma^G_{M,\theta}(.,\mathcal{X}(g,h,X,Y))$ is $W$-invariant and by (5) it follows that

\[\begin{aligned}
\displaystyle \int_{T} L(a)\Gamma^G_{M,\theta}(H_{M,\theta}(a),\mathcal{X}(g,h,X,Y))da & =\frac{1}{\lvert W\rvert}\sum_{w\in W}\int_{T} L(wa) \Gamma^G_{M,\theta}(H_{M,\theta}(a),\mathcal{X}(g,h,X,Y))da \\
 & =\frac{1}{2^d}\int_{T} \Gamma^G_{M,\theta}(H_{M,\theta}(a),\mathcal{X}(g,h,X,Y))da \\
 & =\frac{1}{2^d}\int_{T} \Gamma^G_{M,\theta}(H_{M,\theta}(a),\mathcal{Z}(g,h,X,Y))da
\end{aligned}\]

\noindent where to get the last equality we have performed the variable change $a\mapsto aX_{\mathfrak{t}}$. This proves (10) and ends the proof of (7).

\vspace{2mm}

\noindent \ul{Proof of (8)}: Let $R\in \mathcal{F}^\theta(M)$ with $R\neq G$. We will need to rewrite slightly the weight $\kappa_{N,T,\xi}^{R,Y}(g,h,X)$. By Lemma \ref{lemma Arthur}(ii) and since the functions $(\varphi_N)_{N\geqslant N_0}$ are $T^c$-invariant, for $N$ big enough, all $Y\in \mathcal{A}_{M_0,\theta}$ and all $(g,h,X)\in G_1\times G_1 \times \mathfrak{g}_1$ in the good range, we have

$$\displaystyle \kappa_{N,T,\xi}^{R,Y}(g,h,X)=\vol(T^c)^{-1}\int_{T} \int_{T^c} \xi(a_0aX)^{-1}da_0 \kappa_N(h^{-1}ag)\varphi_N\ast\Delta_R(H_{M,\theta}(a),g,h,X,Y) da$$

\noindent or equivalently

$$\displaystyle \kappa_{N,T,\xi}^{R,Y}(g,h,X)=\int_{T} \widehat{\varphi^\circ}(aX_{\mathfrak{t}}) \kappa_N(h^{-1}ag)\varphi_N\ast\Delta_R(H_{M,\theta}(a),g,h,X,Y) da$$

\noindent where $\varphi^\circ:=\vol(T^c)^{-1}\mathbf{1}_{T^c}$.

\vspace{2mm}

\noindent In what follows, for all $a\in T$, we will write $a=(a_1,\ldots,a_d)$ for the decomposition of $a$ according to the identification \ref{concrete T}.(1) (so that $a_i\in K_i^\times$ for all $1\leqslant i\leqslant d$). Recall that we have identified $\mathcal{A}_{M,\theta}$ with $\mathbf{R}^d$ (see \S \ref{Computation}), the identification being so that

$$\displaystyle H_{M,\theta}(a)=\frac{1}{2}(v_{K_1}(a_1),\ldots,v_{K_d}(a_d))$$

\noindent for all $a\in T$ where $v_{K_i}:=v_F\circ N_{K_i/F}$ for all $1\leqslant i\leqslant d$. Moreover the set of roots of $A_{M,\theta}$ in $U_Q$, identified to a subset of $\mathcal{A}_{M,\theta}^*$, can be explicitly described as

$$\displaystyle R(A_{M,\theta},U_Q)=\left\{-\lambda_i-\lambda_j\mid 1\leqslant i,j\leqslant d \right\}$$

\noindent where for all $1\leqslant i\leqslant d$, we have denoted by $\lambda_i$ the functional $\mathcal{A}_{M,\theta}=\mathbf{R}^d\to \mathbf{R}$ given by $\lambda_i(x_1,\ldots,x_d)=f_ix_i$ with $f_i:=[K_i:F]$.

\noindent Write $R=L_RU_R$ for the unique Levi decomposition of $R$ with $M\subset L_R$ and let $\overline{R}=L_RU_{\overline{R}}$ be the parabolic subgroup opposite to $R$ (with respect to $L_R$). We now distinguish two cases:

\begin{itemize}
\item First assume that $U_{\overline{R}}$ is not included in the parabolic subgroup $Q$. Then, we will actually prove that for $N$ big enough, we have

$$\displaystyle \kappa_{N,T,\xi}^{R,Y}(g,h,X)=0$$

\noindent for all $Y\in \mathcal{A}_{P_0,\theta}^+$ satisfying inequality (A) and all $(g,h,X)\in G_1\times G_1\times \mathfrak{g}_1$ in the good range. This would implies that

$$\displaystyle J^{R,Y}_{N,T,\leqslant N^\epsilon}(f)=0$$

\noindent for $N$ big enough and all $Y\in \mathcal{A}_{P_0,\theta}^+$ satisfying inequality (A).

By our assumption, we have $U_{R}\cap U_{Q}\neq \{ 1\}$ which implies that $R(A_{M,\theta},U_{R})\cap R(A_{M,\theta},U_{Q})\neq \emptyset$. Let $\alpha \in R(A_{M,\theta},U_{R})\cap R(A_{M,\theta},U_{Q})$. Then, by the previous concrete description of $R(A_{M,\theta},U_Q)$, there exist $1\leqslant i,j\leqslant d$ such that

$$\displaystyle \langle \alpha,H_{M,\theta}(a)\rangle=-\frac{f_iv_{K_i}(a_i)+f_jv_{K_j}(a_j)}{2} \leqno (11)$$

\noindent for all $a\in T$. As $\epsilon<\epsilon_1$, there exists $c_1>0$ such that for all $N$ sufficiently large, all $Y\in \mathcal{A}_{P_0,\theta}^+$ satisfying inequality (A) and all $(g,h,X)\in G_1\times G_1\times \mathfrak{g}_1$ in the good range, we have $\langle \beta,\mathcal{Z}(g,h,X,Y)_P\rangle\geqslant c_1N^{\epsilon_1}$ for all $P\in \mathcal{P}^\theta(M)$ and all $\beta\in R(A_{M,\theta},U_P)$. In particular, $\mathcal{Z}(g,h,X,Y)_P\in \mathcal{A}_{P,\theta}^+$ for all $P\in \mathcal{P}^\theta(M)$ and the $(G,M,\theta)$-orthogonal set $\mathcal{Z}(g,h,X,Y)$ is positive. This implies that $\Delta_R(.,\mathcal{Z}(g,h,X,Y))$ is the characteristic function of the sum of $\mathcal{A}_{R,\theta}^+$ with the convex hull of the family $(\mathcal{Z}(g,h,X,Y)_P)_{P\subset R}$ and hence that

$$\displaystyle \Delta_R(\Lambda,\mathcal{Z}(g,h,X,Y))\neq 0\Rightarrow \langle \alpha,\Lambda\rangle\geqslant c_1N^{\epsilon_1}$$

\noindent for all $\Lambda\in \mathcal{A}_{M,\theta}$. Again since $\epsilon<\epsilon_1$ and by the first property satisfied by the sequence $(\varphi_N)_{N\geqslant N_0}$, it follows that there exists $c_2>0$ such that

$$\displaystyle \varphi_N\ast \Delta_R(\Lambda,\mathcal{Z}(g,h,X,Y))\neq 0\Rightarrow \langle \alpha,\Lambda\rangle\geqslant c_2N^{\epsilon_1}$$

\noindent for all $\Lambda\in \mathcal{A}_{M,\theta}$, all $N$ sufficiently large, all $Y\in \mathcal{A}_{P_0,\theta}^+$ satisfying inequality (A) and all $(g,h,X)\in G_1\times G_1\times \mathfrak{g}_1$ in the good range. On the other hand, since the function $\widehat{\varphi^\circ}$ is compactly supported on $\mathfrak{t}$, by (11) there exists $c_3>0$ such that

$$\displaystyle \widehat{\varphi^\circ}(aX_{\mathfrak{t}})\neq 0\Rightarrow \langle \alpha, H_{M,\theta}(a)\rangle \leqslant c_3 N^{\epsilon}$$

\noindent for all $N\geqslant 1$, all $a\in T$ and all $(g,h,X)\in G_1\times G_1\times \mathfrak{g}_1$ in the good range. As $\epsilon<\epsilon_1$, it follows that for $N$ sufficiently large, the supports of the functions

$$\displaystyle a\in T\mapsto \varphi_N\ast \Delta_R(H_{M,\theta}(a),\mathcal{Z}(g,h,X,Y))$$

\noindent and

$$\displaystyle a\in T\mapsto \widehat{\varphi^\circ}(aX_{\mathfrak{t}})$$

\noindent are disjoint and hence

$$\displaystyle \kappa_{N,T,\xi}^{R,Y}(g,h,X)=0$$

\noindent for all $Y\in \mathcal{A}_{P_0,\theta}^+$ satisfying inequality (A) and all $(g,h,X)\in G_1\times G_1\times \mathfrak{g}_1$ in the good range. This proves the claim and ends the proof of (8) in this case.

\item Now assume that $U_{\overline{R}}\subset Q$ or equivalently $U_Q\subset \overline{R}$. Let $V_{\overline{R}}$ and $V_{\overline{R}}^\perp$ be the subspaces of $\mathfrak{g}_1$ such that

$$\displaystyle U_{\overline{R}}\cap U_Q=\left\{\begin{pmatrix} 1 & X \\ & 1 \end{pmatrix}\mid X\in V_{\overline{R}} \right\},$$

$$\displaystyle L_R\cap U_Q=\left\{\begin{pmatrix} 1 & X \\ & 1 \end{pmatrix}\mid X\in V^\perp_{\overline{R}} \right\}$$

\noindent (we can show that $V_{\overline{R}}^\perp$ is the orthogonal of $V_{\overline{R}}$ with respect to $\langle .,.\rangle$ but we won't need it). Let $U_1$, $U_2$ be the unipotent subgroups of $G_1$ such that $U_{\overline{R}}\cap L=U_1\times U_2$ and let $L_1$ be the Levi subgroup of $G_1$ such that $L_R\cap L=L_1\times L_1$ (that $L_R\cap L$ is of this form follows from the fact that $R$ is $\theta$-split). Then, we have $\mathfrak{g}_1=V_{\overline{R}}\oplus V_{\overline{R}}^\perp$ (this follows from the fact that $U_Q=(L_R\cap U_Q)(U_{\overline{R}}\cap U_Q)$) and $u_1Xu_2^{-1}-X\in V_{\overline{R}}^\perp$ for all $(u_1,u_2)\in U_1\times U_2$ and all $X\in V_{\overline{R}}^\perp$ (this follows from the fact that $l^{-1}ulu^{-1}\in U_{\overline{R}}$ for all $(l,u)\in L_R\times U_{\overline{R}}$). We will denote by $X\mapsto X_{\overline{R}}^\perp$ the projection $\mathfrak{g}_1\to V_{\overline{R}}^\perp$ relative to the decomposition $\mathfrak{g}_1=V_{\overline{R}}\oplus V_{\overline{R}}^\perp$.

\vspace{2mm}

\noindent For all $(g,h,Y)\in G_1\times G_1\times \mathcal{A}_{P_0,\theta}^+$, we introduce a new $(G,M,\theta)$-orthogonal set $\mathcal{Z}(g,h,Y)$ defined by

$$\displaystyle \mathcal{Z}(g,h,Y)_{w\overline{P}}:=Y_{w\overline{P}}-wH_{P,\theta}\begin{pmatrix} g &\\ & h \end{pmatrix}$$

\noindent for all $P\in \mathcal{P}^{\theta,Q}(M)$ and all $w\in W$ (that it indeed defines a $(G,M,\theta)$-orthogonal set follows from Lemma \ref{lem limit of GMT family}). As before, we set

$$\displaystyle \Delta_R(\Lambda,\mathcal{Z}(g,h,Y)):=\Gamma^R_{M,\theta}(\Lambda,\mathcal{Z}(g,h,Y))\tau^G_{R,\theta}(\Lambda-\mathcal{Z}(g,h,Y)),\;\;\; \Lambda\in \mathcal{A}_{M,\theta}.$$

\noindent For all $N\geqslant N_0$ and all $Y\in \mathcal{A}_{P_0,\theta}^+$, we define a new weight by

$$\displaystyle \kappa_{N,T,\xi}^{R,Y,\star}(g,h,X):=\int_T \xi(aX_{\mathfrak{t}})^{-1}\kappa_N(h^{-1}ag) \varphi_N\ast \Delta_R(H_{M,\theta}(a),\mathcal{Z}(g,h,Y))da$$

\noindent for all $(g,h,X)\in G_1\times G_1\times \mathfrak{g}_1$. We will now need the following lemma whose proof is postponed to the next section:

\begin{lem}\label{lem tech}
For $N$ big enough, we have

$$\displaystyle \kappa_{N,T,\xi}^{R,Y}(u_1l_1k_1,u_2l_2k_2,X)=\kappa_{N,T,\xi}^{R,Y,\star}(l_1,l_2,X_{\overline{R}}^\perp)$$

\noindent for all $(u_1,u_2)\in U_1\times U_2$, $(l_1,l_2)\in L_1\times L_1$, $(k_1,k_2)\in K_1\times K_1$, $X\in \mathfrak{g}_1$ and $Y\in \mathcal{A}_{P_0,\theta}^+$ such that $Y$ satisfies inequality (A) and $(u_1l_1k_1,u_2l_2k_2,X)$ is in the good range.
\end{lem}

\noindent Assuming the above lemma, by the Iwasawa decompositions $G_1=U_1L_1K_1=U_2L_1K_1$ and since $K_1\times K_1\subset K$, for $N$ sufficiently large, we have

\[\begin{aligned}
\displaystyle J^{R,Y}_{N,T,\leqslant N^\epsilon}(f)= & \int_{T} D^{G_1}(t) \int_{T\times T\backslash L_1\times L_1} \int_{U_1\times U_2} \int_{\mathfrak{g}_1}  \\
 & f^K\left(\begin{pmatrix} l_1^{-1}u_1^{-1} & \\ & l_2^{-1}u_2^{-1} \end{pmatrix} \begin{pmatrix} t & \\ & t\end{pmatrix} \begin{pmatrix} 1 & X \\ & 1\end{pmatrix} \begin{pmatrix} u_1l_1 & \\ & u_2l_2\end{pmatrix}\right) \\
 & \kappa^{R,Y,\star}_{N,T,\xi}(l_1,l_2,X_{\overline{R}}^\perp)\mathbf{1}_{T\backslash G_1,\leqslant N^\epsilon}(u_1l_1) \mathbf{1}_{T\backslash G_1,\leqslant N^\epsilon}(u_2l_2) \mathbf{1}_{\mathfrak{g}_1^*,\leqslant N^\epsilon}(l_1^{-1}u_1^{-1}X_{\mathfrak{t}}u_2l_2) \\
 & dXdu_1du_2\delta_{P_1}(l_1)^{-1}\frac{dl_1}{\nu(l_1)^n} \delta_{P_2}(l_2)^{-1}\nu(l_2)^n dl_2\omega(t)^{-1}dt
\end{aligned}\]

\noindent for all $Y\in \mathcal{A}_{P_0,\theta}^+$ satisfying inequality (A) where we have denoted by $\delta_{P_1}$ and $\delta_{P_2}$ the modular characters of the parabolic subgroups $P_1=L_1U_1$ and $P_2=L_1U_2$ of $G_1$ respectively.

\vspace{2mm}

\noindent Introduce the following expression:

\[\begin{aligned}
\displaystyle J_{N,T}^{R,Y,\star}(f):=& \int_{T} D^{G_1}(t) \int_{T\times T\backslash L_1\times L_1} \int_{U_1\times U_2} \int_{\mathfrak{g}_1}  \\
 & f^K\left(\begin{pmatrix} l_1^{-1}u_1^{-1} & \\ & l_2^{-1}u_2^{-1} \end{pmatrix} \begin{pmatrix} t & \\ & t\end{pmatrix} \begin{pmatrix} 1 & X \\ & 1\end{pmatrix} \begin{pmatrix} u_1l_1 & \\ & u_2l_2\end{pmatrix}\right) \\
 & \kappa^{R,Y,\star}_{N,T,\xi}(l_1,l_2,X_{\overline{R}}^\perp) dX du_1du_2\delta_{P_1}(l_1)^{-1}\frac{dl_1}{\nu(l_1)^n} \delta_{P_2}(l_2)^{-1}\nu(l_2)^n dl_2\omega(t)^{-1}dt
\end{aligned}\]

\noindent for all $N\geqslant N_0$ and all $Y\in \mathcal{A}_{P_0,\theta}^+$. By \ref{def of an expression}.(3), for some $k>0$, we have

$$\displaystyle \left\lvert \kappa^{R,Y,\star}_{N,T,\xi}(l_1,l_2,X_{\overline{R}}^\perp)\right\rvert\ll \sigma_{T\backslash G_1}(u_1l_1k_1)^k\sigma_{T\backslash G_1}(u_2l_2k_2)^k$$

\noindent for all $(l_1,l_2)\in L_1\times L_1$, $(u_1,u_2)\in U_1\times U_2$ and $(k_1,k_2)\in K_1\times K_1$. Hence, by Proposition \ref{prop estimate} (and using the Iwasawa decompositions $G_1=U_1L_1K_1=U_2L_1K_1$ backwards), the expression defining $J_{N,T}^{R,Y,\star}(f)$ is absolutely convergent. And for all $r>0$, we have

$$\displaystyle \left\lvert J^{R,Y}_{N,T,\leqslant N^\epsilon}(f) -J_{N,T}^{R,Y,\star}(f)\right\rvert\ll N^{-r}$$

\noindent for all $N\geqslant N_0$ and all $Y\in \mathcal{A}_{P_0,\theta}^+$ satisfying inequality (A). Thus to prove (8) in this case, it suffices to establish that

$$\displaystyle J_{N,T}^{R,Y,\star}(f)=0$$

\noindent for all $N\geqslant N_0$ and all $Y\in \mathcal{A}_{P_0,\theta}^+$. This vanishing follows from the fact that $f$ is $\theta$-strongly cuspidal by the following sequence of equalities:

\[\begin{aligned}
\displaystyle & \int_{U_1\times U_2} \int_{\mathfrak{g}_1}  f^K\left(\begin{pmatrix} l_1^{-1}u_1^{-1} & \\ & l_2^{-1}u_2^{-1} \end{pmatrix} \begin{pmatrix} t & \\ & t\end{pmatrix} \begin{pmatrix} 1 & X \\ & 1\end{pmatrix} \begin{pmatrix} u_1l_1 & \\ & u_2l_2\end{pmatrix}\right) \kappa^{R,Y,\star}_{N,T,\xi}(l_1,l_2,X_{\overline{R}}^\perp) dX du_1du_2 \\
 & =\int_{U_1\times U_2} \int_{\mathfrak{g}_1}  f^K\left(\begin{pmatrix} l_1^{-1} & \\ & l_2^{-1} \end{pmatrix} \begin{pmatrix} u_1^{-1}tu_1 & \\ & u_2^{-1}tu_2\end{pmatrix} \begin{pmatrix} 1 & X \\ & 1\end{pmatrix} \begin{pmatrix} l_1 & \\ & l_2\end{pmatrix}\right) \kappa^{R,Y,\star}_{N,T,\xi}(l_1,l_2,X_{\overline{R}}^\perp) dX du_1du_2 \\
 & =\Delta(t)\int_{U_1\times U_2} \int_{\mathfrak{g}_1}  f^K\left(\begin{pmatrix} l_1^{-1} & \\ & l_2^{-1} \end{pmatrix} \begin{pmatrix} tu_1 & \\ & tu_2\end{pmatrix} \begin{pmatrix} 1 & X \\ & 1\end{pmatrix} \begin{pmatrix} l_1 & \\ & l_2\end{pmatrix}\right) \kappa^{R,Y,\star}_{N,T,\xi}(l_1,l_2,X_{\overline{R}}^\perp) dX du_1du_2 \\
 & =\Delta(t) \int_{V_{\overline{R}}^\perp} \int_{U_{\overline{R}}} f^K\left(\begin{pmatrix} l_1^{-1} & \\ & l_2^{-1} \end{pmatrix} \begin{pmatrix} t & \\ & t\end{pmatrix} U\begin{pmatrix} 1 & X_{\overline{R}}^\perp \\ & 1\end{pmatrix} \begin{pmatrix} l_1 & \\ & l_2\end{pmatrix}\right) \kappa^{R,Y,\star}_{N,T,\xi}(l_1,l_2,X_{\overline{R}}^\perp) dUdX_{\overline{R}}^\perp \\
 & =0
\end{aligned}\]

\noindent for all $N\geqslant N_0$, all $t\in T\cap G_{1,reg}$, all $l_1,l_2\in L_1$ and all $Y\in \mathcal{A}_{P_0,\theta}^+$, where $\Delta(t):=D^{G_1}(t)^{-1}D^{L_1}(t)$ and where the first equality follows from the variable change $X\mapsto u_1Xu_2^{-1}$ (recall that $(u_1Xu_2^{-1})_{\overline{R}}^\perp=X_{\overline{R}}^\perp$), the second equality is a consequence of the classical change of variable inverse to $(u_1,u_2)\mapsto (t^{-1}u_1^{-1}tu_1,t^{-1}u_2^{-1}tu_2)$ (whose jacobian equals $D^{G_1/L_1}(t)^{-1}$ because $\delta_{P_1}(t)=\delta_{P_2}(t)^{-1}$), in the third equality we have merge an integral over $U_1\times U_2$ and an integral over $V_{\overline{R}}$ into an integral over $U_{\overline{R}}$, and finally the last equality follows from the fact that $f$ is $\theta$-strongly cuspidal (note that $\begin{pmatrix} 1 & X_{\overline{R}}^\perp \\ & 1\end{pmatrix}$ belongs to $L_R$) so that the inner integral already vanish identically. This finishes the proof of (8) in this case and the proof of the proposition. $\blacksquare$
\end{itemize}

\subsection{Proof of Lemma \ref{lem tech}}

For convenience we recall the statement here.

\begin{lem}
Let $0<\epsilon<\epsilon_1<1$ and let $R\in \mathcal{F}^\theta(M)$ with $U_{\overline{R}}\subset Q$. Then, for $N$ big enough, we have

$$\displaystyle \kappa_{N,T,\xi}^{R,Y}(u_1l_1k_1,u_2l_2k_2,X)=\kappa_{N,T,\xi}^{R,Y,\star}(l_1,l_2,X_{\overline{R}}^\perp)$$

\noindent for all $(u_1,u_2)\in U_1\times U_2$, $(l_1,l_2)\in L_1\times L_1$, $(k_1,k_2)\in K_1\times K_1$, $X\in \mathfrak{g}_1$ and $Y\in \mathcal{A}_{P_0,\theta}^+$ satisfying the inequalities
\begin{itemize}
\item $N^{\epsilon_1}\leqslant \inf_{\alpha\in \Delta_0} \alpha(Y)$;
\item $\sigma_{T\backslash G_1}(u_1l_1)\leqslant N^\epsilon$, $\sigma_{T\backslash G_1}(u_2l_2)\leqslant N^\epsilon$;
\item $\sigma_{\mathfrak{g}_1^*}(l_1^{-1}u_1^{-1}X_{\mathfrak{t}}u_2l_2)\leqslant N^\epsilon$.
\end{itemize}
\end{lem}

\noindent\ul{Proof}: We will use all the notations already introduced in the proof of Proposition \ref{prop change of weight}. In particular, we have a decomposition $\mathfrak{g}_1=V_{\overline{R}}\oplus V_{\overline{R}}^\perp$. Recall that the linear map $X\mapsto X_{\overline{R}}^\perp$ is just the projection onto $V_{\overline{R}}^\perp$ relative to this decomposition. We will set $V_{\overline{R},\mathfrak{t}}:=V_{\overline{R}}\cap \mathfrak{t}$ and $V_{\overline{R},\mathfrak{t}}^\perp:=V_{\overline{R}}^\perp\cap \mathfrak{t}$. Then, we have $\mathfrak{t}=V_{\overline{R},\mathfrak{t}}\oplus V_{\overline{R},\mathfrak{t}}^\perp$ (this follows from the fact that the adjoint action of $T$ preserves the decomposition $\mathfrak{g}_1=V_{\overline{R}}\oplus V_{\overline{R}}^\perp$). For all $X\in \mathfrak{g}_1$, we will set $X_{\overline{R},\mathfrak{t}}^\perp:=(X_{\overline{R}}^\perp)_{\mathfrak{t}}=(X_{\mathfrak{t}})_{\overline{R}}^\perp$. Note that we have $X_{\mathfrak{t}}-X_{\overline{R},\mathfrak{t}}^\perp\in V_{\overline{R},\mathfrak{t}}$ for all $X\in \mathfrak{g}_1$.

\vspace{2mm}

\noindent By the definition of $\kappa_{N,T,\xi}^{R,Y}$, it is invariant by right translation by $K_1$ in both the first and the second variables. Thus, we just need to prove the lemma in the case where $k_1=k_2=1$. Moreover, since $\kappa_{N,T,\xi}^{R,Y}(g,h,X)=\kappa_{N,T,\xi}^{R,Y}(t_1g,t_2h,t_1Xt_2^{-1})$ for all $(g,h,X,t_1,t_2)\in G_1^2\times \mathfrak{g}_1\times T^2$, by \ref{geom side}(1) we may replace the conditions that $\sigma_{T\backslash G_1}(u_1l_1)\leqslant N^\epsilon$ and $\sigma_{T\backslash G_1}(u_2l_2)\leqslant N^\epsilon$ by the conditions that $\sigma_{G_1}(u_1l_1)\leqslant N^\epsilon$ and $\sigma_{G_1}(u_2l_2)\leqslant N^\epsilon$. Note that there exists $c>0$ such that for all $(l_1,l_2,u_1,u_2,X)\in L_1^2\times U_1\times U_2\times \mathfrak{g}_1$, the inequalities $\sigma_{G_1}(u_1l_1)\leqslant N^\epsilon$, $\sigma_{G_1}(u_2l_2)\leqslant N^\epsilon$ and $\sigma_{\mathfrak{g}_1^*}(l_1^{-1}u_1^{-1}X_{\mathfrak{t}}u_2l_2)\leqslant N^\epsilon$ imply that

$$\displaystyle \max\left(\sigma_{G_1}(l_1),\sigma_{G_1}(l_2),\sigma_{G_1}(u_1),\sigma_{G_1}(u_2),\sigma_{\mathfrak{t}^*}(X_{\mathfrak{t}}) \right)\leqslant cN^\epsilon.$$

\noindent Fix such a $c$. From now on and until the end of the proof, for a fixed $N$, we will say that $(l_1,l_2,u_1,u_2,X,Y)\in L_1^2\times U_1\times U_2\times \mathfrak{g}_1\times \mathcal{A}_{P_0,\theta}^+$ is in the {\it good range} if it satisfies the inequality above and moreover $N^{\epsilon_1}\leqslant \inf_{\alpha\in \Delta_0} \alpha(Y)$. Then in order to prove the lemma, we just need to establish the following fact:

\vspace{4mm}

\flushleft{(1)}\hspace{5mm} For $N$ big enough and all $(l_1,l_2,u_1,u_2,X,Y)$ in the good range, we have

$$\displaystyle \kappa_{N,T,\xi}^{R,Y}(u_1l_1,u_2l_2,X)=\kappa_{N,T,\xi}^{R,Y,\star}(l_1,l_2,X_{\overline{R}}^\perp).$$

\vspace{4mm}

\noindent First of all, by Lemma \ref{lemma Arthur}(ii) and since the functions $(\varphi_N)_{N\geqslant N_0}$ are $T^c$-invariant, for $N$ sufficiently large and all $(l_1,l_2,u_1,u_2,X,Y)$ in the good range, we have

$$\displaystyle \kappa_{N,T,\xi}^{R,Y}(u_1l_1,u_2l_2,X)=\int_T \widehat{\varphi^\circ}(a X_{\mathfrak{t}}) \kappa_N(l_2^{-1}u_2^{-1}au_1l_1) \varphi_N\ast \Delta_R(H_{M,\theta}(a),\mathcal{Z}(u_1l_1,u_2l_2,X,Y))da$$

\noindent and

$$\displaystyle \kappa_{N,T,\xi}^{R,Y,\star}(l_1,l_2,X_{\overline{R}}^\perp)=\int_T \widehat{\varphi^\circ}(a X_{\overline{R},\mathfrak{t}}^\perp) \kappa_N(l_2^{-1}al_1) \varphi_N\ast \Delta_R(H_{M,\theta}(a),\mathcal{Z}(l_1,l_2,Y))da$$

\noindent where $\varphi^\circ:=\vol(T^c)^{-1}\mathbf{1}_{T^c}$. Therefore, (1) is a consequence of the two following facts:

\vspace{4mm}

\flushleft{(2)}\hspace{5mm} For $N$ sufficiently large and all $(l_1,l_2,u_1,u_2,X,Y)$ in the good range, we have

$$\displaystyle \kappa_N(l_2^{-1}u_2^{-1}au_1l_1)=\kappa_N(l_2^{-1}al_1)$$

\hspace{12mm} and

$$\displaystyle \widehat{\varphi^\circ}(a X_{\mathfrak{t}})=\widehat{\varphi^\circ}(a X_{\overline{R},\mathfrak{t}}^\perp)$$

\hspace{12mm} for all $a\in T$ with $\varphi_N\ast \Delta_R(H_{M,\theta}(a),\mathcal{Z}(l_1,l_2,X,Y))\neq 0$.

\vspace{4mm}

\flushleft{(3)}\hspace{5mm} For $N$ sufficiently large and all $(l_1,l_2,u_1,u_2,X,Y)$ in the good range, we have

$$\displaystyle \varphi_N\ast \Delta_R(H_{M,\theta}(a),\mathcal{Z}(u_1l_1,u_2l_2,X,Y))=\varphi_N\ast \Delta_R(H_{M,\theta}(a),\mathcal{Z}(l_1,l_2,Y))$$

\hspace{12mm} for all $a\in T$ such that $\widehat{\varphi^\circ}(aX_{\mathfrak{t}})\neq 0$.

\vspace{4mm}

\noindent\ul{Proof of (2)}: As $\overline{R}$ is $\theta$-split, the parabolic subgroups $P_1=L_1U_1$ and $P_2=L_2U_2$ of $G_1$ are opposite to each other. Therefore, $R(A_T,U_2)=-R(A_T,U_1)$. Let us denote by $R(A_T,V_{\overline{R},\mathfrak{t}})$ the set of roots of $A_T$ in $V_{\overline{R},\mathfrak{t}}$ for the linear action $(a,X)\mapsto aX$. Then, we will need the following

\vspace{4mm}

\flushleft{(4)}\hspace{5mm} For all root $\alpha\in R(A_T,U_1)$ (resp. $\alpha\in R(A_T,V_{\overline{R},\mathfrak{t}})$), there exists a root

\vspace{0.4pt}

\hspace{11mm} $\beta\in R(A_{M,\theta},U_{\overline{R}})$ such that

$$\displaystyle \langle \alpha,H_T(a)\rangle=2\langle \beta, H_{M,\theta}(a)\rangle \; (\mbox{resp. }\langle \alpha,H_T(a)\rangle=\langle \beta, H_{M,\theta}(a)\rangle)$$

\hspace{11mm} for all $a\in T$.

\vspace{4mm}

\noindent Let $\mathfrak{u}_{1,\alpha}$ (resp. $V_{\overline{R},\mathfrak{t},\alpha}$) be the root subspace corresponding to $\alpha$. Then, we easily check that $A_M=A_T\times A_T$ acts on this root subspace by a character $\gamma$ (in the case where $\alpha\in R(A_T,V_{\overline{R},\mathfrak{t}})$, we consider the action given by $(a_1,a_2).X=a_1Xa_2^{-1}=a_1a_2^{-1}X$). Denoting by $\beta$ the restriction of $\gamma$ to $A_{M,\theta}$ (so that in particular $\beta\in R(A_{M,\theta},U_{\overline{R}})$), we have

$$\displaystyle \langle \beta, H_{M,\theta}(a)\rangle=\frac{1}{2}\langle \beta, H_{M,\theta}\left(\begin{pmatrix} a & \\ & 1 \end{pmatrix}\theta \begin{pmatrix} a & \\ & 1 \end{pmatrix}^{-1}\right)\rangle=\frac{1}{2}\langle \beta, H_{M,\theta}\begin{pmatrix} a & \\ & a^{-1} \end{pmatrix}\rangle=\frac{1}{2}\log \left\lvert \beta\begin{pmatrix} a & \\ & a^{-1} \end{pmatrix}\right\rvert$$

\noindent for all $a\in T$. On the other hand, the action of $\begin{pmatrix} a & \\ & a^{-1} \end{pmatrix}$ on $\mathfrak{u}_{1,\alpha}$ (resp. $V_{\overline{R},\mathfrak{t},\alpha}$) coincides with the action of $a$ (resp. $a^2$). Hence, $\beta\begin{pmatrix} a & \\ & a^{-1} \end{pmatrix}=\alpha(a)$ (resp. $\beta\begin{pmatrix} a & \\ & a^{-1} \end{pmatrix}=\alpha(a)^2$) and this shows the claim.

\vspace{2mm}

\noindent Now, since $\epsilon<\epsilon_1$, there exists $c_1>0$ such that for $N$ sufficiently large and all $(l_1,l_2,u_1,u_2,X,Y)$ in the good range, the $(G,M,\theta)$-orthogonal set $\mathcal{Z}(l_1,l_2,X,Y)$ is positive and $\langle \beta, \mathcal{Z}(l_1,l_2,X,Y)_P\rangle\leqslant -c_1N^{\epsilon_1}$ for all $P\in \mathcal{P}^\theta(M)$ with $P\subset R$ and all $\beta\in R(A_{M,\theta},U_{\overline{R}})$. Since the $(G,M,\theta)$-orthogonal set $\mathcal{Z}(l_1,l_2,X,Y)$ is positive, the function $\Delta_R(.,\mathcal{Z}(l_1,l_2,X,Y))$ is the characteristic function of the sum of $\mathcal{A}_{R,\theta}^+$ with the convex hull of $(\mathcal{Z}(l_1,l_2,X,Y)_P)_{P\subset R}$ and it follows that $\langle \beta, \Lambda\rangle\leqslant -c_1N^{\epsilon_1}$ for all $\Lambda\in \mathcal{A}_{M,\theta}$ in the support of this function and all $\beta\in R(A_{M,\theta},U_{\overline{R}})$. Again by using the fact that $\epsilon<\epsilon_1$, we deduce from this and the first property satisfied by the sequence $(\varphi_N)_{N\geqslant N_0}$ (see Lemma \ref{lem varphi}), that there exists $c_2>0$ such that for $N$ sufficiently large and all $(l_1,l_2,u_1,u_2,X,Y)$ in the good range, we have $\langle \beta, \Lambda\rangle\leqslant -c_2N^{\epsilon_1}$ for all $\Lambda\in \mathcal{A}_{M,\theta}$ in the support of the function $\varphi_N\ast \Delta_R(.,\mathcal{Z}(l_1,l_2,X,Y))$ and all $\beta\in R(A_{M,\theta},U_{\overline{R}})$. By (4), it follows that we have, again for $N$ sufficiently large and all $(l_1,l_2,u_1,u_2,X,Y)$ in the good range,

$$\displaystyle \langle \alpha, H_T(a)\rangle\leqslant -c_2N^{\epsilon_1}$$

\noindent for all $\alpha\in R(A_T,U_1)\cup R(A_T,V_{\overline{R},\mathfrak{t}})$ and all $a\in T$ such that

$$\varphi_N\ast \Delta_R(H_{M,\theta}(a),\mathcal{Z}(l_1,l_2,X,Y))\neq 0$$

\noindent Still assuming that $(l_1,l_2,u_1,u_2,X,Y)$ is in the good range, this implies (recall that $R(A_T,U_2)=-R(A_T,U_1)$ and $X_{\mathfrak{t}}-X_{\overline{R},\mathfrak{t}}^\perp\in V_{\overline{R},\mathfrak{t}}$)

$$\displaystyle l_2^{-1}u_2^{-1}au_1a^{-1}u_2l_2\in K_1, \;\; l_1^{-1}a^{-1}u_2^{-1}al_1\in K_1 \mbox{ and } a(X_{\mathfrak{t}}-X_{\overline{R},\mathfrak{t}}^\perp)\in \mathcal{L}$$

\noindent for $N$ sufficiently large, where $\mathcal{L}$ is a lattice in $\mathfrak{t}$ by which $\widehat{\varphi^\circ}$ is invariant. Hence, we get

$$\displaystyle \kappa_N(l_2^{-1}u_2^{-1}au_1l_1)=\kappa_N(l_2^{-1}u_2^{-1}al_1)=\kappa_N(l_2^{-1}al_1)$$

\noindent and

$$\displaystyle \widehat{\varphi^\circ}(aX_{\mathfrak{t}})=\widehat{\varphi^\circ}(aX_{\overline{R},\mathfrak{t}}^\perp)$$

\noindent for all $N$ sufficiently large, all $(l_1,l_2,u_1,u_2,X,Y)$ in the good range and all $a\in T$ such that $\varphi_N\ast \Delta_R(H_{M,\theta}(a),\mathcal{Z}(l_1,l_2,X,Y))\neq 0$.

\vspace{2mm}

\noindent\ul{Proof of (3)}: Since the function $\widehat{\varphi^\circ}$ is compactly supported on $\mathfrak{t}$, there exists $c_3>0$ such that for all $N$, all $(l_1,l_2,u_1,u_2,X,Y)$ in the good range and all $a\in T$, if $\widehat{\varphi^\circ}(aX_{\mathfrak{t}})\neq 0$, then $v_{K_i}(a_i)\geqslant -c_3 N^{\epsilon}$ for all $1\leqslant i\leqslant d$ (recall that we are denoting by $(a_1,\ldots,a_d)$ the decomposition of $a$ according to the identification \ref{concrete T}.(1)). These last inequalities mean that $H_{M,\theta}(a)\in \mathbf{R}_+^d+c_3Z_N$ where $Z_N:=(-N^{\epsilon},\ldots,-N^\epsilon)$ (recall that we are identifying $\mathcal{A}_{M,\theta}$ with $\mathbf{R}^d$ as in \S \ref{Computation}). Moreover, by the first property satisfied by the sequence $(\varphi_N)_{N\geqslant N_0}$ (see Lemma \ref{lem varphi}), there exists $c_4>0$ such that, for all $N\geqslant N_0$ and all function $\Delta$ on $\mathcal{A}_{M,\theta}$, the restriction of the function $\varphi_N\ast \Delta$ to $c_3Z_N+\mathbf{R}_+^d$ only depends on the restriction of $\Delta$ to $c_4Z_N+\mathbf{R}_+^d$. Note that

$$\displaystyle \mathbf{R}_+^d=\bigcup_{P\in \mathcal{P}^{Q,\theta}(M)} \overline{\mathcal{A}_{\overline{P},\theta}^+}$$

\noindent (this follows from the fact that $\mathbf{R}_+^d=\{ \Lambda\in\mathcal{A}_{M,\theta}\mid \langle \alpha,\Lambda\rangle\leqslant 0\; \forall \alpha\in R(A_{M,\theta},U_Q) \}$). Hence, it suffices to show that for all $P\in \mathcal{P}^{Q,\theta}(M)$, all $N$ sufficiently large and all $(l_1,l_2,u_1,u_2,X,Y)$ in the good range, the two functions

$$\displaystyle \Delta_R(.,\mathcal{Z}(u_1l_1,u_2l_2,X,Y))$$

\noindent and

$$\displaystyle \Delta_R(.,\mathcal{Z}(l_1,l_2,Y))$$

\noindent coincide on $c_4Z_N+\overline{\mathcal{A}_{\overline{P},\theta}^+}$. Or equivalently that the two functions

$$\displaystyle \Delta_R(.,\mathcal{Z}(u_1l_1,u_2l_2,X,Y)-c_4Z_N)$$

\noindent and

$$\displaystyle \Delta_R(.,\mathcal{Z}(l_1,l_2,Y)-c_4Z_N)$$

\noindent coincide on $\overline{\mathcal{A}_{\overline{P},\theta}^+}$. Fix $P\in \mathcal{P}^{Q,\theta}(M)$. As $\epsilon<\epsilon_1$, for $N$ sufficiently large and all $(l_1,l_2,u_1,u_2,X,Y)$ in the good range, we have

$$\mathcal{Z}(u_1l_1,u_2l_2,X,Y)_{P'}-c_4Z_N\in \mathcal{A}_{P',\theta}^+ \mbox{ and } \mathcal{Z}(l_1,l_2,Y)_{P'}-c_4Z_N\in \mathcal{A}_{P',\theta}^+$$

\noindent for all $P'\in \mathcal{P}^\theta(M)$. By \ref{(G,M) and (G,M,theta)-orthogonal sets}(4), it follows that the restrictions of the functions

$$\Delta_R(.,\mathcal{Z}(u_1l_1,u_2l_2,X,Y)-c_4Z_N) \mbox{ and } \Delta_R(.,\mathcal{Z}(l_1,l_2,Y)-c_4Z_N)$$

\noindent to $\overline{\mathcal{A}_{\overline{P},\theta}^+}$ only depend on $\mathcal{Z}(u_1l_1,u_2l_2,X,Y)_{\overline{P}}$ and $\mathcal{Z}(l_1,l_2,Y)_{\overline{P}}$ respectively. Returning to the definitions, we see that

$$\displaystyle \mathcal{Z}(u_1l_1,u_2l_2,X,Y)_{\overline{P}}=Y_{\overline{P}}-H_{P,\theta}\begin{pmatrix} u_1l_1 & \\ & u_2l_2 \end{pmatrix}=Y_{\overline{P}}-H_{P,\theta}\begin{pmatrix} l_1 & \\ & l_2 \end{pmatrix}=\mathcal{Z}(l_1,l_2,Y)_{\overline{P}}$$

\noindent for all $(l_1,l_2,u_1,u_2,X,Y)\in L_1^2\times U_1\times U_2\times \mathfrak{g}_1\times \mathcal{A}_{P_0,\theta}^+$. Hence, for $N$ sufficiently large and all $(l_1,l_2,u_1,u_2,X,Y)$ in the good range, the functions $\Delta_R(.,\mathcal{Z}(u_1l_1,u_2l_2,X,Y)-c_4Z_N)$ and $\Delta_R(.,\mathcal{Z}(l_1,l_2,Y)-c_4Z_N)$ indeed coincide on $\overline{\mathcal{A}_{\overline{P},\theta}^+}$. This finishes the proof of (3) and  the proof of the lemma. $\blacksquare$

\subsection{Computation of $\lim\limits_{N\to \infty} J_{N,T}(f)$ for $\theta$-strongly cuspidal functions}\label{computation of the limit}

\begin{prop}\label{prop limit theta cuspidal}
Let $f\in \mathcal{C}(G)$ be a $\theta$-strongly cuspidal function. Then, we have

$$\displaystyle \lim\limits_{N\to \infty} J_{N,T}(f)=\frac{(-1)^d}{2^d} \int_{T} D^{G_1}(t)^2 \Phi_{M,\theta}\left(\begin{pmatrix} t & \\ & t\end{pmatrix}, f \right) \omega(t)^{-1} dt$$

\noindent where $\Phi_{M,\theta}\left(\begin{pmatrix} t & \\ & t\end{pmatrix}, f \right)$ is the 'singular $\theta$-weighted orbital integral' introduced in \S \ref{first application}.
\end{prop}

\noindent\ul{Proof}: By the same argument as in the proof of Proposition 4.4.1 of \cite{BeuGalP}, except that we replace Lemma 1.8.4 and Proposition 4.3.1 of {\it loc. cit.} by Proposition \ref{prop estimate} and Proposition \ref{prop change of weight} of this paper, we have

\[\begin{aligned}\displaystyle \lim\limits_{N\to \infty} J_{N,T}(f) & =\frac{1}{2^d}\int_T D^{G_1}(t)\int_{T\times T\backslash G_1\times G_1}\int_{\mathfrak{g}_1} f^K\left(\begin{pmatrix} g^{-1} & \\ & h^{-1} \end{pmatrix}\begin{pmatrix} t & \\ & t \end{pmatrix} \begin{pmatrix} 1 & X \\ & 1 \end{pmatrix}\begin{pmatrix} g & \\ & h \end{pmatrix}\right) \\
 & v_{M,\theta}(-\mathcal{Z}(g,h,X)) dX \frac{dg}{\nu(g)^n}\nu(h)^n dh \omega(t)^{-1} dt \\
 & =\frac{(-1)^d}{2^d}\int_T D^{G_1}(t)\int_{T\times T\backslash G_1\times G_1}\int_{\mathfrak{g}_1} f^K\left(\begin{pmatrix} g^{-1} & \\ & h^{-1} \end{pmatrix}\begin{pmatrix} t & \\ & t \end{pmatrix} \begin{pmatrix} 1 & X \\ & 1 \end{pmatrix}\begin{pmatrix} g & \\ & h \end{pmatrix}\right) \\
  & w_{M,\theta}(g,h,X) dX \frac{dg}{\nu(g)^n}\nu(h)^n dh \omega(t)^{-1} dt \\
  & =\frac{(-1)^d}{2^d} \int_{T} D^{G_1}(t)^2 \Phi_{M,\theta}\left(\begin{pmatrix} t & \\ & t\end{pmatrix}, f \right) \omega(t)^{-1} dt.
\end{aligned}\] $\blacksquare$

\subsection{Proof of the geometric side}\label{proof of the geometric side}

The geometric side of Theorem \ref{theo trace formula} now follows immediately from \ref{def of an expression}(1), Proposition \ref{prop limit theta cuspidal} and the following proposition:

\begin{prop}
Assume that the function $f\in \mathcal{C}(G)$ is strongly cuspidal. Then,

\begin{enumerate}[(i)]
\item If $T$ is not elliptic in $G_1$ (i.e. $d>1$), then

$$\displaystyle \Phi_{M,\theta}\left(\begin{pmatrix} t & \\ & t\end{pmatrix}, f \right)=0$$

\noindent for all $t\in T\cap G_{1,reg}$.

\item If $T$ is elliptic in $G_1$ (i.e. $d=1$), then

$$\displaystyle \Phi_{M,\theta}\left(\begin{pmatrix} t & \\ & t\end{pmatrix}, f \right)=-2c_f\begin{pmatrix} t & \\ & t\end{pmatrix}$$

\noindent for all $t\in T\cap G_{1,reg}$.
\end{enumerate}
\end{prop}

\noindent\ul{Proof}:
\begin{enumerate}[(i)]
\item By Proposition \ref{prop first application}, it suffices to show that

$$\displaystyle \Phi_{M,\theta}\left(\begin{pmatrix} t_1 & \\ & t_2\end{pmatrix}, f \right)=0$$

\noindent for all $\begin{pmatrix} t_1 & \\ & t_2\end{pmatrix}\in (T\times T)\cap G_{reg}$. This follows from the descent formula \ref{(G,M) and (G,M,theta)-orthogonal sets}(6) which gives

$$\displaystyle \Phi_{M,\theta}\left(\begin{pmatrix} t_1 & \\ & t_2\end{pmatrix}, f \right)=\sum_{L\in \mathcal{L}(M)} d^G_{M,\theta}(L) \Phi^Q_M\left(\begin{pmatrix} t_1 & \\ & t_2\end{pmatrix}, f \right);$$

\noindent and the fact that for $Q\neq G$, since $f$ is strongly cuspidal, we have (see \S \ref{cusp forms and theta-strongly cuspidal functions})

$$\displaystyle \Phi^Q_M\left(\begin{pmatrix} t_1 & \\ & t_2\end{pmatrix}, f \right)=0$$

\noindent for all $\begin{pmatrix} t_1 & \\ & t_2\end{pmatrix}\in (T\times T)\cap G_{reg}$ (Note that here $d^G_{M,\theta}(G)=0$ since $\mathcal{A}_M^{G,\theta}\neq 0$).

\item By Proposition \ref{prop first application} and \ref{cusp forms and theta-strongly cuspidal functions}(1), it suffices to notice that

$$\displaystyle \Phi_{M,\theta}\left(\begin{pmatrix} t_1 & \\ & t_2\end{pmatrix}, f \right)=\Phi_{M}\left(\begin{pmatrix} t_1 & \\ & t_2\end{pmatrix}, f \right), D^G\begin{pmatrix} t & \\ & t \end{pmatrix}^{1/2}=D^{G_1}(t)^2 \mbox{ and } \lvert W(G_x,T_{x,qd})\rvert=2$$

\noindent for all $\begin{pmatrix} t_1 & \\ & t_2\end{pmatrix}\in (T\times T)\cap G_{reg}$ and all $t\in T\cap G_{1,reg}$ where we have set $x=\begin{pmatrix} t & \\ & t \end{pmatrix}$.
$\blacksquare$
\end{enumerate}

\section{Applications}\label{section Applications}
In this section, we will give applications of the multiplicity formula we proved in Proposition \ref{prop multiplicities}. In Section \ref{application 1}, we study the behavior of the multiplicities under the local Jacquet-Langlands correspondences (i.e. over Vogan $L$-packets). We will show that the multiplicities are constant on every (extended) discrete L-packet. Then in Section \ref{application 2}, we will study the relations between the generalized Shalika model and the Ginzburg-Rallis model. We keep the notations introduced in Section \ref{Shalika triples}.

\subsection{The multiplicities over L-packets}\label{application 1}
In this subsection, we are going to prove the main theorems (Theorem \ref{main theorem 1} and Theorem \ref{main theorem 2}) of this paper. Let $\mathcal{A}'$ be another degree $n$ central simple algebra over $F$. Set $G':=GL_2(\mathcal{A}')$ and define subgroups $H_0'$, $N'$, $H':=H_0'\rtimes N'$ analogous to the subgroups $H_0$, $N$ and $H$ of $G$. We define a character $\xi':N'\to \mathbf{C}^\times$ in the same way as the character $\xi$ of $N$ and we extend it to a character of $H'$ trivial on $H_0'$. We also identify the character $\omega:F^\times\to \mathbf{C}^\times$ with a character of $H'$ by composition with projection $H'\to H_0'$ and the reduced norm. This gives us a character $\omega\otimes \xi'$ of $H'$. For all $\pi'\in \Irr(G')$, we set

$$\displaystyle m(\pi', \omega):=\dim\Hom_{H'}(\pi',\omega\otimes \xi').$$

\noindent We want to prove the following theorem.

\begin{thm}\label{main 2}
Let $\pi\in \Pi_2(G)$ and $\pi'\in \Pi_2(G')$ which correspond to each other under the local Jacquet-Langlands correspondence (see \cite{DKV84}). Then, we have

$$\displaystyle m(\pi,\omega)=m(\pi',\omega).$$
\end{thm}

\begin{proof}
If the central characters of $\pi$ and $\pi'$ does not coincide with $\omega^n$, then by looking at the action of the center, we see that $m(\pi,\omega)=m(\pi',\omega)=0$. Assume now that the central characters of $\pi$ and $\pi'$ are equal to $\omega^n$. Set $G_1:=\mathcal{A}^\times$ and $G_1':=(\mathcal{A}')^\times$. Then, by the multiplicity formula of Proposition \ref{prop multiplicities}, we have

$$\displaystyle m(\pi,\omega)=\sum_{T\in \CT_{ell}(G_1)}\lvert W(G_1,T)\rvert^{-1}\int_{A_{G_1}\backslash T} D^{G_1}(t)^2 c_{\pi}\begin{pmatrix} t & \\ & t\end{pmatrix} \omega(t)^{-1}dt,$$

$$\displaystyle m(\pi',\omega)=\sum_{T'\in \CT_{ell}(G'_1)}\lvert W(G'_1,T')\rvert^{-1}\int_{A_{G_1'}\backslash T'} D^{G'_1}(t')^2 c_{\pi'}\begin{pmatrix} t' & \\ & t'\end{pmatrix} \omega(t')^{-1}dt'.$$

There is a natural bijection $\CT_{ell}(G_1)\simeq \CT_{ell}(G_1')$ which sends $T\in \CT_{ell}(G_1)$ to the unique torus $T'\in \CT_{ell}(G_1')$ such that $T\simeq T'$. Moreover, such an isomorphism can be obtained as conjugation by an element in $G_{1}(\overline{F})\simeq G'_{1}(\overline{F})$ where $G_{1}(\overline{F})$ and $G_{1}'(\overline{F})$ denote the groups of points of $G_1$, $G_1'$ over a fixed algebraic closure $\overline{F}$ of $F$ and the isomorphism $G_{1}(\overline{F})\simeq G_{1}'(\overline{F})$ is induced by an $\overline{F}$-isomorphism $\mathcal{A}\otimes_F \overline{F}\simeq \mathcal{A}'\otimes_F\overline{F}$. Let $T\in \CT_{ell}(G_1)$ and $T'\in \CT_{ell}(G_1')$ corresponding to each other and fix an isomorphism $T\simeq T'$ as before (i.e. induced from conjugation by an element over the algebraic closure). Then, conjugation by the same element induces an isomorphism $W(G_1,T)\simeq W(G_1',T')$. Let $t\in T$ and denote by $t'$ its image in $T'$. We easily check that $D^{G_1}(t)^2=D^{G_1'}(t')^2$ and $\omega(t)=\omega(t')$. Moreover, if $t\in T\cap G_{1,reg}$, then $t'\in T'\cap G_{1,reg}'$; and by \ref{representations}(2), we have

$$\displaystyle c_{\pi}\begin{pmatrix} t & \\ & t \end{pmatrix}=\frac{1}{2}D^{G} \begin{pmatrix}t & \\ & t \end{pmatrix}^{-1/2} \lim\limits_{\lambda\in F^{\times}\to 1} D^G\begin{pmatrix}t & \\ & \lambda t\end{pmatrix}^{1/2} \Theta_{\pi} \begin{pmatrix}t & \\ & \lambda t\end{pmatrix},$$

$$\displaystyle c_{\pi'}\begin{pmatrix} t' & \\ & t'\end{pmatrix}=\frac{1}{2}D^{G'}\begin{pmatrix}t' & \\ & t'\end{pmatrix}^{-1/2}  \lim\limits_{\lambda\in F^{\times}\to 1} D^{G'}\begin{pmatrix}t' & \\ & \lambda t'\end{pmatrix}^{1/2} \Theta_{\pi'}\begin{pmatrix} t' & \\ & \lambda t'\end{pmatrix}.$$

\noindent Once again we can easily check that

$$D^{G}\begin{pmatrix}t & \\ & t\end{pmatrix}=D^{G'}\begin{pmatrix}t' & \\ & t' \end{pmatrix},\;\;\; D^{G}\begin{pmatrix}t & \\ & \lambda t\end{pmatrix}=D^{G'}\begin{pmatrix}t' & \\ & \lambda t'\end{pmatrix}$$

\noindent for all $\lambda\in F^\times$. Furthermore, by the relations of the distribution characters under the local Jacquet-Langlands correspondence (see \cite{DKV84}), we also have

$$\Theta_{\pi}\begin{pmatrix}t & \\ & \lambda t\end{pmatrix}=\Theta_{\pi'}\begin{pmatrix} t' & \\ & \lambda t'\end{pmatrix}$$

\noindent for all $\lambda\in F^\times$. This shows that

$$c_\pi\begin{pmatrix} t & \\ & t \end{pmatrix}=c_\pi\begin{pmatrix} t' & \\ & t' \end{pmatrix}.$$

\noindent All in all, we get the equality

\[\begin{aligned}
\displaystyle & \lvert W(G_1,T)\rvert^{-1}\int_{A_{G_1}\backslash T} D^{G_1}(t)^2 c_{\pi}\begin{pmatrix} t & \\ & t\end{pmatrix} \omega(t)^{-1}dt \\
 =& \lvert W(G'_1,T')\rvert^{-1}\int_{A_{G_1'}\backslash T'} D^{G'_1}(t')^2 c_{\pi'}\begin{pmatrix} t' & \\ & t'\end{pmatrix} \omega(t')^{-1}dt'.
\end{aligned}\]

\noindent Summing this over all $T\in \CT_{ell}(G_1)$, we get the desired equality.
\end{proof}

Then we study the relations between the multiplicity and the local exterior square L-function. Assume that $\mathcal{A}'=M_n(F)$, $\pi'\in \Pi_2(G')$ and $\omega=\mathbf{1}$ is the trivial character of $H_0'$ in which case we will simply set

$$\displaystyle m(\pi'):=m(\pi',\mathbf{1})$$

\noindent The following theorem seems to be well-known but apparently only half of its proof has appeared in the literature (by combining results of different sources). For sake of completeness, we provide a proof of the other half. The exterior square $L$-function appearing in the statement is the Artin local $L$-function defined through the local Langlands correspondence of Harris-Taylor \cite{HT}, Henniart \cite{He00} and Scholze \cite{Sc}.

\begin{thm}\label{L function}
With the notations above, the local exterior square $L$-function $L(s,\pi',\wedge^2)$ has a pole at $s=0$ if and only if $m(\pi')=1$.
\end{thm}

\begin{proof}
By the multiplicity one result of \cite{JR} (see \cite{CS} for a more general result), in order to prove the theorem, it is enough to show that the local exterior square $L$-function $L(s,\pi',\wedge^2)$ has a pole at $s=0$ if and only if $m(\pi')\neq 0$. This has been proved in \cite{JNY} when $\pi'$ is supercuspidal. When $\pi'$ is a discrete series, by Theorem 4.3, Corollary 4.4 of \cite{K11} and Corollary 1.4 of \cite{KR}, we know that if the local exterior square $L$-function $L(s,\pi',\wedge^2)$ has a pole at $s=0$, then $m(\pi')\neq 0$. Hence it remains to prove the other direction. We will follow the method introduced in \cite{JNY}.

\vspace{1em}
Assume that $m(\pi')\neq 0$. Then, of course, the central character of $\pi'$ is trivial. Let $\SO_{4n}(F)$ be the $F$-split special orthogonal group of rank $2n$, and let $P=MN$ be a Siegel parabolic subgroup of $\SO_{4n}(F)$ with Levi component $M$ isomorphic to $GL_{2n}(F)$. For $s\in \BC$, let $\Pi(s)=I_{P}^{SO_{4n}(F)}(\pi'\otimes \lvert\det\rvert^{s})$ be the normalized parabolic induction of $\pi'\otimes \lvert\det\rvert^s$ to $\SO_{4n}(F)$. By Theorem 3.1 of \cite{JQ} and Proposition 2.3 of \cite{JNY}, together with the assumption that $m(\pi')\neq 0$, we know that the representation $\Pi(\frac{1}{2})$ of $SO_{4n}(F)$ is reducible. Hence it remains to prove the following lemma.

\begin{lem}
With the notations above, the representation $\Pi(\frac{1}{2})$ is reducible if and only if then the local exterior square $L$-function $L(s,\pi',\wedge^2)$ has a pole at $s=0$.
\end{lem}

\begin{proof}
Let $N_M\subset M$ be the unipotent radical of a Borel subgroup of $M$ and set $N_0:=N_M\ltimes N$. Then, $N_0$ is the unipotent radical of a Borel subgroup of $SO_{4n}(F)$. Let $\xi_0$ be a generic character on $N_0$ and set $\xi_M:=\xi_0\mid_{N_M}$ for the restriction of $\xi_0$ to $N_M$. As $\pi'$ is square-integrable it is in particular generic with respect to $(N_M,\xi_M)$. By Lemma B.2 of \cite{GI} and the Standard Module Conjecture \cite{HO}, we have that $\Pi(\frac{1}{2})$ is reducible if and only if $\gamma^{Sh}(2s,\pi',\wedge^2,\psi)$ has a pole at $s=1/2$ where $\gamma^{Sh}(s,\pi',\wedge^2,\psi)$ denotes the gamma factor defined by Shahidi (with respect to any non-trivial additive character $\psi$ since it doesn't change the existence or not of a pole at $s=1/2$). Since $\pi'$ is square-integrable, $L(s,\pi',\wedge^2)$ has no pole with positive real part and in particular at $s=1$. Hence, by the main result of \cite{He10}, $\Pi(\frac{1}{2})$ is reducible if and only if $L(s,\pi',\wedge^2)$ has a pole at $s=0$.
\end{proof}
\end{proof}

As the Jacquet-Langlands correspondence preserves $L$-function, we get a Corollary from the two theorems above which proves Theorem \ref{main theorem 2}.

\begin{cor}
For all $\pi\in \Pi_2(G)$, the local exterior square $L$-function $L(s,\pi,\wedge^2)$ has a pole at $s=0$ if and only if $m(\pi)=1$.
\end{cor}

\begin{rmk}
If $\omega$ is the square of a character of $F^\times$, then by twisting $\pi$ by a suitable character of $G$ together with the above corollary, we see that for $\pi\in \Pi_2(G)$, the local twisted exterior square $L$-function $L(s,\pi,\wedge^2\otimes \omega^{-1})$ has a pole at $s=0$ if and only if $m(\pi,\omega)=1$. The same should be true in general although it would necessitate to redo the works of \cite{JNY}, \cite{JQ}, \cite{K11} and \cite{KR} by introducing this twist. However, note that in the case where $n=2$, this more general result has been proved in Theorem 1.5 of \cite{GT10}.
\end{rmk}

\subsection{The relations between the generalized Shalika model and the Ginzburg-Rallis model}\label{application 2}

In this subsection, we will consider the particular cases where $\mathcal{A}=M_2(F)$ or a non-split quaternion algebra $D$ over $F$. More precisely, we will give some relations between the multiplicities for the Ginzburg-Rallis model of $\GL_3(D)$ (resp. $GL_6(F)$) and for the generalized Shalika model for $\GL_2(D)$ (resp. $GL_4(F)$).
\vspace{1em}

Assume $\omega$ unitary. Let $\pi_D$ be a tempered representation of $\GL_2(D)$ with central character $\omega^2$ and let $\Pi_D$ be the normalized parabolic induction of the representation $\pi_D\boxtimes \omega^{-1}$ of the Lei subgroup $GL_2(D)\times GL_1(D)$ to $GL_3(D)$ (where again we have identified $\omega^{-1}$ with a character of $GL_1(D)$ by composition with the reduced norm). Then $\Pi_D$ is a tempered representation of $\GL_3(D)$ with trivial central character. Let $m_{GR}(\Pi_D)$ be the multiplicity of the Ginzburg-Rallis model (with trivial character). We refer the readers to the previous papers \cite{Wan15} and \cite{Wan16a} of the second author for the definition of the Ginzburg-Rallis model. The following theorem tells us a relation between the Ginzburg-Rallis model for $\GL_3(D)$ and the generalized Shalika model for $\GL_2(D)$.

\begin{thm}\label{GR 1}
With the notations above, we have
$$m(\pi_D,\omega)=m_{GR}(\Pi_D).$$
\end{thm}

\begin{proof}
In Section 8.1 of \cite{Wan16a}, we proved a multiplicity formula $m_{GR}(\Pi_D)=m_{geom}(\Pi_D)$ for the Ginzburg-Rallis model. By the definition of $m_{geom}(\Pi_D)$ in \cite{Wan16a}, together with the definition of $\Pi_D$ above and Lemma 2.3 of \cite{WalGGPII}, we can easily show that
$$\displaystyle m_{geom}(\Pi_D)=\sum_{T\in \CT_{ell}(\GL_1(D))}|W(\GL_1(D),T)|^{-1}\int_{T/A_G} D^{\GL_1(D)}(t)^2 c_{\pi_D}\begin{pmatrix}t & \\ & t\end{pmatrix} \omega(t)^{-1}dt.$$
Hence by Proposition \ref{prop multiplicities}, we get $m(\pi_D)=m(\Pi_D)$, and this proves the theorem.
\end{proof}

Next, we will prove a relation between the Ginzburg-Rallis model for $\GL_6(F)$ and the Shalika model for $\GL_4(F)$. Let $\pi$ be a discrete series of $\GL_4(F)$ with central character $\omega^2$, and let $\pi_D$ be the Jacquet-Langlands lift of $\pi$ to $GL_2(D)$. Let $St(\omega^{-1})$ be the Steinberg representation of $\GL_2(F)$ twisted by the character $\omega^{-1}$ and let $\Pi$ be the normalized parabolic induction of the representation $\pi\boxtimes St(\omega^{-1})$ of the Levi subgroup $GL_4(F)\times GL_2(F)$ to $GL_6(F)$. Then $\Pi$ is a tempered representation of $\GL_6(F)$ with trivial central character. Let $m_{GR}(\Pi)$ denote the multiplicity of $\Pi$ relative to the Ginzburg-Rallis model (with trivial character) of $GL_6(F)$.

\begin{thm}\label{GR 2}
With the notations above, we have
$$m(\pi,\omega)+m_{GR}(\Pi)=1.$$
\end{thm}

\begin{proof}
Let $\Pi_D$ be the Jacquet-Langlands lift of $\Pi$ to $\GL_3(D)$. Then $\Pi_D$ is the parabolic induction of the representation $\pi_D\otimes \omega^{-1}$ of the Levi subgroup $\GL_2(D)\times \GL_1(D)$. By Theorem \ref{GR 1} and Theorem \ref{main 2}, we have
$$m_{GR}(\Pi_D)=m(\pi_D,\omega)=m(\pi,\omega).$$
On the other hand, in \cite{Wan15} and \cite{Wan16a}, we have proved that $m_{GR}(\Pi)+m_{GR}(\Pi_D)=1$. The theorem follows.
\end{proof}

For the rest part of this section, we are going to prove some partial results for the epsilon dichotomy conjecture of the Ginzburg-Rallis model. Let $\Pi$ be an irreducible generic representation of $\GL_6(F)$ with central character $\beta^2$ where $\beta$ is some character on $F^{\times}$, and let $m_{GR}(\Pi)$ be the multiplicity relative to the Ginzburg-Rallis model with character $\beta$. We first recall the epsilon dichotomy conjecture from \cite{Wan16a}.

\begin{conj}\label{epsilon}
With the notations above, we have
\begin{eqnarray*}
m_{GR}(\Pi)=1 &\iff& \epsilon(1/2,(\wedge^3\Pi)\otimes \beta^{-1})=1,\\
m_{GR}(\Pi)=0 &\iff& \epsilon(1/2,(\wedge^3\Pi)\otimes \beta^{-1})=-1.
\end{eqnarray*}
\end{conj}
In \cite{Wan16a}, we have proved the conjecture when $\Pi$ is not a discrete series, or the parabolic induction of a discrete series of $\GL_4(F)\times \GL_2(F)$. Now we are going to prove the conjecture when $\Pi$ is the parabolic induction of a discrete series $\pi\boxtimes \pi_0$ of $\GL_4(F)\times \GL_2(F)$ with $\pi_0$ being a twist of the Steinberg representation of $\GL_2(F)$.

Let $\pi_0=St(\alpha)$ for some character $\alpha:F^{\times}\rightarrow \BC^{\times}$. Then the central character of $\pi_0$ is ${\alpha}^2$, which implies that the central character of $\pi$ is ${\alpha}^{-2}{\beta}^{2}$. Let $\phi_{\pi}$ (resp. $\phi_{\pi_0}$) be the Langlands parameter of $\pi$ (resp. $\pi_0$). Then the Langlands parameter $\phi_{\Pi}$ of $\Pi$ is just $\phi_{\pi}\oplus \phi_{\pi_0}$.

\begin{prop}
With the notations above, we have
\begin{equation}\label{epsilon 1}
\epsilon(1/2,(\wedge^3\Pi)\otimes \beta^{-1})=\epsilon(1/2,\wedge^2(\phi_{\pi})\otimes \phi_{\pi_0}\otimes \beta^{-1}).
\end{equation}
\end{prop}

\begin{proof}
Since $\phi_{\Pi}=\phi_{\pi}\oplus \phi_{\pi_0}$, we have
$$\wedge^3(\phi_{\Pi})=(\wedge^2(\phi_{\pi})\otimes \phi_{\pi_0})\oplus (\phi_{\pi}\otimes \det(\phi_{\pi_0}))\oplus (\wedge^3(\phi_{\pi})).$$
Since the central character of $\pi$ is ${\alpha}^{-2}{\beta}^{2}$ and the central character of $\pi_0$ is ${\alpha}^2$, we have $(\wedge^3(\phi_{\pi}) \otimes \beta^{-1})^{\vee} =\phi_{\pi}\otimes {\alpha}^2\otimes \beta^{-1} =\phi_{\pi}\otimes \det(\phi_{\pi_0})\otimes \beta^{-1}$. This implies that
\begin{eqnarray*}
\epsilon(1/2,(\wedge^3\Pi)\otimes \beta^{-1})&=&\epsilon(1/2,\wedge^2(\phi_{\pi})\otimes \phi_{\pi_0}\otimes \beta^{-1})\epsilon(1/2,\phi_{\pi}\otimes \det(\phi_{\pi_0})\otimes \beta^{-1}) \\
&&\epsilon(1/2, \wedge^3(\phi_{\pi})\otimes \beta^{-1})\\
&=&\det(\wedge^3(\phi_{\pi})\otimes \beta^{-1})(-1)\times \epsilon(1/2,\wedge^2(\phi_{\pi})\otimes \phi_{\pi_0}\otimes \beta^{-1}) \\
&=& (\det(\phi_{\pi}))^3(-1) \beta^{-4}(-1)\times \epsilon(1/2,\wedge^2(\phi_{\pi})\otimes \phi_{\pi_0}\otimes \beta^{-1})\\
&=& \epsilon(1/2,\wedge^2(\phi_{\pi})\otimes \phi_{\pi_0}\otimes \beta^{-1}).
\end{eqnarray*}
This proves the proposition.
\end{proof}

Now up to twist $\pi_0$ by the character $\beta^{-1}$, we may assume that $\beta=1$. This is allowable since twist by characters will not change the multiplicity and the epsilon factor. In fact, by the proposition above, we know that the epsilon factor will not be changed if we twist $\pi_0$ by characters. As for the multiplicity, by Corollary 5.15 of \cite{Wan16a}, the multiplicity $m_{GR}(\Pi)$ for the Ginzburg-Rallis model is equal to the multiplicity $m_{MM}(\pi\otimes \pi_0)$ of the middle model (we refer the readers to Appendix A of \cite{Wan16a} for the definition of the middle model). But it is easy to see from the definition that the multiplicity of the middle model will not be changed if we twist $\pi_0$ by characters.

Assuming that $\beta=1$, by Theorem \ref{GR 2}, we have $m_{GR}(\Pi)+m(\pi,\alpha^{-1})=1$. Let $\pi_D$ be the Jacquet-Langlands lift of $\pi$ to $\GL_2(D)$. By Theorem \ref{main 2}, we have $m(\pi,\alpha^{-1})=m(\pi_D,\alpha^{-1})$. Combining with the proposition above, we see that in order to prove the epsilon dichotomy conjecture for $\Pi$, it is enough to prove that
\begin{eqnarray*}
m(\pi_D,\alpha^{-1})=0 &\iff& \epsilon(1/2,\wedge^2(\phi_{\pi})\otimes \phi_{\pi_0})=1,\\
m(\pi_D,\alpha^{-1})=1 &\iff& \epsilon(1/2,\wedge^2(\phi_{\pi})\otimes \phi_{\pi_0})=-1.
\end{eqnarray*}
But the above relations have already been proved in Theorem 1.5(2) of \cite{GT10}. This finishes the proof of the epsilon dichotomy conjecture for $\Pi$.

\section{A connection to the local r-trace formula}\label{r trace formula}
In this section, we will rewrite the local trace formula for the Shalika model in terms of the local r-trace formula. In Langlands' proposal \cite{L04} for beyond endoscopy, one of the most important ingredient is a global r-trace formula (the name ``r-trace formula" was first introduced by Arthur in his note \cite{A15} for beyond endoscopy). To be specific, let $G$ be a connected reductive group defined over a number field $k$, and let $r$ be a finitely dimensional algebraic representation of the L-group of $G$. For any automorphic representation $\pi$ of $G(\BA_k)$, let $L(s,\pi,r)$ be the global automorphic L-function. For a given test function $f$ on $G(\BA_k)$, we let $S_{cusp}^{1}(f)$ be the cuspidal part of the stable Arthur-Selberg trace formula. Then we let $S_{cusp}^{r}(f)$ be a generalization of $S_{cusp}^{1}(f)$, in which the stable multiplicities of representations $\pi$ that occur in $S_{cusp}^{1}(f)$ are weighted by the order of poles
$$m_{r}(\pi):=-ord_{s=1}(L(s,\pi,r))$$
at $s=1$ of the automorphic L-function $L(s,\pi,r)$. The goal of the r-trace formula is to find a decomposition of $S_{cusp}^{r}(f)$ in terms of stable distributions on some smaller groups $G'$. Here $G'$ should run over elliptic ``beyond endoscopic data". Like the theory of endoscopy, the most important step is to find a geometric expansion for the distribution $S_{cusp}^{r}(f)$.

Guided by the same philosophy, we can also consider the local r-trace formula. Let $G$ be a connected reductive group defined over a local field $F$. We still let $r$ be a finitely dimensional algebraic representation of the L-group of $G$, and let $L(s,\pi,r)$ be the local L-function. For a test function $f$ on $G$, we can define the local analogy of the distribution $S_{cusp}^{r}(f)$ to be
\begin{equation}\label{r.1}
I_{disc}^{r}(f):=\displaystyle \sum_{\pi\in \Pi_2(G)} m_{r}(\pi) \Tr(\pi(f))
\end{equation}
where $m_{r}(\pi):=-ord_{s=0}(L(s,\pi,r))$ is the order of the poles at $s=0$ of the L-function $L(s,\pi,r)$. Note that if $G$ has non-trivial split center, we need to include the central character in the trace formula. Same as the global case, we want to find a geometric expansion for $I_{disc}^{r}(f)$.
\vspace{1em}

For the rest part of this section, we will discuss a special case of the local r-trace formula. We consider the case when $F$ is a p-adic field, $G=\GL_{2n}(F)$, and $r=\wedge^2$ is the exterior square representation of the L-group ${}^L G=\GL_{2n}(\BC)$. For $f\in {}^{\circ}\mathcal{C}(G,1)$, by applying the local trace formula of Theorem \ref{theo trace formula} in the case where $\omega$ is trivial, we will give a geometric expansion for $I_{disc}^{r}(f)$. Let us denote by $I_{spec}(f)$ and $I_{geom}(f)$ the spectral and geometric sides of this trace formula respectively. Then the spectral side becomes

$$\displaystyle I_{spec}(f)=\sum_{\pi\in \Pi_{2}(G,1)} \Tr(\pi(f)) m(\pi^\vee).$$

\noindent where, fo simplicity, we have set $m(\pi^\vee):=m(\pi^\vee,\mathbf{1})$. As every $\pi\in \Pi_2(G,1)$ is unitary, it is easy to see that $m(\pi^\vee)=m(\pi)$. Moreover, by Theorem \ref{L function}, the multiplicity $m(\pi)$ is nonzero if and only if the exterior square L-function $L(s,\pi,\wedge^2)$ has a pole at $s=0$. Since $\pi$ is a discrete series, the order of the pole of $L(s,\pi,\wedge^2)$ at $s=0$ is either 0 or 1. In the mean time, we know that $m(\pi)\leq 1$ (see \cite{JR}). Therefore, for all discrete series $\pi\in \Pi_2(G,1)$, we have
$$m(\pi^\vee)=-ord_{s=0}(L(s,\pi,r))=m_{r}(\pi).$$
This implies that
$$\displaystyle I_{spec}(f)=\sum_{\pi\in \Pi_{2}(G,1)}  \Tr(\pi(f)) m_r(\pi).$$
In particular, we have
$$I_{disc}^{r}(f)=I_{spec}(f)=I_{geom}(f).$$
To conclude, the geometric side $I_{geom}(f)$ of the local trace formula for the Shalika model gives a geometric expansion for $I_{disc}^{r}(f)$.

\begin{rmk}
In general, the local multiplicity problem for many spherical pairs are closely related to the Langlands functoriality and the poles of some L-functions $L(s,\pi,r)$. Hence if we can prove the local trace formula for these models, we can find the geometric expansion of the corresponding local r-trace formulas.
\end{rmk}

\appendix

\section{Slight generalization of a result of M\oe{}glin and Waldspurger}\label{appendix}

Let $F$ be a nonarchimedean field of characteristic zero with ring of integer $\mathcal{O}_F$ and normalized absolute value $\lvert .\rvert$. We fix henceforth an uniformizer $\varpi_F\in \mathcal{O}_F$. As in the core of this paper, we will abuse notations by denoting algebraic groups over $F$ and the corresponding group of $F$-points by the same letter. We will do the same for Lie algebras.

\vspace{2mm}

\noindent Let $G$ be a connected reductive group over $F$ with Lie algebra $\mathfrak{g}$. Fix a $G$-invariant nondegenerate symmetric bilinear pairing $\langle .,.\rangle:\mathfrak{g}\times \mathfrak{g}\to F$ and a nontrivial additive character $\psi:F\to \mathbb{C}^\times$. Let $P_0=M_0N_0$ be a minimal parabolic subgroup of $G$ with unipotent radical $N_0$ and a fixed Levi component $M_0$. Let $\overline{P}_0=M_0\overline{N}_0$ be the opposite parabolic subgroup and $A_0\subset M_0$ the maximal central split torus. Let $\mathfrak{n}_0$, $\overline{\mathfrak{n}}_0$ and $\overline{\mathfrak{p}}_0$ denote the Lie algebra of $N_0$, $\overline{N}_0$ and $\overline{P}_0$ respectively. Let $\Delta_0$ be the set of simple roots of $A_0$ in $\mathfrak{n}_0$ and for all $\alpha\in \Delta_0$ fix a nonzero vector $Y_{-\alpha}$ in the root subspace $\mathfrak{g}_{-\alpha}$ corresponding to $-\alpha$. Set $Y:=\sum_{\alpha\in \Delta_0}Y_{-\alpha}\in \overline{\mathfrak{n}}_0$. It is well-known that

$$\displaystyle [Y,\overline{\mathfrak{p}}_0]=\overline{\mathfrak{n}}_0. \leqno (1)$$

\noindent The exponential map induces a regular isomorphism $\exp: \mathfrak{n}_0\simeq N_0$. Let $\log:N_0\to \mathfrak{n}_0$ be its inverse and set

$$\displaystyle \xi(n):=\psi(\langle Y,\log n\rangle),\;\;\; n\in N_0.$$

\noindent Then $\xi$ is a character of $N_0$ which is generic (i.e. with a stabilizer in $M_0$ of minimal dimension). We will denote by $M_{0,\xi}$ the stabilizer of $\xi$ in $M_0$ (i.e. the centralizer of $Y$ in $M_0$).

\vspace{2mm}

\noindent Let $\pi$ be a smooth irreducible representation of $G$ and $\Theta_\pi$ its Harish-Chandra character. Recall that for all semi-simple element $x\in G$ there is a local expansion (see \cite{HCH} Theorem 16.2)

$$\displaystyle \Theta_\pi(x \exp(X))=\sum_{\mathcal{O}\in Nil(\mathfrak{g}_x)} c_{\pi,\mathcal{O}}(x)\widehat{j}(\mathcal{O},X)$$

\noindent for $X\in \mathfrak{g}_{x,reg}$ sufficiently close to $0$.
\vspace{1em}

\noindent For all $x\in M_{0,\xi}$, denote by $\mathcal{O}_x$ the $G_x$-adjoint orbit of $Y$ in $\mathfrak{g}_x$. Then $\mathcal{O}_x$ is an element of $Nil(\mathfrak{g}_x)$ which is maximal for the order defined by $\mathcal{O}'\leqslant \mathcal{O}\Leftrightarrow \mathcal{O}'\subset \overline{\mathcal{O}}$ where $\overline{\mathcal{O}}$ denotes closure of $\mathcal{O}$ for the $F$-analytic topology.

\vspace{2mm}

\noindent Let $W$ be the space of $\pi$ and let $W(N_0,\xi)$ be the subspace generated by all vectors of the form $\pi(n)v-\xi(n)v$ for $v\in V$ and $n\in N_0$. The coinvariant space $W/W(N_0,\xi)$ carries a natural representation of $M_{0,\xi}$ that we shall denote by $\pi_{N_0,\xi}$. In \cite{MW} it was shown that (see \cite{Var} for the case of residual characteristic 2)

$$\displaystyle \dim \pi_{N_0,\xi}=c_{\pi,\mathcal{O}_1}.(1)$$

\noindent The goal of this appendix is to extend slightly this result by using the same techniques. More precisely we will prove:

\begin{prop}
For all $x\in M_{0,\xi}$ we have
$$\displaystyle \Tr \pi_{N_0,\xi}(x)=D^{G/M_0}(x)^{1/2}c_{\pi,\mathcal{O}_x}(x)$$
\noindent where we have set
$$\displaystyle D^{G/M_0}(x):=D^G(x)D^{M_0}(x)^{-1}.$$
\end{prop}

\vspace{2mm}

\noindent\ul{Proof}: In order to include the case of residual characteristic 2, we will use \cite{Var} as our main reference (which however follows very closely \cite{MW}). Let $x\in M_{0,\xi}$. Note that $Ad(x)$ is semi-simple and compact (indeed this follows from the facts that $M_0$ is anisotropic modulo $A_0$ and $M_{0,\xi}\cap A_0$ is contained in the center of $G$). To agree with the conventions of \cite{Var}, we will assume that $\psi$ is unramified (i.e. its conductor is $\mathcal{O}_F$). This is no loss in generality since we can always scale $\psi$ so that it become unramified and up to scaling $\langle .,.\rangle$ by the corresponding inverse factor the character $\xi$ doesn't change and the same holds for the coefficient $c_{\pi,\mathcal{O}_x}(x)$ as we easily check from our normalizations. Let $\varphi$ be the sum of the positive coroots of $A_0$ with respect to $P_0$. Then we have $\varphi(s)Y\varphi(s)^{-1}=s^{-2}Y$ for all $s\in F^\times$. Set $t:=\varphi(\varpi_F)$. The adjoint action of $\varphi(F^\times)$ on $\mathfrak{g}$ induces a decomposition
$$\displaystyle \mathfrak{g}=\bigoplus_{i\in \mathbb{Z}} \mathfrak{g}_i$$
\noindent where $\mathfrak{g}_i:=\{X\in \mathfrak{g};\; \varphi(s)X\varphi(s)^{-1}=s^iX\; \forall s\in F^\times \}$. Set $\mathfrak{g}^-:=\bigoplus_{i\leqslant 0} \mathfrak{g}_i$ and $\mathfrak{g}^2:=\bigoplus_{i\geqslant 2} \mathfrak{g}_i$. Then we have $\mathfrak{g}=\mathfrak{g}^-\oplus \mathfrak{g}^2$ and $\mathfrak{n}_0=\mathfrak{g}^2$(this is because $\mathfrak{g}_1=0$). Let $Y^\sharp$ be the centralizer of $Y$ in $\mathfrak{g}$ and $V$ be a complement subspace to $Y^\sharp$ which is invariant by $\varphi(F^\times)M_{0,\xi}$ (in particular it is invariant by $\varphi(F^\times)$). By $\varphi(F^\times)$-invariance, we have a decomposition $V=V^-\oplus V^2$ where $V^-:=V\cap\mathfrak{g}^-$ and $V^2:=V\cap\mathfrak{g}^2$. Note that
$$\displaystyle V^2=\mathfrak{g}^2=\mathfrak{n}_0. \leqno (2)$$
\noindent Indeed, this follows from (1) and the fact that $\langle .,.\rangle$ induces a perfect pairing between $\overline{\mathfrak{n}}_0$ and $\mathfrak{n}_0$.

\vspace{2mm}

\noindent Since $Ad(x)$ is semi-simple, we have a decomposition $\mathfrak{g}=\mathfrak{g}_x\oplus \mathfrak{g}^x$ where $\mathfrak{g}_x$ and $\mathfrak{g}^x$ stand respectively for the kernel and the image of $Ad(x)-1$ in $\mathfrak{g}$. As $V$ is $Ad(x)$-stable, we have a similar decomposition $V=V_x\oplus V^x$. The bilinear form

$$\displaystyle B_Y:(Z,X)\in \mathfrak{g}\times\mathfrak{g}\mapsto B_Y(Z,X):=\langle Y, [Z,X]\rangle$$

\noindent is alternating $M_{0,\xi}$-invariant and nondegenerate when restricted to $V$. Moreover, the decomposition $V=V_x\oplus V^x$ is orthogonal for this alternating form. Set $V_x^-:=V_x\cap \mathfrak{g}^-$, $V_x^2:=V_x\cap \mathfrak{g}^2$, $V^{x,-}:=V^x\cap \mathfrak{g}^-$ and $V^{x,2}:=V^x\cap \mathfrak{g}^2$. Then we have $V_x=V_x^-\oplus V_x^2$ and $V^x=V^{x,-}\oplus V^{x,2}$ (indeed this follows from the facts that $V$ is $\varphi(F^\times)$-stable and $x$ centralizes $\varphi(F^\times)$). Moreover, the form $B_Y$ induces a perfect pairing between $V_x^-$ and $V_x^2$ on the one hand and between $V^{x,-}$ and $V^{x,2}$ on the other hand. Since $Ad(x)$ is compact and commutes with $Ad(\varphi(F^\times))$, we can find a lattice of $L^{x,-}$ of $V^{x,-}$ which is $Ad(x)$-stable and such that

$$\displaystyle L^{x,-}=\bigoplus_{i\leqslant 0} L^{x,-}\cap \mathfrak{g}_i.$$

\noindent Let $L^{x,2}$ be the dual lattice of $V^{x,2}$ with respect to $B_Y$, that is

$$\displaystyle L^{x,2}:=\{X\in V^{x,2};\; B_Y(Z,X)\in \mathcal{O}_F\; \forall Z\in L^{x,-}  \}.$$

\noindent Similarly, we fix a lattice $L_x^-\subset V_x^-$ with the property that

$$\displaystyle L_x^-=\bigoplus_{i\leqslant 0} L_x^-\cap\mathfrak{g}_i$$

\noindent and denote by $L_x^2\subset V_x^2$ the dual lattice. Finally, we also fix a lattice $L_Y\subset Y^\sharp$ which is $Ad(x)$-invariant and such that $L_Y=\bigoplus_{i\in \mathbf{Z}}L_Y\cap \mathfrak{g}_i$, $L_Y=L_Y\cap \mathfrak{g}_x\oplus L_Y\cap \mathfrak{g}^x$. We set

$$\displaystyle L:=L_Y\oplus L_x^-\oplus L_x^2\oplus L^{x,-}\oplus L^{x,2}$$

\noindent Then, $L$ is a lattice of $\mathfrak{g}$ which is $Ad(x)$-invariant and satisfies properties (i) and (ii) of \S 3 of \cite{Var}. Moreover, by construction we have

$$\displaystyle L=L_x\oplus L^x$$

\noindent where $L_x:=L\cap \mathfrak{g}_x$ and $L^x:=L\cap \mathfrak{g}^x$.

\vspace{2mm}

\noindent For all integer $n$ sufficiently large we set

$$\displaystyle G_n:=\exp(\varpi_F^n L),\;\;\; G_n':=t^{-n}G_nt^n.$$

\noindent When $n$ is large enough, these are compact-open subgroups of $G$. Again for $n$ sufficiently large, we define two characters $\xi_n:G_n\to \mathbf{C}^\times$ and $\xi_n':G_n'\to \mathbf{C}^\times$ by

$$\displaystyle \xi_n(\exp(X)):=\psi(\varpi_F^{-2n}\langle Y,X+\frac{1}{2}[X_+,X_-]\rangle),\;\;\; X\in \varpi_F^n L;$$

$$\displaystyle \xi_n'(\gamma):=\xi_n(t^n\gamma t^{-n}),\;\;\; \gamma\in G_n'$$

\noindent where $X\mapsto X_-$ and $X\mapsto X_+$ denote respectively the projections onto $\mathfrak{g}^-$ and $\mathfrak{g}^2$ relative to the decomposition $\mathfrak{g}=\mathfrak{g}^-\oplus \mathfrak{g}^2$. These characters are $Ad(x)$-invariant (since $x$ is in the centralizer of both $Y$ and $t$). Moreover, we can easily check, using the Campbell-Hausdorff formula, that for $n$ large enough, the character $\xi_n$ coincide with the character $\chi_n$ constructed in Lemma 6 of \cite{Var}. For all $n$ for which $G'_n$ and $\xi'_n$ are defined, we set

$$\displaystyle W'_n:=\{v\in W\mid\; \pi(\gamma)v=\xi_n'(\gamma)v\; \forall \gamma\in G_n' \}.$$

\noindent These subspaces are invariant by $\pi(x)$ and by Lemma 8 and Lemma 9.(b) of \cite{Var}, when $n$ is large enough, the natural projection $W\twoheadrightarrow W/W(N_0,\xi)$ restricts to an isomorphism $W_n'\simeq W/W(N_0,\xi)$. From there it easily follows that
$$\displaystyle \Tr \pi_{N_0,\xi}(x)= \Tr \pi(x)_{\mid W_n'} \leqno (3)$$

\noindent for all $n$ large enough. Fix a Haar measure $dg$ on $G$ and for $n$ sufficiently large, let $\varphi_n, \varphi_n'\in C_c^\infty(G)$ be the functions defined by

$$\displaystyle \varphi_n(\gamma)=\left\{
    \begin{array}{ll}
        \frac{1}{\vol(G_n)} \xi_n(\gamma^{-1}), & \mbox{if } \gamma\in G_n; \\
        0, & \mbox{otherwise.}
    \end{array}
\right.
$$

$$\displaystyle \varphi_n'(\gamma):=\varphi_n(t^n \gamma t^{-n}).$$

\noindent Then $\pi(\varphi'_n)$ is a projection onto $W'_n$ and thus
$$\displaystyle \Tr \pi(x)_{\mid W_n'}=\Tr \pi(x)\pi(\varphi'_n)=\Tr \pi(x)\pi(t^{-n})\pi(\varphi_n)\pi(t^n)=\Tr \pi(L(x)\varphi_n) \leqno (4)$$

\noindent where $(L(x)\varphi_n)(\gamma):=\varphi_n(x^{-1}\gamma)$ for all $\gamma\in G$.

\vspace{2mm}

\noindent Fix open neighborhoods $\omega\subset \mathfrak{g}$ and $\Omega\subset G$ of $0$ and $1$ respectively such that the exponential map induces an $F$-analytic isomorphism $\exp:\omega\simeq \Omega$. Let $\log:\Omega\to \omega$ denote the inverse of this map. We fix a Haar measure on $\mathfrak{g}$ such that the exponential map preserves measures locally in a neighborhood of $0$ and Haar measures on $\mathfrak{g}^x$ and $\mathfrak{g}_x$ whose product is equal to the Haar measure on $\mathfrak{g}$. Choose $e\geqslant 1$ such that $(1-Ad(x^{-1}))L^x\supset \varpi_F^e L^x$ and $\langle Y,L\rangle \subset \varpi_F^{-e}\mathcal{O}_F$.  We can find an integer $B\geqslant 0$ satisfying the following conditions:

\begin{enumerate}[(1)]
\setcounter{enumi}{4}
\item The map

$$\displaystyle (\varpi_F^B L^x)\times (\varpi_F^B L_x)\to G:\;\; (Z,X)\mapsto \exp(-x^{-1}Zx)\exp(X)\exp(Z)$$

\noindent is an $F$-analytic isomorphism onto an open neighborhood of $1$ contained in $\Omega$ and the Jacobian of this map is constant equal to $D^G(x)$.

\item For all $m,m'\geqslant B$, all $Z\in \varpi_F^m L^x$ and all $X\in \varpi_F^{m'}L_x$ we have
$$\displaystyle \log\left(e^{-x^{-1}Zx}e^Xe^Z \right)\in X+Z-x^{-1}Zx+\frac{1}{2}[X,Z+x^{-1}Zx]+\frac{1}{2}[Z,x^{-1}Zx]+\varpi_F^{m+m'+e}L.$$

\item $\displaystyle [L,L]\subset 2\varpi_F^{2e+1-B} L$.

\end{enumerate}

\noindent Assume $n$ is large. By (3) and (4), we have

$$\displaystyle \Tr \pi_{N_0,\xi}(x)=\int_G \Theta_\pi(x \gamma)\varphi_n(\gamma)d\gamma.$$

\noindent And by (5), this equals

$$\displaystyle D^G(x)\int_{\varpi_F^B L^x}\int_{\varpi_F^B L_x} \Theta_\pi(x \exp(X)) \varphi_n(\exp(-x^{-1}Zx)\exp(X)\exp(Z)) dX dZ.$$

\noindent By (6) and (7), we check that for $Z\in \varpi_F^B L^x$ and $X\in \varpi_F^B L_x$, we have

$$\exp(-x^{-1}Zx)\exp(X)\exp(Z)\in G_n$$

\noindent if and only if $X\in \varpi_F^n L_x$ and $Z-x^{-1}Zx\in \varpi_F^n L^x$ and that in that case, we have

\[\begin{aligned}
\displaystyle \xi_n(\exp(-x^{-1}Zx)\exp(X)\exp(Z))= & \psi\bigg(\varpi_F^{-2n}\bigg\langle Y, X+Z-x^{-1}Zx+ \frac{1}{2}\big([X,Z+x^{-1}Zx]+ \\
 & [Z,x^{-1}Zx]+[X_++Z_+-x^{-1}Z_+x, X_-+Z_--x^{-1}Z_-x]\big)\bigg\rangle \bigg).
\end{aligned}\]

\noindent Since $x$ centralizes $Y$, we have $\langle Y,\mathfrak{g}^x\rangle=0$ and thus $Y$ is orthogonal to $Z-x^{-1}Zx$, $[X,Z+x^{-1}Zx]$, $[X_+, Z_--x^{-1}Z_-x]$ and $[Z_+-x^{-1}Z_+x,X_-]$. Similarly, $\langle Y,\mathfrak{g}_i\rangle=0$ for $i\neq 2$ so that $Y$ is orthogonal to $[Z_-,x^{-1}Z_-x]$ and $[Z_+,x^{-1}Z_+x]$. Finally, since $\langle .,.\rangle$ is $Ad(x)$-invariant and $x$ centralizes $Y$, we have $\langle Y,[x^{-1}Z_+,x^{-1}Z_-x]\rangle=\langle Y, [Z_+,Z_-]\rangle$. From all of these, we get by direct computation that the above expression equals

$$\displaystyle \xi_n(\exp(X))\psi\left(\varpi_F^{-2n}B_Y(Z_-,x^{-1}Z_+x-Z_+)\right).$$

\noindent So that finally we end up with

\[\begin{aligned}
\displaystyle \Tr \pi_{N_0,\xi}(x)= & D^G(x)\vol(L^x)^{-1}\int_{(1-Ad(x^{-1}))^{-1}L^x}\psi\left(B_Y(Z_-,x^{-1}Z_+x-Z_+) \right) dZ \\
 & \times \vol(\varpi_F^n L_x)^{-1} \int_{\varpi_F^nL_x}\xi_n(\exp(X))\Theta_\pi(x\exp(X))dX.
\end{aligned}\]

\noindent Note that the lattice $L_x$ of $\mathfrak{g}_x$ satisfies the assumptions (i) and (ii) of \S 3 of \cite{Var}, that is

\begin{itemize}
\item $L_x=\bigoplus_{i\in \mathbf{Z}} L_x\cap \mathfrak{g}_{x,i}$ where $\mathfrak{g}_{x,i}:=\mathfrak{g}_x\cap \mathfrak{g}_i$;

\item The lattice $L_x/(L_x\cap Y^\sharp)$ is self-dual with respect to $\psi\circ B_Y$.
\end{itemize}

\noindent Hence, the same computation as that of the proof of Lemma 7 of \cite{Var} shows that

$$\displaystyle \vol(\varpi_F^n L_x)^{-1} \int_{\varpi_F^nL_x}\xi_n(\exp(X))\Theta_\pi(x\exp(X))dX=c_{\pi,\mathcal{O}_x}(x)$$

\noindent when $n$ is large. Thus, it only remains to show that
$$\displaystyle D^G(x)\vol(L^x)^{-1}\int_{(1-Ad(x^{-1}))^{-1}L^x}\psi\left(B_Y(Z_-,x^{-1}Z_+x-Z_+) \right) dZ=D^{G/M_0}(x)^{1/2}. \leqno (8)$$

\noindent By construction, we have a decomposition $L^x=L^x_Y\oplus L^{x,-}\oplus L^{x,2}$ where $L^x_Y:=L^x\cap Y^\sharp$. Fix Haar measures on $\mathfrak{g}^x_Y:=\mathfrak{g}^x\cap Y^\sharp$, $V^{x,-}$ and $V^{x,2}$ whose product gives the (already fixed) Haar measure on $\mathfrak{g}^x$ through the decomposition $\mathfrak{g}^x=\mathfrak{g}^x_Y\oplus V^{x,-}\oplus V^{x,2}$. Up to scaling all these Haar measures, we may assume that

$$\displaystyle \vol(L^x)=\vol(L^x_Y)=\vol(L^{x,-})=\vol(L^{x,2})=1.$$

\noindent Then, as the function $Z\in \mathfrak{g}\mapsto \psi\left(B_Y(Z_-,x^{-1}Z_+x-Z_+)\right)$ is invariant by translation by $Y^\sharp$, we have

\[\begin{aligned}
\displaystyle  & \int_{(1-Ad(x^{-1}))^{-1}L^x}\psi\left(B_Y(Z_-,x^{-1}Z_+x-Z_+)\right) dZ \\
 & =\vol((1-Ad(x^{-1}))^{-1} L^x_Y)\int_{(1-Ad(x^{-1}))^{-1}L^{x,-}}\int_{(1-Ad(x^{-1}))^{-1}L^{x,2}} \psi\left(B_Y(Z_-,(1-Ad(x^{-1}))Z_+) \right) dZ_+ dZ_- \\
& =\left\lvert \det(1-Ad(x))_{\mid \mathfrak{g}^x_Y\oplus V^{x,2}}\right\rvert^{-1}\int_{(1-Ad(x^{-1}))^{-1}L^{x,-}}\int_{L^{x,2}} \psi\left(B_Y(Z_-,Z_+) \right) dZ_+ dZ_-
\end{aligned}\]

\noindent where to get the last line we have used the fact that if $\mathcal{V}$ is an $F$-vector space, $\mathcal{L}\subset \mathcal{V}$ a lattice and $T$ an endomorphism without the eigenvalue $1$ such that $T\mathcal{L}=\mathcal{L}$, then $\vol((1-T^{-1})^{-1}\mathcal{L})=\lvert \det(1-T)\rvert^{-1} \vol(\mathcal{L})$. Since $L^{x,2}$ is the lattice dual to $L^{x,-}$ with respect to $B_Y$ and $\psi$ unramified, we have

$$\displaystyle \int_{L^{x,2}} \psi\left(B_Y(Z_-,Z_+) \right) dZ_+=\left\{
    \begin{array}{ll}
        1, & \mbox{if } Z_-\in L^{x,-}; \\
        0, & \mbox{otherwise}
    \end{array}
\right.
$$

\noindent for all $Z_-\in V^{x,-}$. Thus, we have

$$\displaystyle \int_{(1-Ad(x^{-1}))^{-1}L^{x,-}}\int_{L^{x,2}} \psi\left(B_Y(Z_-,Z_+) \right) dZ_+ dZ_-= \vol(L^{x,-})=1$$

\noindent and consequently

\[\begin{aligned}
\displaystyle D^G(x)\int_{(1-Ad(x^{-1}))^{-1}L^x}\psi\left(B_Y(Z_-,x^{-1}Z_+x-Z_+) \right) dZ & =D^G(x) \left\lvert \det(1-Ad(x))_{\mid \mathfrak{g}^x_Y\oplus V^{x,2}}\right\rvert^{-1} \\
 & =\left\lvert \det(1-Ad(x))_{\mid V^{x,-}}\right\rvert
\end{aligned}\]

\noindent where for the last equality we have used the fact that

$$\displaystyle D^G(x)=\left\lvert \det(1-Ad(x))_{\mid \mathfrak{g}^x}\right\rvert=\left\lvert \det(1-Ad(x))_{\mid \mathfrak{g}^x_Y\oplus V^{x,2}\oplus V^{x,-}}\right\rvert.$$

\noindent Furthermore, since $V^{x,-}$ and $V^{x,2}$ are in duality under the form $B_Y$ which is $Ad(x)$-invariant, we have
$$\displaystyle \left\lvert \det(1-Ad(x))_{\mid V^{x,-}}\right\rvert=\left\lvert \det(1-Ad(x^{-1}))_{\mid V^{x,2}}\right\rvert=\left\lvert \det(1-Ad(x))_{\mid V^{x,2}}\right\rvert=\left\lvert \det(1-Ad(x))_{\mid \mathfrak{g}^x\cap \mathfrak{n}_0}\right\rvert$$

\noindent where in the second equality we have used the fact that $Ad(x)$ is compact and in the last one we have used (2). Similarly, since $\mathfrak{g}^x\cap \mathfrak{n}_0$ and $\mathfrak{g}^x\cap \overline{\mathfrak{n}}_0$ are in duality under the form $\langle .,.\rangle$ which is $Ad(x)$-invariant, we have

$$\displaystyle \left\lvert \det(1-Ad(x))_{\mid \mathfrak{g}^x\cap \mathfrak{n}_0}\right\rvert=\left\lvert \det(1-Ad(x))_{\mid \mathfrak{g}^x\cap (\mathfrak{n}_0\oplus \overline{\mathfrak{n}}_0)}\right\rvert^{1/2}=D^{G/M_0}(x)^{1/2}.$$

\noindent This shows (8) and ends the proof of the proposition. $\blacksquare$

\vspace{2mm}

The following corollary is a direct consequence of the last proposition.

\begin{cor}\label{cor appendix}
Set $H:=M_{0,\xi}\ltimes N_0$. Let $\tau$ be a smooth irreducible representation of $M_{0,\xi}$ (necessarily of finite dimension) and let $\chi_\tau$ be its character. Assume that the central characters of $\pi$ and $\chi_\tau$ coincide on the split center $A_G$ of $G$. Then, we have

$$\displaystyle \dim \Hom_{H}(\pi,\tau\otimes \xi)=\int_{M_{0,\xi}/A_G} D^{G/M_0}(x)^{1/2}c_{\pi,\mathcal{O}_x}(x) \chi_\tau(x^{-1})dx$$

\noindent where the Haar measure on $M_{0,\xi}/A_G$ is chosen so that $\vol(M_{0,\xi}/A_G)=1$.
\end{cor}

\flushright Rapha\"el Beuzart-Plessis \\
Universit\'e d'Aix-Marseille\\
I2M - CNRS (UMR 7373)\\
Campus de Luminy\\ 
13288 Marseille C\'edex 9, France \\
rbeuzart@gmail.com

\vspace{2mm}

Chen Wan \\
Fuld Hall 422, School of Mathematics \\
Institute for Advanced Study \\
1 Einstein Drive, Princeton, NJ 08540, USA \\
wanxx123@umn.edu

\begin{thebibliography}{99}
\bibitem{ArtInvTF}
J.~Arthur, \textit{The invariant trace formula I. Local theory}, J. Amer. Math. Soc. 1 (1988), no. 2, 323-383

\bibitem{ArthurlocalTF}
J.~Arthur, \textit{A local trace formula}, Inst. Hautes \'Etudes Sci. Publ. Math. No. 73 (1991), 5-96

\bibitem{A15}
J. Arthur,
{\it Problems beyond endoscopy.} To appear in Proceedings of Conference in Honor of the 70th birthday of Roger Howe.

\bibitem{BeuGGP0}
R.~Beuzart-Plessis, \textit{La conjecture locale de Gross-Prasad pour les repr\'esentations temp\'er\'ees des groupes unitaires}, M\'emoires de la SMF 149 (2016), 191p.

\bibitem{BeuGGP}
R.~Beuzart-Plessis, \textit{A local trace formula for the local Gan-Gross-Prasad conjecture for unitary groups: The archimedean case}, preprint 2015 arXiv: math/1506.01452

\bibitem{BeuGalP}
R.~Beuzart-Plessis, \textit{On distinguished square-integrable representations for Galois pairs and a conjecture of Prasad}, preprint 2017

\bibitem{Clo}
L.~Clozel, \textit{Invariant harmonic analysis on the Schwartz space of a reductive $p$-adic group}, in "Harmonic analysis on reductive groups" (Brunswick, ME, 1989), 101-121, Progr. Math., 101, Birkh\"auser Boston, Boston, MA, 1991

\bibitem{CS}
Fulin Chen, Binyong Sun,
{\it Uniqueness of twisted linear periods and twisted Shalika periods.} Preprint, 2017.


\bibitem{DKV84}
P. Deligne, D. Kazhdan, M.-F. Vign\'eras,
{\it Repr\'esentations des alg\`ebres centrales simples p-adiques.}  Repr\'esentations des groupes r\'eductifs sur un corps local, Travaux en Cours, Paris(1984): Hermann, pp. 33-117

\bibitem{Del}
P.~Delorme, \textit{Constant term of smooth $H_\psi$-spherical functions on a reductive $p$-adic group}, Trans. Amer. Math. Soc. 362 (2010), no. 2, 933-955

\bibitem{GI}
W. T.~Gan, A.~Ichino, {\it The Gross-Prasad conjecture and local theta correspondence}, Invent. Math. 206 (2016), no. 3, 705-799

\bibitem{GT10}
W. T.~Gan, S.~Takeda,
{\it On Shalika periods and a theorem of Jacquet-Martin.} American J. of Math. 132 (2010), 475-528.

\bibitem{GO}
M.~Gurevich, O.~Offen, {\it A criterion for integrability of matrix coefficients with respect to a symmetric pair}, preprint 2015

\bibitem{HCH}
Harish-Chandra, \textit{Admissible invariant distributions on reductive p-adic groups}, notes par S. DeBacker et P. Sally, University Lecture series 16, AMS (1999)

\bibitem{HT}
M.~Harris, R.~Taylor, {\it The geometry and cohomology of some simple Shimura varieties}, With an appendix by Vladimir G. Berkovich. Annals of Mathematics Studies, 151. Princeton University Press, Princeton, NJ, 2001. viii+276 pp. ISBN: 0-691-09090-4

\bibitem{HO}
V.~Heiermann, E.~Opdam, {\it On the tempered $L$-functions conjecture}, Amer. J. Math. 135 (2013), no. 3, 777-799

\bibitem{HW}
A. G.~Helminck, S.P.~Wang, \textit{On rationality properties of involutions of reductive groups}, Adv. Math. 99 (1993), no. 1, 26-96

\bibitem{He00}
G.~Henniart, {\it Une preuve simple des conjectures de Langlands pour $GL(n)$ sur un corps $p$-adique}, Invent. Math. 139 (2000), no. 2, 439-455.

\bibitem{He10}
G.~Henniart, {\it Correspondance de Langlands et fonctions $L$ des carr\'es ext\'erieur et sym\'etrique}, Int. Math. Res. Not. IMRN 2010, no. 4, 633-673


\bibitem{JNY}
Dihua Jiang, Chufeng Nien, Yijun Qin,
{\it Local Shalika models and functoriality.} Manuscripta Math. 127, 187--217 (2008)


\bibitem{JQ}
Dihua Jiang, Yijun Qin,
{\it Residues of Eisenstein series and generalized Shalika models for SO(4n).} J. of the Ramanujan Math. Soc. 22, No. 2 (2007) 1--33


\bibitem{JR}
H.~Jacquet, S.~Rallis, \textit{Uniqueness of linear periods}, Compositio Math. 102 (1996), no. 1, 65-123




\bibitem{K11}
P.K. Kewat,
{\it The local exterior square L-function: Holomorphy, non-vanishing and Shalika
functionals.} Journal of Algebra, vol. 347, no. 1, pp. 153-172, 2011.

\bibitem{KR}
P.K.~Kewat, R.~Raghunathan \textit{On the local and global exterior square $L$-functions of $GL_n$}, Math. Res. Lett. 19 (2012), no. 4, 785-804

\bibitem{KottHA}
R.E.~Kottwitz, {\it Harmonic analysis on reductive $p$-adic groups and Lie algebras}, in "Harmonic analysis, the trace formula, and Shimura varieties", 393522, Clay Math. Proc., 4, Amer. Math. Soc., Providence, RI, 2005

\bibitem{LW}
J.-P.~Labesse, J.-L.~Waldspurger, \textit{La formule des traces tordue d'apr\`es le Friday Morning Seminar} CRM Monograph Series, 31. American Mathematical Society, Providence, RI, 2013. xxvi+234 pp

\bibitem{L04}
R. Langlands,
{\it Beyond endoscopy.} in Contributions to Automorphic Forms, Geometry,
and Number Theory, John Hopkins University Press, 2004, 611C698

\bibitem{MW}
C.~M\oe{}glin, J.-L.~Waldspurger, \textit{Mod\`eles de Whittaker d\'eg\'en\'er\'es pour des groupes $p$-adiques}, Math. Z. 196 (1987), no. 3, 427-452

\bibitem{MWGGP}
C.~M\oe{}glin, J.-L.~Waldspurger, {\it La conjecture locale de Gross-Prasad pour les groupes sp\'eciaux orthogonaux: le cas g\'en\'eral}, dans "Sur les conjectures de Gross et Prasad II" Ast\'erisque No. 347 (2012), 167-216

\bibitem{Pras}
D.~Prasad, \textit{A 'relative' local Langlands correspondence}, arXiv preprint 2015, arXiv:1512.04347

\bibitem{SV}
Y.~Sakellaridis, A.~Venkatesh, \textit{Periods and harmonic analysis on spherical varieties}, preprint 2012, arXiv:1203.0039, to appear in Ast\'erisque

\bibitem{Sc}
P.~Scholze, {\it The local Langlands correspondence for $GL_n$ over $p$-adic fields}, Invent. Math. 192 (2013), no. 3, 663-715

\bibitem{Var}
S.~Varma, \textit{On a result of M\oe{}glin and Waldspurger in residual characteristic 2}, Math. Z. 277 (2014), no. 3-4, 1027-1048

\bibitem{WalPlanch}
J.-L.~Waldspurger, \textit{La formule de Plancherel pour les groupes $p$-adiques (d'apr\`es Harish-Chandra)}, J. Inst. Math. Jussieu 2 (2003), no. 2, 235-333

\bibitem{WalGGPI}
J.-L.~Waldspurger, \textit{Une formule int\'egrale reli\'ee \`a la conjecture locale de Gross-Prasad}, Compos. Math. 146 (2010), no. 5, 1180-1290

\bibitem{WalGGPII}
J.-L.~Waldspurger, \textit{Une formule int\'egrale reli\'ee \`a la conjecture locale de Gross-Prasad, 2e partie: extension aux repr\'esentations temp\'er\'ees}, in "Sur les conjectures de Gross et Prasad I" Ast\'erisque No. 346 (2012), 171-312

\bibitem{WalGGPIII}
J.-L.Waldspurger, {\it La conjecture locale de Gross-Prasad pour les repr\'esentations temp\'er\'ees des groupes sp\'eciaux orthogonaux}, dans "Sur les conjectures de Gross et Prasad II" Ast\'erisque No. 347 (2012), 103-165

\bibitem{WalFTLtordue}
J.-L.~Waldspurger, \textit{La formule des traces locales tordue}, preprint 2012

\bibitem{Wan15}
Chen Wan,
{\it A local relative trace formula for the Ginzburg-Rallis model: the geometric side.} 95 pages, accepted by the Memoirs of the AMS, 2015.

\bibitem{Wan16a}
Chen Wan,
{\it Multiplicity One Theorem for the Ginzburg-Rallis Model: the tempered case.} Submitted, 2016. 56 pages



\end{thebibliography}
\end{document}